\documentclass[11pt]{amsart}

\usepackage{amsmath, amssymb,amsthm,color}
\usepackage[all]{xy}
\usepackage{tikz-cd}
\usetikzlibrary{arrows.meta,decorations.markings,arrows}

\newcommand{\nc}{\newcommand}
\newcommand{\cC}{\mathtt C}
\newcommand{\cD}{\mathtt D}

\newcommand{\uz}{{\underline{z}}}

\newcommand{\ul}{{\underline{\lambda}}}
\newcommand{\umu}{{\underline{\mu}}}

\newcommand{\Hom}{\operatorname{Hom}}
\newcommand{\C}{\mathbb{C}}
\newcommand{\R}{\mathbb R}
\newcommand{\PP}{\mathbb P}

\renewcommand{\sl}{\mathfrak{sl}}

\newcommand{\cE}{\mathcal E}
\newcommand{\cX}{\mathcal X}
\newcommand{\wt}{\widetilde}

\newcommand{\Cx}{\C^\times}
\nc{\eps}{\varepsilon}
\nc{\Gr}{\mathbb{G}}

\nc{\A}{{\mathcal A}}
\nc{\ol}{\overline}
\nc\tboxtimes{\wt{\boxtimes}}
\nc{\alp}{\alpha}
\nc{\Wh}{\operatorname{Wh}}

\renewcommand{\l}{\lambda}
\newcommand{\spn}{\operatorname{span}}

\nc{\BA}{{\mathbb{A}}} \nc{\BC}{{\mathbb{C}}}
\nc{\BQ}{{\mathbb{Q}}}
\nc{\BM}{{\mathbb{M}}} \nc{\BN}{{\mathbb{N}}}
\nc{\BP}{{\mathbb{P}}} \nc{\BR}{{\mathbb{R}}}
\nc{\BZ}{{\mathbb{Z}}} \nc{\BS}{{\mathbb{S}}}

\nc{\CA}{{\mathcal{A}}} \nc{\CB}{{\mathcal{B}}}
\nc{\CE}{{\mathcal{E}}} \nc{\CF}{{\mathcal{F}}}
\nc{\CG}{{\mathcal{G}}} \nc{\CH}{{\mathcal{H}}}
\nc{\CI}{{\mathcal{I}}}  \nc{\CJ}{{\mathcal{J}}}
\nc{\CL}{{\mathcal{L}}}
\nc{\CM}{{\mathcal{M}}} \nc{\CN}{{\mathcal{N}}}
\nc{\CO}{{\mathcal{O}}} \nc{\CP}{{\mathcal{P}}}
\nc{\CQ}{{\mathcal{Q}}} \nc{\CR}{{\mathcal{R}}}
\nc{\CS}{{\mathcal{S}}} \nc{\CT}{{\mathcal{T}}}
\nc{\CU}{{\mathcal{U}}} \nc{\CV}{{\mathcal{V}}}
\nc{\CY}{{\mathcal Y}}
\nc{\CW}{{\mathcal{W}}} \nc{\CZ}{{\mathcal{Z}}}

\nc{\ff}{{\mathfrak{f}}} \nc{\fv}{{\mathfrak{v}}}
\nc{\fa}{{\mathfrak{a}}} \nc{\fb}{{\mathfrak{b}}}
\nc{\fd}{{\mathfrak{d}}} \nc{\fe}{{\mathfrak{e}}}
\nc{\fg}{{\mathfrak{g}}} \nc{\fgl}{{\mathfrak{gl}}}
\nc{\fh}{{\mathfrak{h}}} \nc{\fri}{{\mathfrak{i}}}
\nc{\fj}{{\mathfrak{j}}} \nc{\fk}{{\mathfrak{k}}}
\nc{\fm}{{\mathfrak{m}}} \nc{\fn}{{\mathfrak{n}}}
\nc{\ft}{{\mathfrak{t}}} \nc{\fu}{{\mathfrak{u}}}
\nc{\fw}{{\mathfrak{w}}} \nc{\fz}{{\mathfrak{z}}}
\nc{\fl}{{\mathfrak{l}}}
\nc{\fp}{{\mathfrak{p}}} \nc{\frr}{{\mathfrak{r}}}
\nc{\fs}{{\mathfrak{s}}} \nc{\fsl}{{\mathfrak{sl}}}
\nc{\hsl}{{\widehat{\mathfrak{sl}}}}
\nc{\hgl}{{\widehat{\mathfrak{gl}}}}
\nc{\hg}{{\widehat{\mathfrak{g}}}}
\nc{\chg}{{\widehat{\mathfrak{g}}}{}^\vee}
\nc{\hn}{{\widehat{\mathfrak{n}}}}
\nc{\chn}{{\widehat{\mathfrak{n}}}{}^\vee}
\nc{\Xiset}{\Xi\text{-}\mathtt{Set}}
\nc{\Lset}{\Lambda_+\text{-}\mathtt{Set}}
\nc{\set}{\mathtt{Set}}
\nc{\gcrys}{\fg\text{-}\mathtt{Crys}}
\nc{\sqcu}{\bigsqcup\limits}
\newcommand{\sslash}{\mathbin{/\mkern-6mu/}}

\DeclareMathOperator{\can}{can} 
 
 \DeclareMathOperator{\gr}{gr}
\DeclareMathOperator{\ad}{ad} 
 
\DeclareMathOperator{\id}{id}

 \DeclareMathOperator{\End}{End}
 
 \DeclareMathOperator{\Spec}{Spec}
 \DeclareMathOperator{\rk}{rk}

\newtheorem{cor}[subsection]{Corollary}

\newtheorem{lem}[subsection]{Lemma}
\newtheorem{prop}[subsection]{Proposition}

\newtheorem{conj}[subsection]{Conjecture}
\newtheorem{thm}[subsection]{Theorem}
\newtheorem{quest}[subsection]{Question}

\theoremstyle{definition}
\newtheorem{defn}[subsection]{Definition}
\newtheorem{rem}[subsection]{Remark}
\newtheorem{eg}[subsection]{Example}

\theoremstyle{remark}

\newcommand{\thmref}[1]{Theorem~\ref{#1}}
\newcommand{\secref}[1]{Section~\ref{#1}}
\newcommand{\lemref}[1]{Lemma~\ref{#1}}
\newcommand{\propref}[1]{Proposition~\ref{#1}}
\newcommand{\corref}[1]{Corollary~\ref{#1}}
\newcommand{\conjref}[1]{Conjecture~\ref{#1}}

\newcommand{\defnref}[1]{Definition~\ref{#1}}

\title{Crystals and monodromy of Bethe vectors}

\author{Iva Halacheva}
\author{Joel Kamnitzer}
\author{Leonid Rybnikov}
\author{Alex Weekes}

\begin{document}
	\begin{abstract}
Fix a semisimple Lie algebra $ \fg$.  Gaudin algebras are commutative algebras acting on tensor product multiplicity spaces for $ \fg$-representations.  These algebras depend on a parameter which is a point in the Deligne-Mumford moduli space of marked stable genus 0 curves.  When the parameter is real, then the Gaudin algebra acts with simple spectrum on the tensor product multiplicity space and gives us a basis of eigenvectors.  In this paper, we study the monodromy of these eigenvectors as the parameter varies within the real locus; this gives an action of the fundamental group of this moduli space, which is called the cactus group.

We prove a conjecture of Etingof which states that the monodromy of eigenvectors for Gaudin algebras agrees with the action of the cactus group on tensor products of $ \fg$-crystals.   In fact, we prove that the coboundary category of normal $ \fg$-crystals can be reconstructed using the coverings of the moduli spaces.

Our main tool is the construction of a crystal structure on the set of eigenvectors for shift of argument algebras, another family of commutative algebras which act on any irreducible $\fg$-representation.  We also prove that the monodromy of such eigenvectors is given by the internal cactus group action on $ \fg$-crystals.
\end{abstract}
	
	\maketitle

	\section{Introduction}
	\subsection{Crystals and their tensor products}
	Let $ \fg $ be a semisimple Lie algebra.  Crystals are combinatorial objects which model bases of representations of $ \fg $.  They can be pictured as labelled graphs.  Attached to each irreducible representation $ V(\l) $ of $ \fg $, we have a connected crystal $ B(\l) $.  In this paper we will consider ``normal'' crystals; these are crystals which are a disjoint union of the crystals $ B(\l) $.
	
	Given two crystals $ B, C$, we can form their tensor product $ B \otimes C $, a crystal whose underlying vertex set is just the Cartesian product $ B \times C $ (see \defnref{Def:Tensor}).  The tensor product $ B(\l_1) \otimes B(\l_2) $ decomposes as a disjoint union of copies of $ B(\mu) $, with multiplicities equal to the usual tensor product multiplicities in $ V(\l_1) \otimes V(\l_2) $.
	
	The definition of $ B \otimes C $ is not manifestly symmetric.  In \cite{HK}, Henriques and the second author constructed a natural isomorphism $ \sigma_{B,C} : B \otimes C \rightarrow C \otimes B $ (see \secref{se:commutor} for the definition).  We also proved that with this commutor, the category, $ \gcrys$, of normal $ \fg$-crystals forms a coboundary category (see \defnref{Def:Coboundary}).
	
	In a coboundary category, any tensor product $ B_1 \otimes \cdots \otimes B_n $ of objects carries an action of a finitely-presented group called the cactus group $C_n$ (see \secref{se:CactusCn}).   In particular, for any dominant weights $ \l_1, \dots, \l_n, \mu $, we obtain an action\footnote{Actually, $ C_n $ doesn't quite act on this set, as it also permutes the $ \lambda_i $.} of $ C_n $ on the set $ B(\ul)^\mu $ of copies of $ B(\mu) $ in $ B(\ul) := B(\l_1) \otimes \cdots \otimes B(\l_n) $.
	
	The cactus group $C_n $ first appeared in the work of Davis-Januszkiewicz-Scott \cite{DJS} as the equivariant fundamental group of $ \CM_{n+1}(\BR) $, the real locus of the Deligne-Mumford space of stable genus 0 curves with $ n+1 $ marked points.  By the usual correspondence between actions of the fundamental group and covering spaces, we get coverings of the moduli space $  \CM_{n+1}(\BR) $ whose fibres are the sets of $ B(\ul)^\mu $.
	
	The action of the cactus group on a tensor product of crystals is analogous to the action of the braid group on a tensor product of representations of quantum groups.  The famous Drinfeld-Kohno theorem shows that this action of the braid group can be realized as the monodromy of the Knizhnik-Zamalodchikov connection.
	
	This leads us to the following question, which we will answer in this paper.
	\begin{quest} \label{question}
	Can we realize the action of the cactus group on $ B(\ul)^\mu $ or the corresponding coverings of $  \CM_{n+1}(\BR) $ in some natural way, directly from representation theory?
	\end{quest}	
	
	\subsection{Gaudin algebras and the monodromy of their eigenlines}
	The cardinality of the set $ B(\ul)^\mu $ is the same as the dimension of the tensor product multiplicity space
	$$ V(\ul)^\mu := \Hom_\fg(V(\mu), V(\ul)) $$
	where $ V(\ul) := V(\l_1) \otimes \cdots \otimes V(\l_n)$.
	
	The vector spaces $ V(\ul)^\mu $ can be studied with the help of Gaudin algebras.   The Gaudin  algebra $ \A(z_1, \dots, z_n) $ is a maximal commutative subalgebra of $ (U(\fg)^{\otimes n})^{\fg} $, which depends on $ n $ distinct complex numbers (see \secref{se:Gaudin}).  The quadratic part of the Gaudin algebras is spanned by the Gaudin Hamiltonians which are simple expressions involving the Casimir; the full definition of the algebra is more complicated and involves the Feigin-Frenkel centre \cite{FF} of the affine Lie algebra at the critical level (see \secref{se:QuantumGaudin}).  Since the Gaudin algebra commutes with the diagonal copy of $ \fg $, it acts on $ V(\ul)^\mu $.  Recently, the third author \cite{Ryb16} proved that $ \A(z_1, \dots, z_n) $ acts cyclically on $ V(\ul)^\mu $\footnote{In the present paper, we give an independent proof of a stronger fact (see Theorem~\ref{th:CyclicGeneral} and Corollary~\ref{co:SimpleSpecGeneral}) generalizing this to non-homogeneous Gaudin algebras}.  If these $z_1, \dots, z_n $ are distinct real numbers, then the Gaudin algebra is generated by Hermitian operators and hence acts semi-simply.  Thus the cyclicity result implies that in this case, $ \A(z_1, \dots, z_n) $ decomposes $ V(\ul)^\mu $ into a direct sum of eigenlines.  The Gaudin algebras are examples of Bethe algebras and their eigenvectors are generically found by the algebraic Bethe Ansatz method, thus these eigenlines are also called Bethe vectors.

	The tuple $ \uz = (z_1, \dots, z_n) $ can be regarded as a point in $ \CM_{n+1} $ by taking $ \PP^1 $ and marking the points $ z_1, \dots, z_n $ and $ \infty $; in this way such tuples correspond to the non-boundary points of $ \CM_{n+1}$.  In \cite{AFV}, Aguirre-Felder-Veselov proved that the family of subspaces spanned by the Gaudin Hamiltonians extends to a family parametrized by $ \CM_{n+1} $.  In \cite{Ryb13}, the third author extended this result and proved that the family of subalgebras $ \A(\uz) $ extends to a family of subalgebras parametrized by $ \CM_{n+1} $.  This allows us to define a covering $ \CE(\ul)^\mu $ of $ \CM_{n+1}(\BR) $, whose fibre at a point $ \uz \in \CM_{n+1}(\BR) $ is the set of eigenlines for the action of $ \A(\uz) $ on $ V(\ul)^\mu $ (see \secref{se:covering}).  	
	
	The following result was conjectured by Etingof (see \cite{Ryb13}) and will be proven in this paper.
	\begin{thm} \label{co:Etingof}
	The action of $ C_n $ on $ B(\ul)^\mu $ coming from the crystal commutor agrees with the action of $ C_n $ by monodromy on the cover $ \CE(\ul)^\mu $.
	\end{thm}
	This result was first proved by the third author \cite{Ryb13} for $ \fg = \sl_2 $.  The proof uses an older paper of Varchenko \cite{Varchenko}, which first indicated a link between Bethe vectors and crystals.
	
	The result was proved for $ \fg = \sl_m $ by White \cite{W} (see below).
		
	\subsection{Operadic coverings and coboundary categories}
	 Our approach to \thmref{co:Etingof} is to study the entire category $ \gcrys $ at once.  We will exploit a close connection between coloured operadic coverings of the moduli spaces $ \CM_{n+1} $ and coboundary categories.
	
	As motivation, recall that there is a close relationship between genus 0 modular functors and semisimple braided monoidal categories (see Theorem~5.4.1 of \cite{BaK}).  A genus 0 modular functor is a collection of $D$-modules on the moduli spaces $ \CM_{n+1}(\C) $, along with compatibilities between them when we pass to the boundary strata of $ \CM_{n+1} $.
	
	In this paper, we introduce the notion of a $\Xi$-coloured operadic covering of the moduli spaces $ \CM_{n+1}(\BR) $, where $ \Xi$ is a set (see \secref{se:Operadic}).  This is a sequence of covering spaces $    \cX_{n+1} \rightarrow \CM_{n+1}(\BR) \times \Xi^{n+1} $ which forms a coloured operad over the operad $ \CM_{n+1}(\BR)$.  We prove that a $ \Xi $-coloured operadic covering $ \cX $ leads to a coboundary category $ \cC(\cX) $, whose underlying category is the category of $ \Xi$-graded sets (\thmref{th:CoverToCategory}).
	
	Using an operadic description of the Gaudin algebras $ \A(\uz) $ associated to boundary points $ \uz \in \CM_{n+1}(\BR) $, we prove that the coverings $ \CE(\ul)^\mu $ of $ \CM_{n+1}(\BR) $ fit into a $\Lambda_+$-coloured operadic covering (\thmref{th:OperadFromEigenlines}).  Thus we obtain a coboundary category $ \cC(\CE) $ whose tensor product multiplicity sets are encoded by the sets $ \CE(\ul)^\mu $.  Our main theorem (\thmref{th:Main}) is the following.
	
	\begin{thm} \label{th:mainintro}
	There is an equivalence of coboundary categories $ \cC(\CE) \cong \gcrys $.
	\end{thm}
	
	This theorem implies \thmref{co:Etingof}.

	\subsection{Shift of argument algebras}
	In order to connect eigenlines for Gaudin algebras with crystals, we work with inhomogenous Gaudin algebras $ \A_\chi(\uz) \subset U(\fg)^{\otimes n} $, which depend on $n$ distinct complex numbers $ \uz $ and a regular Cartan element $ \chi $ (these algebras were introduced by the third author in \cite{R} and independently by Feigin-Frenkel-Toledano Laredo \cite{FFTL}).  These are maximal commutative subalgebras and we prove that they act cyclically on the tensor product $ V(\ul) $ (\thmref{th:CyclicGeneral}).
	
	Of particular interest is the case $ n = 1$, in which we obtain the shift of argument algebras $ \A_\chi \subset U(\fg) $.  We prove (\thmref{th:dCPFamily}) that the closure of the family of shift of argument algebras $ \A_\chi $ is parametrized by the de Concini-Procesi wonderful compatification $ \CM_\Delta $, which is a compactification of $ \fh^{reg} / \Cx $ obtained by succesive blowups along intersections of root hyperplanes (see \secref{sect-prelim}).  We prove this by combining work of Shuvalov \cite{Sh} on the limits of classical shift of argument algebras and work of Aguirre-Felder-Veselov \cite{AFV2} on the limits of their quadratic subspaces.
	
	In \cite{FFR}, Feigin, Frenkel, and the third author proved that $ \A_\chi$ acts cyclically on $ V(\l) $ for any $ \chi \in \fh^{reg} $ and we extend this result to all $ \chi \in \CM_\Delta $.  In particular, when $ \chi \in \CM_\Delta(\BR) $, we get a decomposition of $ V(\l) $ into eigenlines $ \CE_\chi(\l) $.  This provides us with a covering of the space $ \CM_\Delta(\BR) $ by these eigenlines; the fibres of this covering have the same size as the dimension of $ V(\l) $.
	
	The set of eigenlines $ \CE_\chi(\l) $ is of independent interest; for example, when $ \fg = \sl_m $, and $ \chi $ corresponds to a caterpillar point (see \cite{S}, page 16) of $ \CM_\Delta = \CM_{m+1} $, then $ \CE_\chi(\l) $ is the Gelfand-Tsetlin basis.  In \secref{se:CrystalFibre}, we construct a crystal structure on any such set of eigenlines $ \CE_\chi(\l) $; this definition uses the parallel transport in the covering.  Using restriction to Levi subalgebras, we prove the following (\thmref{th:CEBlambda}).
	
	\begin{thm}
		There is a crystal isomorphism $ \CE_\chi(\l) \cong B(\l) $.
	\end{thm}
	
	\subsection{Completion of the main proof}
	The closure of the space of the subalgebras $ \A_\chi(\uz) $ is unknown in general, but in \secref{subsect-PieceClosure}, we study some subfamilies.  In particular, as $ z\rightarrow \infty $,  the subalgebra $\A_\chi(z,0)$ limits to $ \A_\chi \otimes \A_\chi $, and as $ z \rightarrow 0 $,  it limits to the subalgebra generated by $ \Delta(\A_\chi) $ and $ \A(1,0)$.  When we consider eigenlines for these algebras acting on $V(\l_1) \otimes V(\l_2) $, we can use parallel transport in this family to get an isomorphism of crystals (see \thmref{th:TensorCE})
	$$
	 \CE_{\chi}(\l_1) \times \CE_\chi(\l_2) \rightarrow \sqcu_\mu \CE(\l_1, \l_2)^\mu \times \CE_\chi(\mu)
$$
This proves that the eigenlines $ \CE(\l_1,\l_2)^\mu $  for the Gaudin algebra $ \A(1,0) $ are in bijection with the crystal-theoretic tensor product multiplicity sets.  This allows us to construct the tensor functor which realizes the equivalence in \thmref{th:mainintro}.  The proof that this equivalence is compatible with the coboundary category structures uses further results about the closure of the family of inhomogeneous Gaudin algebras.

	\subsection{Another monodromy result}
	For any dominant weight $ \l $, we have the family $ \CE_\chi(\l) $ of eigenlines of shift of argument algebras $ \A_\chi$, as $ \chi $ varies over $ \CM_\Delta(\BR) $.  Thus, we get an action of $ C_\Delta = \pi_1^W(\CM_\Delta(\BR))$, the cactus group of type $ \Delta $, on any fibre $ \CE_\chi(\l) $.
	
	On the other hand, in \secref{sec:action}, we define an action of $ C_\Delta $ on any normal $ \fg$-crystal, by partial Sch\"utzenberger involutions.  This generalizes a construction of Berenstein-Kirillov \cite{BeK} in the case $ \fg = \sl_m $. It can be thought as the crystal-theoretic analog of Lusztig's quantum Weyl group action on representations of quantum groups.
	
	We prove the following result (\thmref{th:Etingof2}), which is analogous to \thmref{co:Etingof} and was also conjectured by the third author in \cite{Ryb13}.
	\begin{thm} \label{th:Etingof2Intro}
	 The isomorphism of crystals $ \CE_\chi(\l) \cong B(\l) $ intertwines the monodromy action of $ C_\Delta $ with the crystal-theoretic one.
	 \end{thm}

Much as \thmref{co:Etingof} is analogous to the Drinfeld-Kohno theorem, \thmref{th:Etingof2Intro} is analogous to Toledano Laredo's theorem \cite{TL, TL2} describing the monodromy of the Casimir connection.	 	
		
	\subsection{Schubert intersections and Opers}
	As mentioned above, Etingof's conjecture (\thmref{co:Etingof}) was proven for $ \fg = \sl_m $ by White \cite{W}.  His paper uses a connection between the eigenvalues of Gaudin algebras and intersections of Schubert varieties from the work of Mukhin-Tarasov-Varchenko \cite{MTV}.  Working on the Schubert intersection side, Speyer \cite{S} constructed a cover of $ \CM_{n+1}(\BR) $ whose monodromy gives the crystal commutor.  White showed (somewhat indirectly) that the cover from Gaudin eigenlines is isomorphic to Speyer's cover.

	The relationship between Gaudin algebras and intersections of Schubert varieties only holds for $ \fg = \sl_m $.  However, for any $ \fg $, we can give a geometric description of Gaudin eigenvalues using opers for the Langlands dual Lie algebra.  In a followup paper, we will investigate this further and develop the theory of opers on nodal curves.  In particular, we will give a direct proof of the relationship between Speyer's cover and ours.  We also hope that this oper perspective will allow us to develop a higher genus version of the constructions from this paper.
		
	\subsection{Weyl groups and quantum cohomology}
	A further conjecture along the lines of \thmref{co:Etingof} and \thmref{th:Etingof2Intro} was also suggested by Etingof.  In \cite{Losev}, Losev constructed an action of the Cactus group $ C_\Delta $ on the Weyl group $ W $.  The definition of this action involves perverse equivalences coming from wall-crossing functors in category $ \CO$.  (A more elementary definition was given shortly afterwards by Bonnaf\'e \cite{Bon}.)  In the case $ W = S_n$, this action is a special case of the action coming from the crystal commutor;  we can take $\fg = \sl_n$,  $ B(\omega_1)^{\otimes n } $ and look at the 0-weight set $ B(\omega_1)^{\otimes n }_0 $ which is naturally in bijection with $ S_n $ and carries an action of $ C_n $ defined using the crystal commutor.
	
	On the other hand, there is a family of algebras, $ D_\chi $, which act on the group algebra $ \BC W $ (they are generated by degenerations of Dunkl operators).  These algebras depend on a parameter $ \chi \in \fh^{reg}$, so it is natural to expect that they compactify to a family parametrized by $ \CM_\Delta $.  Then, assuming a simple spectrum result, we would get a covering of $ \CM_\Delta(\BR) $ whose fibres are the eigenlines $ \CE_{D_\chi}(\BC W)$.  Etingof has conjectured the following.
	
	\begin{conj} \label{co:Etingof3}
	The monodromy of this covering agrees with the action of $C_\Delta $ on $ W $.
	\end{conj}
	This conjecture is closely related to conjectures of Bonnaf\'e-Rouquier \cite{BR} on a Galois group action on the Weyl group.
	
	We expect that this circle of ideas admits a further generalization.  Suppose that we are given a conical symplectic singularity $ X \rightarrow X_0 $.  Then we obtain a family of symplectic resolutions $ X_\theta $ and corresponding quantizations $ A_\theta$, indexed by generic elements $ \theta \in H^2(X, \BR) $.  Losev \cite{Losev2} has studied wall-crossing functors which correspond to moving $ \theta$ across the walls of a hyperplane arrangement in $ H^2(X, \BR) $.  He proved that these wall-crossing functors give perverse equivalences.  We expect that these perverse equivalences lead to the action of a finitely-generated group analogous to the cactus group on the set of simple objects in a suitable category of $ A_\theta $-modules.
	
	 On the other hand, we can consider the equivariant quantum cohomology algebra $ QH_{\Cx}^\bullet(X) $.  This algebra depends on a parameter $ q \in H^2(X, \BC)$ and acts on $ H_{\Cx}^\bullet(X)$.  Assuming a simple spectrum result for real $ q $, we will then get a covering of the compactified real parameter space, leading to an action of the fundamental group of this space which we expect should be the same as the ``cactus'' group coming from the above perverse equivalences (or perhaps an affine version of this group).  Moreover, we expect a bijection between the eigenlines for the quantum cohomology algebra and the simple objects above compatible with this cactus group.  \thmref{co:Etingof} and \thmref{th:Etingof2Intro} can be seen as special cases of this setup, applied to affine Grassmannian slices and finite-type quiver varieties respectively, whereas \conjref{co:Etingof3} is an instance of this setup for the Springer resolution.
	 	
		\subsection{Notation}
		A complete list of notation may be found in Appendix \ref{append}.
		
	\subsection{Acknowledgements}
	We would like to thank P. Etingof for his conjecture and for many helpful discussions.  We also thank R. Bezrukavnikov, A. Braverman, I. Losev, M. McBreen, E. Mukhin, D. Speyer, V. Toledano Laredo and N. White for helpful discussions.  This project was begun when the authors were visiting {\'E}PFL; we thank {\'E}PFL and Swiss-MAP for making these visits possible. The second author was supported by an NSERC discovery grant, a Sloan fellowship, and a Simons fellowship.  The first and fourth authors were supported in part by graduate scholarships from NSERC and the Government of Ontario.  This research was supported in part by Perimeter Institute for Theoretical Physics. Research at Perimeter Institute is supported by the Government of Canada through the Department of Innovation, Science and Economic Development and by the Province of Ontario through the Ministry of Research and Innovation. The work of the third author has been funded by the  Russian Academic Excellence Project '5-100'. Theorem~\ref{th:CyclicGeneral} has been obtained under support of the Russian Science Foundation grant 16-11-10160. The third author has also been supported in part by the Simons Foundation.
	
	\section{Preliminaries on de Concini-Procesi spaces}\label{sect-prelim}
	\subsection{The de Concini-Procesi wonderful compactification}\label{subsect-dCP}
	Throughout the paper, we fix a rank $ r $ semisimple Lie algebra $ \fg $ with Cartan subalgebra $ \fh $, set of roots $\Delta \subset \fh^*$, positive roots $ \Delta_+ $, and simple roots $ \{ \alpha_i : i \in I \} $.
	
	We let $ \CG $ denote the \emph{minimal building set} associated to the set of roots.  To define $ \CG $, let $ \CG' $ denote the set of all non-zero subspaces of $ \fh^* $ which are spanned by a subset of $ \Delta $.   Let $ V \in \CG' $.  We say that $ V  = V_1 \oplus \cdots \oplus V_k $ is a \emph{decomposition} of $ V $ if  $ V_1, \dots, V_k \in \CG'$, and if whenever $ \alpha  \in \Delta $ and $ \alpha \in V $, then $ \alpha \in V_i $ for some $ i$.  From Section 2.1 of \cite{DCP}, every element of $ \CG' $ admits a unique decomposition.
	
	Then we define $ \CG $ be the set of indecomposable elements of $ \CG' $.
	
	There is an action of the Weyl group $ W $ on $ \fh $.  This action preserves $ \Delta $.  Thus, we get actions of $ W $ on $ \CG$ and $\CG' $.
	
	If $ J \subseteq I $ is a non-empty, connected subset of the Dynkin diagram $ I $ of $ \fg $, we can form $ V_J = \spn(\alpha_j : j \in J) $.  Then $ V_J \in \CG $.  In fact, every $ V \in \CG $ is of the form $ w(V_J) $ for some $ w \in W $ and $ J $ as above.  The subspaces $ V_J $ are compatible with the dominant Weyl chamber $$ \fh_+ := \{ x \in \fh(\R) : \langle \alpha,   x \rangle \ge 0 \text{ for all } \alpha \in \Delta_+ \}, $$ in the sense that $ V_J^\perp \cap \fh_+ $ is a facet of $ \fh_+ $, whose open subset we denote by $ \fh_+^J $, so that
	$$ \fh_+^J = \{ \chi \in \fh(\R) : \alpha_j(\chi) = 0 \text{ for $ j \in J $ and } \alpha_i(\chi) > 0 \text{ for $ i \in I \smallsetminus J $} \} $$

	Let $ \fh^{reg} = \{ \chi \in \fh : \alpha(\chi) \ne 0, \text{ for all } \alpha \in \Delta \} $.  For any $ V \in \CG$, we have a map $ \fh^{reg} \rightarrow \PP(\fh/V^\perp) $.
	
	The de Concini-Procesi space $ \CM_\Delta \subset \prod_{V \in \CG} \PP(\fh/V^\perp) $ is defined to be the closure of the image of the map $ \fh^{reg} \rightarrow  \prod_{V \in \CG} \PP(\fh/V^\perp)$.  So a point of $ \CM_\Delta $ consists of a collection $ (L_V)_{V \in \CG} $ where $ L_V \in \fh/V^\perp $.  We will think of $ L_V $ as a subspace of $ \fh $ containing $ V^\perp $ as a hyperplane.  Note that if $ \chi \in \fh^{reg}$, then $ \chi \notin V^\perp $ for all $ V \in \CG $ and the image of $ \chi $ in $ \CM_\Delta $ is the collection $  L_V = V^\perp + \C\chi $.
	
	\begin{rem}
		Suppose that $ \Delta $ is a reducible root system.  Then $ \fh^* $ will not lie in $ \CG $ and thus $ \CM_\Delta $ will have dimension less than $ r-1 $.  Moreover, if $ \Delta = \Delta_1 \sqcup \Delta_2$, then $ \CG = \CG_1 \sqcup \CG_2 $ and $ \CM_\Delta \cong \CM_{\Delta_1} \times \CM_{\Delta_2} $.
	\end{rem}
	
	Assume that $ \Delta $ is irreducible.  Then $ \fh^* \in \CG $ and so there is a map $ \pi : \CM_\Delta \rightarrow \PP(\fh) $ which is an isomorphism over the locus $ \fh^{reg} / \Cx $.  We will now investigate the other fibres of $ \pi $.
	
	Let $ \chi_0 \in \fh$, $\chi_0 \ne 0 $.  Let $ \Delta_1 = \{ \alpha \in \Delta : \alpha(\chi_0) = 0 \}$.   The roots $ \Delta_1 $ are related to the centralizer of $ \chi_0 $ as follows.  Consider the derived subalgebra of the centralizer of $ \chi_0$, $ \fg_1 = \fz_{\fg}(\chi_0)' $.  Let $ \fh_1 = \fg \cap \fh $ be the Cartan subalgebra of $ \fg_1 $.  Then the image of $ \Delta_1 \subset \fg^* $ maps bijectively onto the root system of $ \fg_1 $ under the projection $ \fh^* \rightarrow \fh_1^*$.
	
	\begin{lem} \label{le:FibresCM}
		There is a natural isomorphism $ \pi^{-1}([\chi_0]) \cong \CM_{\Delta_1}$.
	\end{lem}

\begin{eg}
	Consider $ \fg = \sl_5 $ and take $ \chi_0 = (2,2,2,-3,-3) $.  Then $$
	\Delta_1 = \{ (1,-1,0,0,0), (1,0,-1,0,0), (0,1,-1,0,0), (0,0,0,1,-1)\}, \quad \fg_1 = \sl_3 \oplus \sl_2 $$
	In this case, the Lemma shows us that $ \pi^{-1}([\chi_0]) = \CM_{\Delta_1} = \mathbb P^1 $.
\end{eg}

	\begin{proof}
		Let $ U $ be the span of roots in $ \Delta_1$.  We can identify  $ U $ with $ \fh_1^*$ using the projection $ \fh^* \rightarrow \fh_1^*$.

		Let $ (L_V)_{V \in \CG} \in \pi^{-1}([\chi_0])$.  Note that for all $ V $, we know that $ \chi_0 \in L_V $ since this property holds over $ \fh^{reg} $.
		
		Suppose that $ V $ is not contained in $ U $.  Then $ \chi_0 \notin V^\perp $ and so $ L_V = V^\perp + \C \chi_0 $.  On the other hand, if $ V \subseteq U $ then we have some freedom in the choice of $ L_V $.
		
		We define a map
		\begin{align*}
			\pi^{-1}([\chi_0]) &\rightarrow \CM_{\Delta_1} \\
			(L_V)_{V \in \CG} &\mapsto  (L_V)_{V \subseteq U}
		\end{align*}
		where we identify $ U = \fh_1^*$.
		
		It is easy to see that this map is an isomorphism.
	\end{proof}

	\subsection{Coordinates on the wonderful compactification} \label{Nested}
	For simplicity, we continue to assume that $ \Delta $ is irreducible, though the results explained here apply with small modifications in the general case.
	
	A subset $ \CS \subset \CG $ is called \emph{nested} if, whenever $ P_1, \dots, P_k \in \CS $ are pairwise incomparable (with respect to inclusion), then $ P_1 \oplus \cdots \oplus P_k $ is a decomposition.   A subset is called \emph{maximal nested} if it is not contained in any other nested subset.  It is known that every maximal nested set has size $ r $ and contains $ \fh^* $.
	
	A simple way to make a maximal nested subset $ \CS $ is to start with a maximal nested collection $ \CJ $ of connected non-empty subsets of $ I $.  This means that for each pairwise incomparable $ J_1, \dots, J_k \in \CJ $, we have that $J_1 \cup \cdots \cup J_k $ is a disjoint union.  If we have a maximal collection like this, then $ \CS = \{ V_J : J \in \CJ \} $ will be a maximal nested subset of $ \CG $.  These maximal nested subsets are said to be compatible with the dominant Weyl chamber.
	
	A maximal nested subset $ \CS $ defines a point $ (L_V)_{V \in \CG} \in \CM_\Delta$, which is defined by the condition that $ L_V \subset P^\perp $ whenever $ V \in \CG, P \in \CS, $ and $ P \subsetneq V $.  These are precisely the maximally degenerate points of the space $ \CM_\Delta $.
	
	A basis $ b \subset \Delta $ of $\fh^* $ is called \emph{adapted} to a maximal nested subset $ \CS \subset \CG $ if, for each $ P \in \CS$, $ b\cap P $ forms a basis for $P$.  The following result is Lemma 1.3 of \cite{DCP} for the case of maximal nested subsets.
	
	\begin{lem} \label{NestedToBasis}
		Let $ b $ be an adapted basis to a maximal nested subset $\CS $.  There exists a bijection $ \CS \rightarrow b $, written $ P \mapsto \alpha_P $ such that $$ P = \spn(\alpha_Q : Q \subseteq P) $$
	\end{lem}
	
	Given a maximal nested set $ \CS $ and an adapted basis $ b $, we define $ U_{\CS}^b \subset \CM_\Delta$ by
	$$
	U_{\CS}^b = \{ (L_V)_{V \in \CG} : \alpha_P(L_P) \ne 0 \text{ for all } P \in \CS \}
	$$
	This open subset $ U_{\CS}^b $ can be identified with an affine space as follows.  Let $ \C^{r-1} = \C^{\CS \setminus \fh^*} $ be an affine space indexed by the non-maximal elements of $ \CS $.  For each $ P \in \CS \setminus \fh^*$, let $ c(P) $ be the minimal element of $ \CS $ properly containing $ P $.
	
	The following result is Proposition 1.5 of \cite{DCP}.
	\begin{lem} \label{chart}
		The map $ U_{\CS}^b \rightarrow \C^{\CS \setminus \fh^*}$ given by
		$$ (L_V)_{V \in \CG} \mapsto \left(\frac{\alpha_P(L_{c(P)})}{\alpha_{c(P)}(L_{c(P)})}\right)_{P \in \CS \setminus \fh^*}
		$$
		is an isomorphism.
	\end{lem}
	
	Equivalently, the map $\C^{\CS \setminus \fh^*} \rightarrow \CM_\Delta $ is characterized by saying that the composition $\C^{\CS \setminus \fh^*} \rightarrow \CM_\Delta \rightarrow \PP(\fh) $ is given by
	$$
	(u_P)_{P \in S \setminus \fh^*} \mapsto \left[ \sum_{P \in S} (\prod_{P \subseteq Q} u_Q) \alpha_P^* \right]
	$$
	where $\{\alpha_P^*\}_{P \in \CS} $ denotes the basis of $ \fh $ dual to $ \{\alpha_P\}_{P \in \CS} $.
	
	This is deduced from \lemref{chart} by noticing that $ \frac{\alpha_P(L_{c(P)})}{\alpha_{c(P)}(L_{c(P)})} = \frac{\alpha_P(L_{\fh^*})}{\alpha_{c(P)}(L_{\fh^*})}
	$ for a generic point of $ \CM_\Delta$.
	
	\subsection{Cactus group}
	There is a Dynkin diagram automorphism $\theta: I\longrightarrow I$ defined by $\alpha_{\theta(i)} = -w_0(\alpha_i)$, where $w_0$ is the longest element in $W$.  For any subset $ J \subset I $, we have a parabolic subgroup $ W_J \subset W $ with a longest element $ w_0^J $, and a bijection $ \theta_J : J \rightarrow J $.  We also have a subrootsystem $ \Delta_J $ consisting of roots which are sums of the simple roots $ \alpha_i$ for $ i \in J $.
	
	\begin{defn} \label{def:cactus}
		We define the \emph{cactus group} $C_\Delta$, of type $\Delta$, to be the group with generators $s_J$, where $J$ is a connected subdiagram of $I$, and relations
		\begin{enumerate}
			\item $s^2_J=1 \; \; \forall \, J \subseteq I$
			\item $s_K s_J = s_{\theta_K(J)}s_K \; \; \forall \, J \subset K \subseteq I$
			\item $s_Ks_{J}=s_{J}s_K \; \; \forall \, J,K \subset I$ such that $ J \cup K $ is disconnected.
		\end{enumerate}
	\end{defn}
	
	We have a group homomorphism $ C_\Delta \rightarrow W $ taking $ s_J $ to $ w_0^J $.  	
	
	\subsection{Equivariant fundamental group}
	Let $ G $ be a finite group acting on a path-connected, locally simply-connected space $ X $.  Let $ x \in X $ be a basepoint.
	
	\begin{defn}
	The \emph{equivariant fundamental group} $ \pi_1^G(X, x) $ is defined as follows.
	$$
	\pi_1^G(X, x) = \{ (g,p) : g \in G, \text{ $ p $ is a homotopy class of paths from $ x $ to $ gx $} \}
	$$
	The multiplication in $ \pi_1^G(X, x) $ is defined as follows.  We define $$ (g_1, p_1) \cdot (g_2, p_2) =  (g_1 g_2, g_1(p_2) * p_1) $$ where $ * $ denotes concatenation of paths.
	\end{defn}
	
	A \emph{$G$-equivariant cover} of $ X $ is a covering space $ Y \rightarrow X $ along with an action of $ G $ on $ Y $ compatible with the action of $ G $ on $ X$.
	
	\begin{lem} \label{le:CoverEquiv}
		We have an equivalence of categories between the category of $ G$-equivariant covers $ Y \rightarrow X $ and the category of $ \pi_1^G(X, x)$-sets.
	\end{lem}
	\begin{proof}
	We didn't find this result in the literature, so we include this proof for completeness.
	
	We have an isomorphism $ \pi_1^G(X,x) \cong \pi_1(X \times E_G / G, x) $ where $ E_G $ denotes a contractible space with a free $ G $ action.  So we obtain equivalences
	\begin{align*}
	\{ \text{ $ \pi_1^G(X,x)$-sets } \} & \cong \{ \text{ covers of $ X \times E_G / G $ } \} \\ & \cong \{ \text{$G$-equivariant covers of $ X \times E_G $} \} \\ &\cong \{ \text{$G$-equivariant covers of $ X $} \}
	\end{align*}
	\end{proof}

	\subsection{The real locus} \label{se:RealLocus}
	For this section, we assume that $\Delta $ is irreducible.  If $ \Delta = \Delta_1 \sqcup \Delta_2$, then $ \CM_\Delta = \CM_{\Delta_1} \times \CM_{\Delta_2}$, so the discussion generalizes immediately to the case of a reducible root system
	
	Since the root system $ \Delta $ is defined over $ \BR $, the variety $ \CM_\Delta $ is defined over $ \BR $ and so it make sense to consider the real points $ \CM_\Delta(\BR)$.
	
	Let, as in Section \ref{subsect-dCP},  $$ \fh_+ = \{ \chi \in \fh(\R) : \alpha_i(\chi) \ge 0 \text{ for all } i \in I \} \subset \fh(\R) $$ be the closed dominant Weyl chamber.  This Weyl chamber has faces
	$$\fh_+^J = \{ \chi \in \fh(\R) : \alpha_j(\chi) = 0 \text{ for $ j \in J $ and } \alpha_i(\chi) > 0 \text{ for $ i \in I \smallsetminus J $} \} $$
	for each subset $ J \subset I $.  In particular, we have the open face $ \fh_+^\emptyset  = \fh_+ \cap \fh^{reg} $.
	
	The real locus $ \CM_\Delta(\BR) $ contains a non-negative subset $ \CM_\Delta(\BR)_+ $ defined as the closure of the image of $ \fh_+^\emptyset $.  The set $ \CM_\Delta(\BR)_+ $ is contractible; in fact, it is homeomorphic to a convex polytope (see Theorem 3.2 of \cite{DCP2}).  For each $ J \subsetneq I $, we define $ \CM_\Delta(\BR)_+^J = \CM_\Delta(\BR)_+ \cap \pi^{-1}(\fh_+^J) $.

	Since $ \CM_\Delta(\BR)_+ $ is contractible, for any two points $ y, z \in \CM_\Delta(\BR)_+$, there exists a unique homotopy class of paths between them which stays inside $ \CM_\Delta(\BR)_+$.  We denote this homotopy class of paths by $ p_{y,z} $.
	
	We consider the equivariant fundamental group of $ \CM_\Delta(\BR) $ with respect to the action of $ W $.  Let $ \chi \in \fh_+^\emptyset $ be a fixed base point, regarded as a point in $ \CM_\Delta(\BR) $.  For convenience, we will require $ w_0(\chi) = -\chi $ in $ \fh $, so that $ w_0(\chi) = \chi $ in $ \CM_\Delta(\BR) $.
	
	For each $  J \subsetneq I $, we define a homotopy class of paths, $p_J, $ from $ \chi $ to $ w_0^J(\chi) $.  Namely, we define $ p_J $ to be the concatenation of two halves.  The first half is $ p_{\chi, \chi^J} $, where $ \chi^J $ is a  $w_0^J$-invariant point in $ \CM_\Delta(\BR)_+^J$.  The second half is obtained by applying the element $ w_0^J $ to the first half.  When $J = I $, we take $ p_I $ to be the constant path.
	
	The following result is due to Davis-Januszkiewicz-Scott \cite{DJS}.
	\begin{thm} \label{cor-DJS}
		The map $ s_J \mapsto (w_0^J, p_J) $ gives an isomorphism between $C_\Delta$ and the equivariant fundamental group $\pi_1^W(\CM_\Delta(\BR),\chi)$.
	\end{thm}

	\section{Deligne-Mumford spaces of genus 0 real curves}
	\subsection{The spaces}
	For $ n \ge 2$, let $ \CM_{n+1} = \overline{M}_{0,n+1}$ denote the Deligne-Mumford space of stable genus 0 curves with $ n+1$ marked points.  A point in this space is a projective genus 0 curve, possibly with nodes, and $n+1 $ marked points, such that each component has at least 3 distinguished (marked or nodal) points; we consider such marked curves up to automorphisms.
	
	Note that $ \CM_{n+1} $ contains a dense open subset corresponding to curves with one component.  This open subset is isomorphic to $ ((\PP^1)^{n+1} \setminus \Delta)/ PGL_2$ (here $PGL_2$ enters since it is the automorphism group of $ \PP^1 $).  Note that we have
	$$  ((\PP^1)^{n+1} \setminus \Delta)/ PGL_2 \cong (\C^n \setminus \Delta)/B \cong \fh^{reg}(\sl_n) / \Cx$$
	where $ B \subset PGL_2 $ is the Borel subgroup, $ B \cong \Cx \ltimes \C $, acting on $ \C^n $ by scaling and translation, $ \fh(\sl_n) $ denotes the Cartan subalgebra of $ \sl_n $, and $ \Delta $ denotes the thick diagonal.  From \cite[Section 4.3]{DCP} (which references an earlier work of Keel \cite{Keel}), we see that this can be extended to an identification $ \CM_{n+1} \cong \CM_{\Delta(\sl_n)} $.
	
	Here we fix the following conventions regarding the root system of $ \sl_n $,
	\begin{gather*} \fh = \C^n / \C(1, \dots, 1), \  \fh^* = \{(a_1, \dots, a_n) \in \C^n: a_1 + \dots + a_n = 0 \}, \\ \Delta(\sl_n) = \{ \eps_i - \eps_j : 1\le i \ne j \le n \}, \ I = \{1, \dots, n-1\}, \ \alpha_i = \eps_i - \eps_{i+1}.
	\end{gather*}
	
	\subsection{Operad structure}
	The varieties $ \CM_{n+1} $ carry an operad structure defined by attaching curves at marked points (see for example \cite[Section 1.4]{GK}).   We begin by reviewing the definition of an operad and a coloured operad, following \cite{Fresse}.
	\begin{defn} \label{Def:Operad}
	An \emph{operad} (in the category of varieties) is a sequence of varieties $ X_n $ (for $ n = 1, 2, \dots $) along with the following data:
	\begin{itemize}
	\item An action of $ S_n $ on $ X_n $.
	\item For all $ n \ge 1 $ and $ k_1, \dots, k_n \ge 1 $, a morphism
	$$\gamma : X_n \times X_{k_1} \cdots \times X_{k_n} \rightarrow X_{k_1 + \cdots + k_n}$$
	\end{itemize}
	satisfying the following axioms
	\begin{itemize}
		\item[(unit)] The variety $ X_1 $ is a point and if we take $n = 1 $ or all $ k_i = 1 $, then the map $ \gamma $ becomes the identity map
		\begin{gather*}
			\gamma : X_1 \times X_k = \{pt\} \times X_{k} \rightarrow X_{k} \\
			\gamma : X_n \times X_1^n = X_{n} \times \{pt\}^n \rightarrow X_{n}
		\end{gather*}
		\item[(associativity)] Given $ n, k_1, \dots, k_n, r^{(1)}_1, \dots r^{(1)}_{k_1}, \dots, r^{(n)}_1, \dots, r^{(n)}_{k_n} $, the diagram
		\begin{equation} \label{eq:OperadAssociativity}
		\begin{tikzcd}[cramped]
				X_{n} \times X_{k_1} \times \cdots \times X_{k_n} \times X_{r^{(1)}_1} \times \cdots \times X_{r^{(n)}_{k_n}}  \ar[d,"{\gamma, id}"] \ar[r,"{id,\gamma}"] & X_{n} \times X_{r^{(1)}_1 + \dots r^{(1)}_{k_1}}  \times \cdots \times X_{r^{(n)}_1 + \dots r^{(n)}_{k_n}} \ar[d,"\gamma"] \\
				X_{k_1 + \dots + k_n } \times X_{r^{(1)}_1} \times \cdots \times X_{r^{(n)}_{k_n}} \ar[r,"\gamma"] & X_{r^{(1)}_1 + \cdots + r^{(n)}_{k_n}}
		\end{tikzcd}
		\end{equation}
		commutes.
		
		\item[(equivariance)]
		Given $ n, k_1, \dots, k_n$,  $w \in S_n $ and $ w_i \in S_{k_i} $, the diagram
		\begin{equation} \label{eq:OperadEquivariance}
		\begin{tikzcd}
				X_{n} \times X_{k_1} \times \cdots \times X_{k_n} \ar[d,"\gamma"] \ar[r,"{w, w_1, \dots, w_n}"] & X_n \times X_{k_{w(1)}} \times \cdots \times X_{k_{w(n)}} \ar[d,"\gamma"] \\
				X_{k_1 + \dots + k_n} \ar[r,"{\gamma(w;w_1, \dots, w_n)}"] & X_{k_1 + \dots + k_n}
		\end{tikzcd}
		\end{equation}
		commutes, where $ \gamma(w;w_1, \dots, w_n) $ is the operad structure on symmetric groups (see \cite{Fresse} for the details).		
	\end{itemize}
	
\end{defn}
Later, we will also need the notion of a coloured operad, which is sometimes also called a multicategory.
	\begin{defn} \label{Def:ColouredOperad}
	Let $ \Xi$ be a set. A $\Xi$-\emph{coloured operad} is a sequence of varieties $ X_n $ (for $ n = 1, 2, \dots $) along with maps $ X_n \rightarrow \Xi^{n+1} $ for each $ n $, along with the following data:
	\begin{itemize}
	\item An action of $ S_n $ on $ X_n $ compatible with the action of $ S_n $ on $ \Xi^{n+1} $ permuting the first $n $ factors.
	\item For each $ \ul = (\lambda_1, \dots, \lambda_n) \in \Xi^n $ and $ \mu \in \Xi $, let $ X(\ul)^\mu $ be the fibre over $ (\lambda_1, \dots, \lambda_n, \mu) \in \Xi^{n+1} $.
	
	Then for each $ \umu = (\mu_1, \dots, \mu_n) \in \Xi^n $ and each $  \ul^{(1)} \in \Xi^{k_1}, \dots, \ul^{(n)} \in \Xi^{k_n} $, and $ \nu \in \Xi$, we have a map
		$$
		\Gamma : X(\umu)^\nu \times X(\ul^{(1)})^{\mu_1} \times \cdots \times X(\ul^{(n)})^{\mu_n} \rightarrow X(\ul^{(1)} \sqcup \cdots \sqcup \ul^{(n)})^\nu
		$$	
	\end{itemize}
	(Here $\sqcup $ denotes concatenation of lists.)
	
	We require this data to satisfy unit, associativity, and equivariance axioms similar to that of an operad.  These axioms are essentially unchanged, with the exception of the unit axiom, where we require $ X(\lambda)^\mu $ be a point if $ \lambda = \mu $ and the empty set otherwise; thus $ X_1 = \Xi $.
	\end{defn}
	
	In order to define the operad structure on $ \CM_{n+1} $, we begin by defining $ \CM_2 = \{pt\} $.
	
	Next, we have the action of $ S_n $ on $ \CM_{n+1} $ by permuting the marked points $ 1, \dots, n $.
	
	Finally, we define, for all $ n \ge 1 $ and $ k_1, \dots, k_n \ge 1 $, the morphism
	$$\gamma : \CM_{n+1} \times \CM_{k_1+1} \cdots \times \CM_{k_n +1} \rightarrow \CM_{k_1 + \cdots + k_n + 1}.$$
	 Given $ (C, C_1, \dots, C_n) \in  \CM_{n+1} \times \CM_{k_1+1} \cdots \times \CM_{k_n +1}$, we define $ \gamma(C, C_1, \dots, C_n) $ to be the result of attaching the curves $ C_1, \dots, C_n $ to the curve $ C $, by identifying the marked point $ j $ on $ C $ with the marked point $ k_j + 1 $ on the curve $ C_j $.

	 \begin{thm}
	 This defines an operad structure on $ \CM_{n+1}$.
	 \end{thm}
	
	\begin{rem}
	For $k_2=\ldots=k_n=1$ and $k_1=k$ the map $\gamma$ is just the \emph{clutching map} $\gamma_{0,0,n,k}:\CM_{n+1}\times\CM_{k+1}\to\CM_{n+k}$ from \cite{Knudsen}. In general, the above map $\gamma$ is the composition of the following clutching maps:
	\begin{multline*}
	    \CM_{n+1} \times \CM_{k_1+1} \cdots \times \CM_{k_n +1} \xrightarrow{\gamma_{0,0,n,k_1}\times\id\times\ldots\times\id} \CM_{n+k_1} \times \CM_{k_2+1} \cdots \times \CM_{k_n +1} \to \\ \xrightarrow{\gamma_{0,0,n+k_1-1,k_2}\times\id\times\ldots\times\id} \CM_{n+k_1+k_2-1} \times \CM_{k_2+1} \cdots \times \CM_{k_n +1} \to\ldots\to \\ 
	    \to \CM_{k_1+k_2+\ldots+k_{n-1}+2} \times \CM_{k_n +1} \xrightarrow{\gamma_{0,0,k_1+k_2+\ldots+k_{n-1}+1,k_n}} \CM_{k_1 + \cdots + k_n + 1}
	\end{multline*} 
	\end{rem}
	
	\subsection{Binary rooted trees}
	\begin{defn}
	A \emph{labelled binary rooted tree} with $ n $ leaves is a binary tree, along with a choice of a root vertex, and a labelling of the leaves by the set $ \{1, \dots, n\} $.  In correspondence with biological principles, we will picture the root of the tree at the bottom.
	
		A  \emph{planar labelled binary rooted tree} is a labelled binary rooted tree along with a planar embedding.  The information of this planar embedding is equivalent to the information of a left branch and a right branch for every internal vertex of the tree.
	\end{defn}
	Of course, every planar labelled binary rooted tree gives rise to a labelled binary rooted tree by forgetting the planar embedding.  We will typically use the same letter $ T $ for both a planar tree and one without a chosen planar embedding. 	
	
	The following result is well-known and is illustrated in Figure \ref{fig:1}.
	\begin{lem}
	There are bijections
	\begin{align*}
	\begin{Bmatrix}\text{planar labelled binary rooted trees} \\ \text{with $n $ leaves} \end{Bmatrix} &\leftrightarrow \{ \text{ordered bracketings of $\{1, \dots, n\} $} \} \\
	\begin{Bmatrix} \text{labelled binary rooted trees} \\ \text{with $n $ leaves} \end{Bmatrix} &\leftrightarrow \begin{Bmatrix} \text{ordered bracketings of } \{1, \dots, n\} \\ \text{modulo equivalence} \end{Bmatrix}
	\end{align*}
	where two ordered bracketings of $ \{1, \dots, n\} $ are considered equivalent if we can get from one to another by performing a  flip inside a bracket.
	\end{lem}
	
	\begin{figure}
		\begin{tikzpicture}[scale=0.7]
		\draw
		(0,0)--(0,1)--(2.5,4) node[above]{5}
		(0,1)--(-2.5,4) node[above]{3}
		(1.25,2.5)--(0,4) node[above]{4}
		(1.75,3.1)--(1,4) node[above]{2}
		(1.375,3.55)--(1.75,4) node[above]{7}
		(-2.125,3.55)--(-1.75,4) node[above]{1}
		(-1.75,3.1)--(-1,4) node[above]{6}
		(3.25,2.5)node[right]{$\longleftrightarrow$}
		(5.375,2.5)node[right]{(\;(\;3\;1\;)\;6\;)\;(\;4\;(\;(\;2\;7\;)\;5\;)\;)}
		;
		\end{tikzpicture}
		\caption{A planar labelled binary rooted tree with $7$ leaves and the corresponding ordered bracketing.}  \label{fig:1}
	\end{figure}
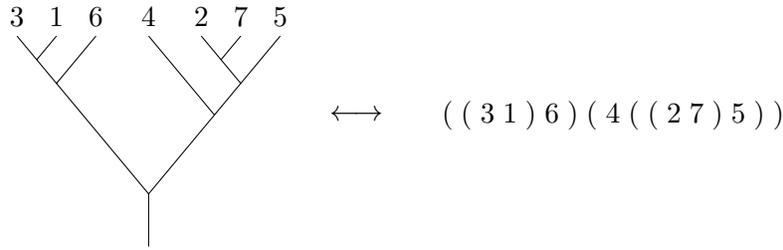

	A labelled binary rooted tree $T$ with $n $ leaves gives rise to a maximal nested subset $ \CS(T) \subset \CG $, where $\CG = \CG(\sl_n) $.  This is done as follows.  For each internal vertex $ v $ of the tree $T$, we consider the subspace
	$$P_v = \spn\bigl(\eps_i - \eps_j : i, j \text{ are labels of leaves above $ v $}\bigr) $$
	and define
	$$ \CS(T) = \{P_v : v \text{ is an internal vertex of $ T $} \}  $$
	
	\begin{lem}
	For each $ T $, $ \CS(T) $ is a maximal nested subset of $ \CG $ and this defines a bijection between  labelled binary rooted trees $T$ with $n $ leaves and maximal nested subsets of $ \CG $.
	\end{lem}
	
	The partial order on $ \CS(T) $, by containment, corresponds to the partial order on the vertices of $ T $, given by proximity to the root.  In particular, for any non-root internal vertex $ v $, $ c(P_v) = P_{c(v)} $, where $ c(v) $ denotes the first vertex along the unique path connecting $ v $ with the root.
	
	Given a planar labelled binary rooted tree $ T$, we define a basis $ b(T) $ adapted to $ \CS(T)$ as follows.  For an internal vertex $ v$ of the tree, we let $l(v) $ denote the maximal label in the left branch at $ v $ and let $ r(v) $ denote the minimal label in the right branch at $ v $.  Define $ b(T) $ by
	$$
	b(T) = \{ \eps_{l(v)} - \eps_{r(v)} : v \text{ is an internal vertex of $ T $} \}  $$
	
	\begin{lem}
	For each planar labelled binary rooted tree $ T $,  $ b(T) $ is a basis adapted to $ \CS(T) $ and the bijection from \lemref{NestedToBasis} is given by $ P_v \mapsto  \eps_{l(v)} - \eps_{r(v)} $.
	
	Moreover, the map $ T \mapsto (\CS(T), b(T)) $ defines a bijection between planar labelled binary rooted trees with $ n $ leaves and pairs consisting a nested set and an adapted basis for $ \CG $.
	\end{lem}
	
	As in \secref{Nested}, the pair $(\CS(T), b(T)) $ gives rise to a subset $ U_T \subset \CM_{n+1} $ and a chart $ \C^{n-2} \rightarrow U_T $, where $\C^{n-2} $ is a shorthand for sequences indexed by the $n-2$ non-root internal vertices of $ T $.

	The nested set $ \CS(T) $ also gives us a point of $ \CM_{n+1} $ (as in \secref{Nested}).  By an abuse of notation, we will often denote this point by $ T $ or by the bracketing corresponding to $ T $.
		
	\subsection{Cactus group} \label{se:CactusCn}
	We write $ C_n = C_{\Delta(\sl_n)} $.  In this case, it is natural to write the generators of the cactus group as $ s_{pq}$ for $ 1 \le p < q \le n $.  The pair $ p,q $ corresponds to the connected subset $ J = \{p, \dots, q-1\} $ of the Dynkin diagram $ I = \{1,\dots, n-1\}$.
	
	Thus, the relations in $C_n $ are as follows
	\begin{enumerate}
		\item $s^2_{pq}=1 \; \; \forall \, p < q $
		\item $s_{pq}s_{kl} = s_{p+q-l \, p+q-k}s_{pq} \; \; $ if $ p \le k < l \le q $.
		\item $s_{pq}s_{kl}=s_{kl}s_{pq} \; \; $ if $ l < p $ or $ q < k $.
	\end{enumerate}
	
	As above, we have a natural map $ C_n \rightarrow S_n $ which carries $ s_{pq} $ to the permutation $ w_{pq} $ which reverses the segment $[p,q] $ and acts by the identity outside this segment.
	
	As above, from the work of \cite{DJS}, we have an isomorphism
	$$
	C_n \cong \pi_1^{S_n}(\CM_{n+1}, \uz)
	$$
	where $ \uz $ is a fixed base point.

	\section{Coboundary categories and moduli spaces}
	
	\subsection{Coboundary categories and the cactus group}
	\begin{defn} \label{Def:Coboundary}
	A \emph{monoidal category} is a category $ \cC $ along with a pair $ (\otimes, \alpha) $ where $ \otimes : \cC \times \cC \rightarrow \cC $ is a bifunctor and $ \alpha_{A,B,C} : (A \otimes B) \otimes C \rightarrow A \otimes (B \otimes C) $ is a natural transformation such that the following pentagon commutes
	\begin{equation} \label{eq:pent}
	\begin{tikzcd}
			((A \otimes B) \otimes C) \otimes D \ar[r,"{\alpha_{A, B, C} \otimes id}", outer sep = 4 pt] \ar[d,"\alpha_{A\otimes B, C, D}"'] &
			(A \otimes (B \otimes C)) \otimes D \ar[r,"\alpha_{A, B \otimes C, D}"'', outer sep = 4 pt] &
			A \otimes ((B \otimes C) \otimes D) \ar[d, "{id \otimes \alpha_{B, C,D}}"]  \\
			(A \otimes B) \otimes (C \otimes D) \ar[rr, "{\alpha_{A, B, C \otimes D}}"', outer sep = 2 pt] & &
			A \otimes (B \otimes (C \otimes D))
	\end{tikzcd}
	\end{equation}
	
	A \emph{coboundary category} is a monoidal category $ \cC $ along with a natural isomorphism $$ \sigma_{A, B} : A \otimes B \rightarrow B \otimes A$$
	satisfying the following two axioms.
	\begin{enumerate}
		\item For all $A, B \in \cC$, we have $\sigma_{B, A} \circ \sigma_{A,B} = id_{A \otimes B} $
		\item For all $ A, B, C \in \cC $, the following hexagon commutes
		\begin{equation} \label{eq:hex}
		\begin{tikzcd}
		  & (A \otimes B) \otimes C \ar[r,"{\alpha_{A, B, C}}"'', outer sep = 2 pt] \ar[dl,"{\sigma_{A\otimes B, C}}"', outer sep = -2 pt] & A \otimes (B \otimes C) \ar[dr,"{\sigma_{A, B \otimes C}}", outer sep = -2 pt]& \\
				C \otimes (A \otimes B) \ar[dr,"{id \otimes \sigma_{A,B}}"', outer sep = -2 pt] & & & (B \otimes C) \otimes A \ar[dl,"{\sigma_{B, C} \otimes id}", outer sep = -2 pt] \\
				& C \otimes (B \otimes A) \ar[r,"{\alpha_{C, B, A}^{-1}}"', outer sep = 2 pt] & (C \otimes B) \otimes A &
		\end{tikzcd}
		\end{equation}
	\end{enumerate}
	\end{defn}
	
	\begin{rem}
		The notion of a coboundary category was introduced by Drinfeld \cite{D} and was studied by the second author and Henriques in \cite{HK}.
	\end{rem}
	
	The second author and Henriques \cite{HK} proved the cactus group $C_n $ acts on tensor products of $n$ objects in coboundary categories.
	\begin{thm} \label{th:CactusCoboundary}
		Let $ \cC $ be a coboundary category and let $ A_1, \dots, A_n $ be $ n $ objects.  For any $ g \in C_n $, there is a morphism
		$$
		g : A_1 \otimes \cdots \otimes A_n \rightarrow A_{g(1)} \otimes \cdots \otimes A_{g(n)}
		$$
		where we use a fixed bracketing of the objects to define this tensor product.  These morphisms compose as multiplication in the group $ C_n $ and they are all the maps obtained from iterating the commutors and associators.
	\end{thm}

	Let $ \cC $ and $ \cD $ be two coboundary categories.
	
	\begin{defn}
	A \emph{coboundary functor} $ \Phi : \cC \rightarrow \cD $ is a functor $ \Phi $ along with a natural isomorphism $ \phi: \otimes \circ (\Phi \times \Phi)  \rightarrow \Phi \circ \otimes $ which is compatible with the associators and the commutors.
	
	This means that for any three objects $ A, B, C$ of $ \cC $, the following diagram commutes
	\begin{equation*}
		\begin{tikzcd}
			(\Phi(A) \otimes \Phi(B)) \otimes \Phi(C) \ar[r,"{\alpha_{\Phi(A), \Phi(B), \Phi(C)}}", outer sep = 4 pt] \ar[d,"{\phi_{A \otimes B, C} \circ (\phi_{A,B} \otimes id)}"', outer sep = 2 pt] & \Phi(A) \otimes (\Phi(B) \otimes \Phi(C)) \ar[d,"{\phi_{A, B \otimes C} \circ (id \otimes \phi_{B,C})}", outer sep = 2 pt]\\
			\Phi((A \otimes B) \otimes C) \ar[r,"{\Phi(\alpha_{A,B,C})}"', outer sep = 2 pt] & \Phi(A \otimes (B \otimes C))
		\end{tikzcd}
	\end{equation*}
	and for any two objects $ A, B $ of $ \cC $ the following diagram commutes
	\begin{equation} \label{eq:commutors}
		\begin{tikzcd}
			\Phi(A) \otimes \Phi(B) \ar[r,"{\sigma_{\Phi(A), \Phi(B)}}", outer sep = 4 pt] \ar[d,"{\phi_{A,B}}"'] & \Phi(B) \otimes \Phi(A) \ar[d,"{\phi_{B,A}}"] \\
			\Phi(A \otimes B) \ar[r,"{\Phi(\sigma_{A,B})}"', outer sep = 2 pt] & \Phi(B \otimes A) \\
		\end{tikzcd}
	\end{equation}
	
	\end{defn}
	
	\subsection{$\Xi$-coloured categories}
	We will define the notion of a $\Xi$-coloured category, where $ \Xi $ is a set.  To motivate this notion, suppose that $ \cC $ is a $ \C$-linear semisimple abelian category.  Let $ \{ V(\l) \}_{\l \in \Xi} $ be a complete list (up to isomorphism) of the simple objects of $ \cC$, where $\Xi $ is an indexing set.  Then every object of $ \cC $ is canonically isomorphic to a direct sum of its isotypic components $$ A = \bigoplus_{\l \in \Xi} A_\l \otimes V(\l) $$ where each $A_\l $ is a finite dimensional vector space, with only finitely many $ A_\l $ non-zero.
	
	Let $ A = \oplus_\l A_\l \otimes V(\l) $, $ B = \oplus_\l B_\l \otimes V(\l) $.  Then by the categorical version of Schur's Lemma, we have
	$$
	\Hom_{\cC}(A, B) = \bigoplus_\l \Hom_\C(A_\l, B_\l)
	$$
	In particular, the category $ \cC$ is equivalent to the category of $ \Xi$-graded vector spaces.
	
	The objects we consider in this paper, such as crystals, have no linear structure, but still admit isotypic components.  This motivates us to consider $\Xi$-graded sets.
	
	\begin{defn}
	Let $ \Xi$ denote an arbitrary set.  We define the category of $\Xi$-graded sets, $ \Xiset $, as follows.  The objects of $ \Xiset$ are sequences $ (A_\l)_{\l \in \Xi} $ where each $ A_\l $ is a finite set, with only finitely many $ A_\l $ non-empty.  The morphisms of $ \Xiset $ are given by
	$$
	\Hom_{\Xiset}(A, B) = \prod_\l \Hom_{\set}(A_\l, B_\l)
	$$
	
	A category $ \cC $ is called \emph{$\Xi$-coloured}, if we are given an equivalence between $ \cC $ and $ \Xiset $.	
	\end{defn}
	
	So a $\Xi$-coloured category is quite boring, but it might carry an interesting monoidal or coboundary structure.
	
	\subsection{Operadic coverings} \label{se:Operadic}
	Again fix a set $ \Xi $.
	
	We are going to define the notion of an ``operadic covering'' of the moduli spaces $ \CM_{n+1}(\BR) $.  This means that we have a cover of each moduli space, such that along the boundary strata, the cover is isomorphic to the product of the covers of smaller moduli spaces.  Moreover, we will actually have a sequence of coverings $ \cX_{n+1}(\l_1, \dots, \l_n)^\mu $ where the $ \l_i, \mu $ lie in $ \Xi $ and should be thought of as colours on the marked points.
	
	\begin{defn}
	A $\Xi$-\emph{coloured operadic covering of the moduli spaces of genus 0 real curves} is a sequence of covering spaces $  \cX_{n+1} \rightarrow \CM_{n+1}(\BR) \times \Xi^{n+1}$  along with the following data. 
	\begin{enumerate}
		\item For each $ n $, an action of $ S_n $ on $ \cX_{n+1} $, compatible with the actions of $ S_n $ on $\CM_{n+1}(\BR) $ and on $ \Xi^{n+1} $ (permuting the first $ n $ factors).
		\item For each $ \ul = (\l_1, \dots, \l_n) \in \Xi^n $ and $ \mu \in \Xi $, we let $ \cX_{n+1}(\ul)^\mu $ be the fibre over $ (\l_1, \dots, \l_n, \mu) \in \Xi^{n+1} $; this is a covering of $ \CM_{n+1}(\BR) $.
		
		Then, for each $ \umu = (\mu_1, \dots, \mu_n) \in \Xi^n $ and each $  \ul^{(1)} \in \Xi^{k_1}, \dots, \ul^{(n)} \in \Xi^{k_n} $, and $ \nu \in \Xi$, we have a map
		$$
		\Gamma : \cX(\umu)^\nu \times \cX(\ul^{(1)})^{\mu_1} \times \cdots \times \cX(\ul^{(n)})^{\mu_n} \rightarrow \cX(\ul^{(1)} \sqcup \cdots \sqcup \ul^{(n)})^\nu
		$$
		covering the map
		$$
		\gamma : \CM_{n+1}(\BR) \times \CM_{k_1+1}(\BR) \times \cdots \times \CM_{k_n+1}(\BR) \rightarrow \CM_{k_1 + \cdots + k_n + 1}(\BR)
		$$
	\end{enumerate}
	We require this data to satisfy the following conditions
	\begin{enumerate}
		\item The actions of $ S_n $ and the maps $ \Gamma $ give $\cX_{n+1} $ the structure of a coloured operad (see \defnref{Def:ColouredOperad}).
				
		\item For any  $  \ul^{(1)} \in \Xi^{k_1}, \dots, \ul^{(n)} \in \Xi^{k_n} $ and $ \nu \in \Xi$, the map $ \Gamma $ gives a bijection between
		$$  \bigsqcup_{\umu \in \Xi^n} \cX(\umu)^\nu \times \cX(\ul^{(1)})^{\mu_1} \times \cdots \times \cX(\ul^{(n)})^{\mu_n}
		$$
		and the restriction of $ \cX(\ul^{(1)} \sqcup \cdots \sqcup \ul^{(n)})^\nu $ to the locus $\gamma(\CM_{n+1}(\BR) \times \CM_{k_1+1}(\BR) \times \cdots \times \CM_{k_n+1}(\BR))
		$
	\end{enumerate}
	\end{defn}
	For the remainder of this section, we will say ``$\Xi$-coloured operadic covering'' to refer to this definition.
	
	\begin{rem}
		This definition should be compared to the notion of genus 0 modular functor, as defined in Bakalov-Kirillov \cite{BaK}.
	\end{rem}
	
	\subsection{From operadic coverings to coboundary categories} \label{se:CoverToCategory}
	
	We can go back and forth between $\Xi$-coloured operadic coverings and $\Xi$-coloured coboundary categories.
	
	Let $  \cX_{n+1} \rightarrow \CM_{n+1}(\BR) \times \Xi^{n+1}$ be a $\Xi$-coloured operadic covering.  From this data, we will define a coboundary category $ \cC(\cX) $ as follows.
	
	As a category, $ \cC(\cX) $ is just $ \Xiset $.  The covering will be used to define the coboundary structure.
	
	We define the tensor product $ \otimes $ on $ \cC(\cX) $ using $ \cX_3 $.  Note that $ \CM_3(\BR)$ is a point, so $ \cX_3(\l_1, \l_2)^\mu $ is a finite set for each $\l_1, \l_2, \mu \in \Xi $. We will use this finite set as the tensor product multiplicity.  More precisely, for two objects $ A =(A_\l)_{\l \in \Xi}, B = (B_\l)_{\l \in \Xi} $ we define
	$$
	(A \otimes B)_\mu = \bigsqcup_{\l_1, \l_2 \in \Xi} \cX(\l_1, \l_2)^\mu \times A_{\l_1} \times B_{\l_2}
	$$
	
	Let $ S(\l) $ denote the object of $ \cC(\cX) $ defined by
	$$
	S(\l)_{\l'} = \begin{cases} \{ * \}, \text{ if $\l = \l'$ } \\
	\emptyset, \text{ otherwise }
	\end{cases}
	$$
	So we have
	$$ (S(\l_1) \otimes S(\l_2))_\mu = \cX(\l_1, \l_2)^\mu $$
	
	For defining the associator $ \alpha $ and the commutor $ \sigma $ and checking their axioms, it will be convenient to just work with these objects $ S(\l) $.  The interested reader can extend our definitions to arbitrary objects.
	
	Before proceeding to the definition of $ \alpha$ and $\sigma$, we note that if we have any ordered complete bracketing $ T $ of $ \{1, \dots, n\} $, along with a list $ \ul =  (\l_1, \dots, \l_n) \in \Xi^n $, then we can consider a tensor product $ S(\ul)_T$.  For example, associated to $ T = (12)3$, we have the tensor product $$S(\ul)_{(12)3} = (S(\l_1) \otimes S(\l_2) ) \otimes S(\l_3)$$
	which is given by
	$$
	(S(\ul)_{(12)3})_\nu = \bigsqcup_{\mu} \cX(\l_1, \l_2)^\mu \times \cX(\mu, \l_3)^\nu
	$$
	Continuing in this example, we can consider the map
	$$
	\Gamma : \bigsqcup_{\mu_1, \mu_2} \cX(\mu_1, \mu_2)^\nu \times \cX(\l_1, \l_2)^{\mu_1} \times \cX(\l_3)^{\mu_2} \rightarrow \cX(\l_1, \l_2, \l_3)^\nu_{(12)3},
	$$
	covering the map
	$$
	\gamma:pt=\CM_3\times\CM_3\times\CM_2\to\CM_4, 
	$$
	where we denote by $ (12)3$ the point in $ \CM_4(\BR) $ which is the image of the above map $\gamma$.  By the definition of an operadic cover, the map $\Gamma$ is a bijection.  Using this bijection and the unit axiom, we can identify
	$$
	(S(\ul)_{(12)3})_\nu = \cX(\l_1, \l_2, \l_3)^\nu_{(12)3}
	$$
	
	From the associativity axiom for a coloured operad, we immediately deduce the following.
	\begin{lem} \label{BracketsPoints}
		For any ordered complete bracketing $ T $ of $ \{1, \dots, n\} $ and any $ \ul \in \Xi^n $, we can use $ \Gamma $ to canonically identify
		$$
		(S(\ul)_T)_\mu = \cX(\ul)^\mu_T
		$$
		for all $ \mu \in \Xi$.
	\end{lem}
	
	With \lemref{BracketsPoints} in mind, we can define $ \alpha$ on the category $ \cC(\cX) $ as follows
	\begin{align*}
		\alpha_{\l_1, \l_2, \l_3} : (S(\l_1) \otimes S(\l_2)) \otimes S(\l_3) &\rightarrow S(\l_1) \otimes (S(\l_2)\otimes S(\l_3)) \\
		\text{ by } (S(\ul)_{(12)3})_\mu = \cX(\ul)^\mu_{(12)3} &\xrightarrow{p_{(12)3,1(23)}}  \cX(\ul)^\mu_{1(23)} = (S(\ul)_{1(23)})_\mu
	\end{align*}
	where, as above, $p_{(12)3,1(23)}$ denotes the unique path in $ \CM_4(\BR)_+ $ between $ (12)3 $ and $ 1(23) $, and where (by abuse of notation) we also use $ p_{(12)3,1(23)} $ to denote the monodromy of the cover $ \cX(\ul)^\mu $ along this path.
	
	Similarly, we define $\sigma $ on the category $ \cC(\cX) $ as follows
	\begin{align*}
		\sigma_{\l_1, \l_2} : S(\l_1) \otimes S(\l_2) &\rightarrow S(\l_2) \otimes S(\l_1) \\
		\text{ by } S(\l_1 , \l_2)_\mu = \cX(\l_1, \l_2)^\mu &\xrightarrow{w_{12}} \cX(\l_2, \l_1)^\mu = S(\l_2, \l_1)_\mu
	\end{align*}
	where $ w_{12} $ denotes the non-trivial element of $ S_2 $ which is acting on $ \cX_3 $ by the definition of operadic covering.

	We will now verify that we have constructed a coboundary category.
	
	\begin{thm} \label{th:CoverToCategory}
		If $ \cX $ is a $\Xi$-coloured operadic covering, then the above $ \otimes, \alpha, \sigma $ define a coboundary category $ \cC(\cX) $.
	\end{thm}
	\begin{proof}
		We must check the pentagon axiom for $ \alpha$, the symmetry axiom for $ \sigma $, and the hexagon axiom involving $ \alpha, \sigma $.
		
		\noindent {\bf Pentagon axiom:} \\
		Consider $ \l_1, \l_2, \l_3, \l_4 \in \Xi $.  Then we must show that the following diagram commutes
		\begin{equation*}
			\begin{tikzcd}[ column sep = -6 em]
				& (S(\l_1) \otimes (S(\l_2) \otimes S(\l_3))) \otimes S(\l_4) \ar[dr,"{\alpha_{S(\l_1), S(\l_2) \otimes S(\l_3), S(\l_4)}}", outer sep = -2 pt] & \\
				((S(\l_1) \otimes S(\l_2)) \otimes S(\l_3)) \otimes S(\l_4) \ar[ur,"{\alpha_{S(\l_1), S(\l_2), S(\l_3)} \otimes id}", outer sep = -2 pt] \ar[d,"{\alpha_{S(\l_1)\otimes S(\l_2), S(\l_3), S(\l_4)}}"', outer sep = 2 pt] &	 &
				S(\l_1) \otimes ((S(\l_2) \otimes S(\l_3)) \otimes S(\l_4)) \ar[d,"{id \otimes \alpha_{S(\l_2), S(\l_3),S(\l_4)}}", outer sep = 2 pt]  \\
				(S(\l_1) \otimes S(\l_2)) \otimes (S(\l_3) \otimes S(\l_4)) \ar[rr,"{\alpha_{S(\l_1), S(\l_2), S(\l_3) \otimes S(\l_4)}}"', outer sep = 6 pt] & &
				S(\l_1) \otimes (S(\l_2) \otimes (S(\l_3) \otimes S(\l_4)))
			\end{tikzcd}
		\end{equation*}
		When we apply \lemref{BracketsPoints}, the definition of $\alpha$, and the fact that $ \Gamma $ commutes with parallel transport, we see that this equation turns into (for each $ \mu \in \Xi$ ) the diagram
		\begin{equation*}
			\begin{tikzcd}
				& \cX(\ul)^\mu_{(1(23))4}  \ar[dr,"{p_{(1(23))4, 1((23)4)}}", outer sep =- 2 pt] & \\
				\cX(\ul)^\mu_{((12)3)4} \ar[ur,"{p_{((12)3)4, (1(23))4}}", outer sep = -2 pt] \ar[d,"{p_{((12)3)4, (12)(34)}}"', outer sep = 2 pt] & &
			           \cX(\ul)^\mu_{1((23)4)}		\ar[d,"{p_{1((23)4), 1(2(34))}}", outer sep = 2 pt]  \\
				\cX(\ul)^\mu_{(12)(34)} \ar[rr,"{p_{(12)(34), 1(2(34))}}"', outer sep = 2 pt]& &  \cX(\ul)^\mu_{1(2(34))}
			\end{tikzcd}
		\end{equation*}
		Both sides in this diagram give us the monodromy along a path from $ ((12)3)4 $ to $ 1(2(34)) $ staying inside $ \CM_5(\BR)_+$.  Since this is a contratible space, the monodromies along these two paths are equal and so the diagram commutes.
		
		\noindent {\bf Symmetry axiom:} \\
		Since $ w_{12}^2 = 1 $ in $ S_2$, the symmetry of $ \sigma $ is immediate.
		
		\noindent {\bf Hexagon axiom:} \\
		Consider $ \l_1, \l_2, \l_3 \in \Xi $.  We must check that the following diagram commutes
		\begin{equation} \label{eq:hex2}
			\begin{tikzcd}
 \big(S(\l_1) \otimes S(\l_2)\big) \otimes S(\l_3) \ar[r,"{\alpha_{S(\l_1), S(\l_2), S(\l_3)}}", outer sep = 6 pt] \ar[d,"{\sigma_{S(\l_1)\otimes S(\l_2), S(\l_3)}}"'] & S(\l_1) \otimes \big(S(\l_2) \otimes S(\l_3)\big) \ar[d,"{\sigma_{S(\l_1), S(\l_2) \otimes S(\l_3)}}"] \\
				S(\l_3) \otimes \big(S(\l_1) \otimes S(\l_2)\big) \ar[d,"{id \otimes \sigma_{S(\l_1),S(\l_2)}}"'] &  \big(S(\l_2) \otimes S(\l_3)\big) \otimes S(\l_1) \ar[d,"{\sigma_{S(\l_2), S(\l_3)} \otimes id}"] \\
				 S(\l_3) \otimes \big(S(\l_2) \otimes S(\l_1)\big) \ar[r,"{\alpha_{S(\l_3), S(\l_2), S(\l_1)}^{-1}}"', outer sep = 6 pt] & \big(S(\l_3) \otimes S(\l_2)\big) \otimes S(\l_1)
			\end{tikzcd}
		\end{equation}
		When we apply \lemref{BracketsPoints}, the definitions of $\alpha$ and $ \sigma $, the equivariance axiom for coloured operads, and the fact that $ \Gamma $ commutes with parallel transport, we see that this equation turns into (for each $ \mu \in \Xi$ ) the diagram
		\begin{equation} \label{eq:hex3}
			\begin{tikzcd}[cramped, column sep = scriptsize]
				\cX(\l_1, \l_2, \l_3)_{(12)3}^\mu  \ar[r,"{p_{(12)3, 1(23)}}", outer sep = 6 pt] \ar[d,"{u}"'] & \cX(\l_1, \l_2, \l_3)_{1(23)}^\mu \ar[d,"{u^{-1}}"] \\
				\cX(\l_3, \l_1, \l_2)^\mu_{1(23)} \ar[d,"{w_{23}}"']  & \cX(\l_2, \l_3, \l_1)_{(12)3}^\mu  \ar[d,"{w_{12}}"] \\
				\cX(\l_3, \l_2, \l_1)_{3(21)}^\mu \ar[r,"{p_{1(23),(12)3}}"', outer sep = 8 pt] & \cX(\l_3, \l_2, \l_1)_{(12)3}^\mu
			\end{tikzcd}
		\end{equation}
		where $u $ is the permutation $231 $.  Note that the products along the right and left vertical arrows both give the long element $ w_{13} $.  Thus the commutativity of the diagram follows from noting that $ w_{13} $ takes the point $ 1(23) $ to the point $ 3(21)= (12)3 $ and so $w_{13}(p_{1(23),(12)3}) = p_{(12)3, 1(23)}$.
	\end{proof}
	
	\subsection{From $\Xi$-coloured categories to $\Xi$-coloured operadic covers}
	We will now explain how to go from a coboundary category structure on $ \Xiset$ to an operadic covering.  Since we will not use this further in this paper, we will just give a sketch of this construction.
	
	We define
	$$ X_{n+1} = \bigsqcup_{\l_1, \dots, \l_n, \mu} (( \cdots (S(\l_1) \otimes S(\l_2)) \otimes \cdots ) \otimes S(\l_n))_\mu $$
	We have an obvious projection $ X_{n+1} \rightarrow \Xi^{n+1} $.  By Theorem \ref{th:CactusCoboundary}, there is an action of $C_n $ on $X_{n+1} $ compatible with this projection.
	
	Applying \thmref{cor-DJS} and \lemref{le:CoverEquiv}, we obtain an $S_n $-equivariant covering $ \cX_{n+1}$ of $ \CM_{n+1}(\BR) $ for each $ n $.
	
	Now, we will define the maps $ \Gamma $ in order to give this the structure of an operadic covering.
	
	First, note that for any two objects $ A, B \in \Xiset $, we see that
	$$
	(A \otimes B)_\mu = \bigsqcup_{\l_1, \l_2 \in \Xi} (S(\l_1) \otimes S(\l_2))_\mu \times A_{\l_1} \times B_{\l_2}
	$$
	(because we require that the tensor product must commute with the coproduct in the category $ \Xiset $).
	
	Let  $  \ul^{(1)} \in \Xi^{k_1}, \dots, \ul^{(n)} \in \Xi^{k_n} $ and $ \nu \in \Xi$.  From the above description of the tensor product, we see that there is a canonical bijection provided by the associator
	$$
	S(\ul^{(1)} \sqcup \cdots \sqcup \ul^{(n)})_\nu \rightarrow \bigsqcup_{\umu} S(\ul^{(1)})_{\mu_1} \times \cdots \times S(\ul^{(n)})_{\mu_n} \times S(\umu)_\nu
	$$
	(here $ S(\umu) $ denotes any fixed bracketed tensor product of $ S(\mu_1), \dots, S(\mu_n) $.)  Using these canonical bijections, we get the bijection $\Gamma$.
	\begin{thm}
		This defines a $\Xi$-coloured operadic covering.
	\end{thm}
	
	The constructions of this section and the previous one lead to the following theorem.
	
	\begin{thm}
		We have an equivalence of categories between the category of $\Xi$-coloured operadic coverings and the category of $\Xi$-coloured coboundary categories.  The morphisms in the latter category are given by tensor functors $ (\Phi, \phi) $, where $ \Phi $ commutes with the equivalence with $ \Xiset$.
	\end{thm}
	
	\section{Crystals}
	
	\subsection{Crystals and normal crystals}
	Recall that we fixed a semisimple Lie algebra $ \fg $ with Cartan subalgebra $ \fh $.  Let $ \Lambda \subset \fh^* $ denote the weight lattice and let $ \Lambda_+ $ denote the set of dominant weights.  We fix root vectors $ e_\alpha, f_\alpha  $ for each $ \alpha \in \Delta_+ $.  For each $ \lambda \in \Lambda_+ $, we write $ V(\l) $ for the irreducible representation of $ \fg $ of highest weight $ \l $.
	
	A crystal $ B $ is a finite set which models a weight basis for a representation of $ \fg $.  A crystal has two structures: a weight map, which tells us in which weight space our basis vector lives, and partially defined maps $ e_i, f_i $ for $ i \in I $ which indicate the leading order behaviour of the simple root vectors $ e_{\alpha_i}, f_{\alpha_i} $ on the basis.  Here is the precise definition.
	
	\begin{defn}
		A $\fg$-\emph{crystal} is a finite set $ B $ along with a map $ wt : B \rightarrow \Lambda $ and maps $ e_i, f_i : B \rightarrow B \sqcup \{0\} $ for each $ i \in I $ satisfying the following conditions
		\begin{enumerate}
			\item For each $ i \in I, b \in B $, if $ e_i(b) \ne 0 $, then $ wt(e_i(b)) = wt(b) + \alpha_i $
			\item For each $ i \in I, b \in B $, if $ f_i(b) \ne 0 $, then $ wt(f_i(b)) = wt(b) - \alpha_i $
			\item For each $ i \in I, b \in B $, if $ e_i(b) \ne 0 $, then $ f_i(e_i(b)) = b $
			\item For each $ i \in I, b \in B $, if $ f_i(b) \ne 0 $, then $ e_i(f_i(b)) = b $
		\end{enumerate}
	\end{defn}
	
	Two elements $ b, b' $ are said to lie in the same \emph{connected component} of $ B $ if they can be connected by a sequence of the crystal operators $ e_i $ and $ f_i $.
	
	If $ B_1, B_2 $ are two crystals, then their disjoint union $ B_1 \sqcup B_2 $ carries a crystal structure, where $ e_i, f_i $ are defined in the natural way.
	
	For each $ i \in I $ and $ b \in B $, let
	\begin{align*}
		\varepsilon_i(b) &= \max(k : e_i^k(b) \ne 0) \\
		\varphi_i(b) &= \max(k : f_i^k(b) \ne 0)
	\end{align*}
	
	A crystal $ B $ is called \emph{semi-normal} if for each $ i \in I $ and $ b \in B $, we have
	$$ \varphi_i(b) - \varepsilon_i(b) = \langle \alpha_i^\vee, wt(b) \rangle $$
	
	A morphism between two crystals $ B_1, B_2 $ is a map $ \psi : B_1 \rightarrow B_2 $ commuting with the structure maps $ wt, e_i, f_i $.  (This is sometimes called a strict morphism elsewhere in the literature.)
	
	Let $ J \subset I $.   We write $ \fg_J $ for the semisimple Lie algebra generated by all $ e_{\alpha_j}, f_{\alpha_j} $, for $ j \in J $ (this is the derived subalgebra of a Levi subalgebra).  We have a surjection $ \fh^* \rightarrow \fh_J^* $ which takes the weight lattice $ \Lambda $ of $ \fg $ to the weight lattice $\Lambda_J $ of $ \fg_J $.
	
	 Let $ B $ be a $ \fg$-crystal, then we define $ B_J $ to be the $ \fg_J $-crystal with the same underlying set as $ B $, but where we only consider $ e_j, f_j $ for $ j \in J $ and where we define $ wt_{B_J} $ to be the composite $ B \xrightarrow{wt_B} \Lambda \rightarrow \Lambda_J $.
	
	There is a construction which assigns a crystal to each representation of $ \fg $ (in fact, there are several such constructions, but they all have the same output).  For each dominant weight $ \lambda $, we let $ B(\lambda) $ denote the crystal of the representation $ V(\l) $.
	
	\begin{defn}
	A crystal $ B $ is called \emph{normal} if it is isomorphic to a disjoint union of the crystals $B(\l) $ (equivalently if it is the crystal of a representation of $\fg$). We let $ \gcrys$ denote the category of normal $\fg $-crystals.
	\end{defn}
	
	We record a few elementary facts about normal crystals.
	\begin{lem} \label{le:normal}
		\begin{enumerate} \item Any connected component of a normal crystal is isomorphic to $ B(\l) $ for some $ \l $.
			\item Any subcrystal of a normal crystal is normal.
			\item
			Any normal crystal admits a canonical decomposition analogous to the isotypic decomposition of a representation of $ \fg $.
			
			$$
			B \cong \sqcu_{\lambda \in \Lambda_+} \Hom(B(\l), B) \times B(\l)
			$$
			\item For any $ \l_1, \l_2 \in \Lambda_+$, we have
			\begin{equation*}
				\Hom_{\gcrys}(B(\l_1), B(\l_2)) = \begin{cases} \emptyset, \ \text{ if   $ \l_1 \ne \l_2 $} \\
				\{ id \}, \  \text{ if  $ \l_1 = \l_2 $}
				\end{cases}
			\end{equation*}
		\end{enumerate}
	\end{lem}
	
	This immediately implies the following description of the category of normal crystals.
	\begin{cor}
	The category $ \gcrys$ is equivalent to the category  $ \Lset$ via the functor
	\begin{align*}
		\Lset &\rightarrow \gcrys\\
		(A_\lambda)_{\lambda \in \Lambda_+} &\mapsto \sqcu_\lambda A_\lambda \times B(\lambda)
	\end{align*}
	\end{cor}
	
	\subsection{Recognizing normal crystals}
	There are a number of results in the literature concerning how to show that a given crystal is normal.  The following two results are immediate.
	\begin{prop}
	Let $ B $ be a $ \fg $-crystal.
	\begin{enumerate}
	\item If $ B $ is normal, then $ B_J $ is normal for any $ J \subset I $.
	\item $ B $ is semi-normal if and only if $ B_{\{i \}} $ is normal for all $ i \in I $.
	\end{enumerate}
	\end{prop}
	
	In particular, we have the following result of Kashiwara \cite{KKMMNN}.
	
	\begin{thm} \label{th:rank2}
		Let $ B $ be a semi-normal $ \fg$-crystal.  $ B $ is normal if and only if $B_J $ is normal for each $ J \subset I $, $ |J| = 2 $.
	\end{thm}
	
	We will also need the following result.  A crystal $ B $ is called \emph{multiplicity-free} if there is at most one element of each weight.
	
	\begin{prop} \label{pr:MultFree}
		Suppose that $ B $ is a multiplicity-free normal crystal.  Let $ B' $ be another semi-normal crystal with a weight-preserving bijection $ B \rightarrow B'  $. Then this bijection is a crystal isomorphism.
	\end{prop}
	
	\begin{proof}
		Using the given bijection, we can imagine that we have two semi-normal crystal structures $ (B, wt, e_i, f_i) $ and $ (B, wt, e'_i, f'_i) $ on the same set $ B $ (with the same weight function).  We wish to show that $ e_i = e'_i, f_i = f'_i $ for all $ i$.
		
		Fix some $ i $.  Suppose that $ f_i \ne f'_i $.  Choose $ b \in B $ of maximal weight such that $ f_i(b) \ne f'_i(b) $.  Since there is at most one element in $ B $ of weight $ wt(b) - \alpha_i $, the only way for $ f_i(b) $ and $ f'_i(b) $ to be unequal is for one of them to be 0 and the other non-zero.  In particular this means that $ \varphi_i(b) \ne \varphi'_i(b) $.
		
		Since $ b $ is chosen with maximal weight, we can see that above $ b$, the two crystals structures are the same.  Hence $ \varepsilon_i(b) = \varepsilon'_i(b) $.  Since both crystals are semi-normal this implies that $ \varphi_i(b) = \varphi'_i(b) $.  This is a contradiction, so no such $ b $ can exist.  This implies that $ f_i = f'_i $ which implies that $ e_i = e'_i $.
	\end{proof}
	
	\subsection{Tensor products}
	\begin{defn} \label{Def:Tensor}
	If $ B_1, B_2 $ are two crystals, then we define $ B_1 \otimes B_2 $ to be the crystal whose underlying set is $ B_1 \times B_2 $ and whose structure maps are defined by
	\begin{align*}
		wt(b_1, b_2) &= wt(b_1) + wt(b_2)\\
		e_i(b_1, b_2) &= \begin{cases} (e_i(b_1),b_2), & \; \text{if} \; \varepsilon_i(b_1) > \varphi_i(b_2) \\
			(b_1,e_i(b_2)), & \; \text{otherwise}
		\end{cases} \\
		f_i(b_1, b_2) &= \begin{cases} (f_i(b_1),b_2), & \; \text{if} \; \varepsilon(b_1) \geq \varphi(b_2) \\
			(b_1,f_i(b_2)), & \; \text{otherwise}
		\end{cases}
	\end{align*}
	\end{defn}
	
	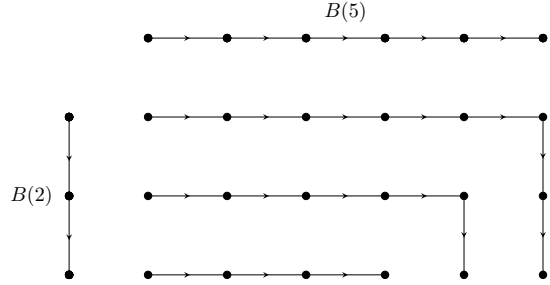
\begin{figure} 
	\begin{center}
		\scalebox{0.7}{
			\begin{tikzpicture}
			[decoration={
				markings,
				mark=between positions .02 and .1 step 15mm with {\arrowreversed{stealth}},
				mark=between positions .115 and 1 step 15mm with {\arrow{stealth}}}]
			\draw[postaction={decorate}]
			(0.5,0) -- (0.5,1.5) node[left=5pt]{$B(2)$} -- (0.5,3)
			(2,4.5) -- (5.75,4.5) node[above=5pt]{$B(5)$} -- (9.5,4.5)
			(2,3) -- (9.5,3) -- (9.5,0)
			(2,1.5) -- (8,1.5) -- (8,0)
			(2,0) -- (6.5,0);
			\foreach \x in {2,3.5,5,6.5,8,9.5}
			\foreach \y in {0,1.5,3}
			\filldraw (\x,\y) circle (2pt) (0.5,\y) circle (2pt) (\x,4.5) circle (2pt);
			\end{tikzpicture}}
	\end{center}
		\caption{The tensor product $B(2)\otimes B(5)$ of the two $\fsl_2$-crystals $B(2)$ and $B(5)$ corresponding to the irreducible $\fsl_2$-representations $V(2)$ and $V(5)$. The arrows denote the Kashiwara operator $ f $.} \label{CrystalTensor}
	\end{figure}
	
	The following results are well-known.
	\begin{prop}
	\begin{enumerate}
	\item The tensor product of normal crystals is again a normal crystal and moreover the decomposition of $ B(\l_1) \otimes B(\l_2) $ matches that of $V(\l_1) \otimes V(\l_2) $.
	\item The tensor product of crystals is ``associative'' in the sense that the obvious map
	\begin{align*}
		\alpha : (B_1 \otimes B_2) \otimes B_3 &\rightarrow  B_1 \otimes (B_2 \otimes B_3) \\
		((b_1, b_2), b_3) &\mapsto (b_1, (b_2, b_3))
	\end{align*}
	is an isomorphism of crystals.  This associator obviously satisfies the pentagon axiom.
	\end{enumerate}
	Thus $ \gcrys$ is a monoidal category.
	\end{prop}
	
	Because the associator for $ \gcrys $ is so simple, we will write multiple tensor products of crystals without brackets.
	
	Given a sequence $ \ul = (\l_1, \dots, \l_n) $, we define
	$$
	B(\ul) = B(\l_1) \otimes \cdots \otimes B(\l_n), \text{ and } B(\ul)^\mu = \Hom(B(\mu), B(\ul))
	$$
	Note that $ B(\ul)^\mu $ is a finite set whose cardinality agrees with the tensor product multiplicity of $ V(\mu) $ in $ V(\ul):= V(\l_1) \otimes \cdots \otimes V(\l_n) $.	
	
	\subsection{The Sch\"utzenberger involution and the crystal commutor} \label{se:commutor}
\begin{defn}
Let $ B $ be a normal crystal.  The \emph{Sch\"{u}tzenberger involution} $ \xi_B : B \rightarrow B $ is the unique map of sets which is natural with respect to crystal morphisms (i.e. it is a natural transformation of the forgetful functor $ \gcrys \rightarrow \set$) and satisfies
\begin{equation} \label{eq:Schutz}
\begin{aligned}
e_i (\xi_B(b)) &= \xi_B(f_{\theta(i)}( b)) \\
f_i (\xi_B(b)) &= \xi_B(e_{\theta(i)}( b)) \\
wt(\xi_B(b)) &= w_0 wt(b)
\end{aligned}
 \end{equation}
for any $ b \in B $.  Here, as in the rest of the paper, $ w_0 $ denotes longest element of the Weyl group and $\theta: I \rightarrow I $ the bijection coming from $ w_0 $.
\end{defn}

The property of being natural with respect to crystal morphisms is equivalent to preserving connected components.  In other words, we can compute $ \xi_B $ by writing $ B $ as a disjoint union of the $ B(\l) $ and applying $ \xi $ to each one.  The existence (and uniqueness) of the Sch\"utzenberger involution was proved in \cite{HK}.

Note that $\xi^2_B $ is a natural automorphism of crystals and since each $ B(\l) $ only admits the identity automorphism, we see that $ \xi_B^2 = 1 $.

Following \cite{HK}, we can use the Sch\"utzenberger involution to define the crystal commutor.

\begin{defn}
Given two normal crystals $ B_1, B_2 $, we define the \emph{crystal commutor}
$$
\sigma_{B_1, B_2} : B_1 \otimes B_2 \rightarrow B_2 \otimes B_1
$$
by
$$ \sigma(b_1, b_2) = \xi_{B_2 \otimes B_1} (\xi_{B_2}(b_2) \otimes \xi_{B_1} (b_1)) $$
\end{defn}

The following result was established in \cite{HK}.
\begin{thm}
$\sigma_{B_1, B_2} $ is an isomorphism of crystals and gives $ \gcrys$ the structure of a coboundary category.
\end{thm}

This theorem and \thmref{th:CactusCoboundary} show that we get an action of the cactus group $C_n$ on a tensor product $ B_1 \otimes \cdots \otimes B_n $ of normal crystals.  We call this the external cactus group action.

\subsection{Internal cactus group action}
\label{sec:action}
Given a normal $ \fg$-crystal $ B$, we will now define an action of the cactus group $C_\Delta $ of type $ \Delta $ on the crystal $ B $.  This action is analogous to Lusztig's quantum Weyl group action on any quantum group representation.

Our construction uses partial Sch\"utzenberger involutions and was first studied in the case $ \fg = \sl_m $ by Berenstein-Kirillov \cite{BeK}.

\begin{defn}
Let $ J \subset I$.  We define the partial Sch\"utzenberger involution $ \xi_J : B \rightarrow B $ by $ \xi_J := \xi_{B_J} $.  So, we regard $ B $ as a $ \fg_J $ crystal and apply its Sch\"utzenberger involution.  In other words, we decompose $ B $ as a $ \fg_J $ crystal and then apply the Sch\"utzenberger involution to each component.
\end{defn}

Consider the case where $ J $ has only one element.  In this case, we can see that $ \xi_i = \xi_{\{i\}} $ just acts as a reflection of each $i$ root string.  More precisely,
$$
\xi_i(b) = \begin{cases} e_i^{-\langle \alpha_i^\vee, wt(b) \rangle}(b), \  \text{ if } \langle \alpha_i^\vee, wt(b) \rangle \le 0 \\
f_i^{\langle \alpha_i^\vee, wt(b) \rangle}(b), \  \text{ if } \langle \alpha_i^\vee, wt(b) \rangle \ge 0
\end{cases}
$$
These bijections $ \xi_i $ were originally studied by Kahiswara.

Now, we consider all the partial Sch\"utzenberger involutions.

\begin{thm} \label{th:InternalAction}
The map $ s_J \mapsto \xi_J $ defines an action of the cactus group $ C_\Delta $ on the set $ B $.  Moreover, for each $ g \in C_\Delta $ and $ b \in B $, we have that $ wt(g(b)) = g(wt(b)) $.
\end{thm}
In the last equality, we use the map $ C_\Delta \rightarrow W $ to define an action of $ C_\Delta $ on the weight lattice.

\begin{rem}
	We emphasize that the internal cactus group acts by set morphisms, not crystal morphisms.  To be more precise, $ C_\Delta $ acts as automorphisms of the forgetful functor from $\gcrys$ to $\set$.
\end{rem}

\begin{proof}
We must check the three relations in the cactus group from \defnref{def:cactus}.

The first relation $ \xi_J^2 =1 $ is clear since the Sch\"utzenberger involution is an involution.

For the second relation, suppose that $J \subset K \subseteq I $.  We must show that $ \xi_K \xi_J = \xi_{\theta_K(J)} \xi_K $.  For simplicity, let us assume that $ K = I$.

It suffices to show that $ \xi \xi_J \xi = \xi_{\theta(J)} $.  Consider $ \xi \xi_J \xi $ as a map $ B_{\theta(J)} \rightarrow B_{\theta(J)} $.  It is immediate that it satisfies the conditions from \eqref{eq:Schutz}.  Thus, it suffices to show that $ \xi \xi_J \xi  $ preserves the connected components of $ B_{\theta(J)} $.  Suppose that $ b, b' \in B $ lie in the same $ B_{\theta(J)} $ connected component.  Then since $ \xi $ commutes $ e_i $ to $ f_{\theta(i)} $, we can see that $ \xi(b), \xi(b') $ lie in the same $ B_J $ connected component.  This implies that $ \xi_J(\xi(b)), \xi_J(\xi(b')) $ lie in the same $B_J $ connected component and thus that $\xi(\xi_J(\xi(b))) $ and  $\xi(\xi_J(\xi(b'))) $ lies in the same $ B_{\theta(J)} $ connected component.  Thus, $ \xi \xi_J \xi $ preserves the connected components of $ B_{\theta(J)} $, and we conclude that $ \xi \xi_J \xi = \xi_{\theta(J)}$.

Before proving the third relation, consider the following general observation.   Suppose that we have a semisimple Lie algebra $ \fg = \fg_1 \oplus \fg_2 $ with two semisimple factors and Dynkin diagram $ I = I_1 \sqcup I_2 $, then every irreducible representation can be written as $ V(\l_1 + \l_2) = V(\l_1) \otimes V(\l_2) $ where $ V(\l_1) $ is an irreducible representation of $ \fg_1 $ and $ V(\l_2) $ is an irreducible representation of $ \fg_2 $.  This implies that $B(\l_1+ \l_2) = B(\l_1) \times B(\l_2) $ where the crystal operators act componentwise.  This implies that the Sch\"utzenberger involutions $ \xi_{I_1} $ and $ \xi_{I_2} $ act componentwise on $ B(\l_1 + \l_2) $ and hence commute. Thus, they commute on any normal $ \fg $-crystal.

Now, to establish the third relation, let $ J, K \subset I $ with $ J \cup K $ disconnected, so $ \fg_{J \sqcup K} = \fg_J \oplus \fg_K $.  Consider the crystal $ B_{J \sqcup K}$.  This will be a normal $ \fg_{J \sqcup K} $ crystal and thus the Sch\"utzenberger involutions $ \xi_J, \xi_K $ will commute by the above observation.

Finally, the statement about the behaviour of weights under the cactus group action follows from the case of the generators $ s_J $.
\end{proof}

\begin{rem}
Kashiwara \cite[Theorem 7.2.2]{K} showed that the $ \xi_i $, for $ i \in I $, satisfy the braid relations and thus the action of the internal cactus group on $ B $ actually factors through the quotient of this group by the braid relations.  We should note that this is not true for the external cactus group action --- the simplest example is the three fold tensor product of adjoint crystals of $ \sl_3$.
\end{rem}

\begin{rem}
In \cite{H}, the first author showed that in type A, the internal and external cactus group actions are related by a crystal version of skew Howe duality.
\end{rem}

\section{Generalities on families of commutative subalgebras and their eigenlines}

\subsection{Families of subalgebras}
For the remainder of the paper, we will be studying families of subalgebras.

Let $ U $ be an algebra equipped with an increasing filtration $ F^0 U \subseteq F^1 U \subseteq F^2 U \subseteq \cdots $ with finite-dimensional filtered pieces.  In our setting, $ U $ will be $ U(\fg) $ or $ U(\fg)^{\otimes n} $, equipped with the PBW filtration.

\begin{defn} \label{Def:Family}
Let $ X $ be a variety.  A \emph{family of subalgebras} of $U $ parametrized by $ X $ means that we are given a subalgebra $ \A(x) $, for each $ x \in X $, such that for each $ N \in \BN $, $ d_N :=  \dim (\A(x) \cap F^N U) $ is independent of $ x $, and such that the resulting map
\begin{align*}
X &\rightarrow \Gr(d_N, F^N U) \\
x &\mapsto \A(x) \cap F^N U
\end{align*}
is a morphism of algebraic varieties (here $ \Gr(d, V)$ denotes the Grassmannian of $ d $-dimensional subspaces of $ V $).
\end{defn}
In this circumstance, we will sometimes say that we have a map from $ X $ to subalgebras of $ U $.

Suppose that we have a family of subalgebras of $ U $ parametrized by a variety $ X$.  Then we can take the closure of this family in the following way.  For each $N$, we let $ Z_N \subset \Gr(d_N, F^N U)  $ be the closure of the image of the map $X \rightarrow \Gr(d_N, F^N U) $.  Then there are surjective restriction maps $Z_N\to Z_M$ for any $M<N$. The inverse limit $Z=\lim\limits_{\leftarrow} Z_N$ is well-defined as a pro-algebraic scheme.  The restriction of the tautological vector bundle on the Grassmannian gives a sheaf $\A$ of commutative algebras on $Z$.  For most families of subalgebras that we consider in this paper, the pro-algebraic scheme $Z$ is in fact an algebraic variety (though not always smooth) and in fact the natural map $Z \rightarrow Z_N $ is an isomorphism for some small number $N $ (often $ N = 2$).

\subsection{Eigenlines}

Let $ V $ be an $n$-dimensional vector space.  Let $ \A \subset \End V $ be a commutative algebra.
\begin{defn}
We say that $ \A $ acts with \emph{simple spectrum} on $ V $, if there exist distinct algebra maps $ \psi_1, \dots, \psi_n : \A \rightarrow \BC $ such that for each $ i $, the eigenspace
$$
E_i = \{ v \in V \ : \  a v = \psi_i(a) v, \text{ for all $ a \in \A $} \}
$$
is one-dimensional.  In this case, $ V = E_1 \oplus \cdots \oplus E_n$ and we call $ \CE_\A(V):= \{ E_1, \dots, E_n \} $ the set of eigenlines for the action of $ \A $ on $ V $.
\end{defn}
Note that if $ \A $ acts  with simple spectrum, then it is just the algebra of diagonal matrices with respect to a fixed basis.  In particular, $ \A \cong \C^n $ as an algebra.  Also note that $ V $ will be a cyclic module for $ \A $.  Conversely, if $ V $ is a cyclic $\A$-module and $\A $ acts semisimply on $ V $, then it has simple spectrum.

Let $ g : V \rightarrow W$ be an isomorphism of vector spaces.  Then we have $ g\A g^{-1} \subset \End W $, the conjugate of $ \A $ by $ g $.
\begin{lem} \label{le:conjugate}
The action of $ g $ gives a bijection $ \CE_\CA(V) \rightarrow \CE_{g\CA g^{-1}}(W)$.
\end{lem}

\subsection{Families of eigenlines}
Suppose we have a family of commutative subalgebras of $ U $ parametrized by $ X $ and we have an algebra map $ U \rightarrow \End V $ for some finite-dimensional vector space $ V $.  This gives us a family of subalgebras of $ \End V$, but typically these subalgebras are not all of the same dimension  nor do they all act semisimply.  So we will typically restrict $ X $ to some topological subspace $ Y \subset X $ (for example $Y = X(\BR) $) such that for each $ y \in Y $, $ \A(y) $ acts semisimply with simple spectrum on $ V $.  Then we get a covering space $ \CE(V) \rightarrow Y $ whose fibres are $ \CE_{\CA(y)}(V) $.  In particular, if we have a path $ p : [0,1] \rightarrow Y $, we get a monodromy map $ p: \CE_{\CA(p(0))}(V) \rightarrow \CE_{\CA(p(1))}(V)$.

Fix such a family and let $ g : V \rightarrow W $ be an isomorphism.

\begin{lem} \label{le:SquareMonodromies}
Let $ p :[0,1] \rightarrow Y $ be a path in $ Y $.  Then we get a resulting path $g(p)$ of subalgebras of $ \End(W) $, given by $ x \mapsto g \A(x) g^{-1} $.  With the above setup, we have the following commutative square of bijections
\begin{equation*}
\begin{tikzcd}
\CE_{\CA(p(0))}(V) \ar[r,"g"] \ar[d,"p"'] & \CE_{\CA(g(p(0)))}(W) \ar[d,"{g(p)}"] \\
\CE_{\CA(p(1))}(V) \ar[r,"g"] & \CE_{\CA(g(p(1)))}(W)
\end{tikzcd}
\end{equation*}
\end{lem}

\section{Gaudin algebras} \label{se:Gaudin}

\subsection{Generalities on Poisson centers}
As above, we work with a fixed semisimple Lie algebra $\fg$. We consider the universal enveloping algebra $U(\fg)$.  In addition to the previous notation, we will also write $ h_\alpha \in \fh $ for the coroot corresponding to $ \alpha \in \Delta_+ $ and in particular we will write $ h_i = h_{\alpha_i} $.

It has a natural PBW filtration such that the associated graded algebra $\gr U(\fg)$ is the symmetric algebra $S(\fg)$. This gives a natural Poisson structure on $S(\fg)$ defined on the generators as the commutator operation. The algebra $S(\fg)$ is naturally isomorphic to $\CO(\fg^*)$, the coordinate ring of the coadjoint representation with the usual Lie-Kirillov-Kostant bracket.

Since $\fg$ is semisimple, the Killing form identifies $\fg^*$ with $\fg$, so we have $S(\fg)=\CO(\fg)$ as well. The Poisson center of $S(\fg)$ is naturally the algebra of invariants $S(\fg)^\fg=\CO(\fg \sslash G)$. The center $ZU(\fg)$ of the universal enveloping algebra $U(\fg)$ equals $U(\fg)^\fg$ and we have $\gr ZU(\fg)=S(\fg)^\fg$. The algebra $S(\fg)^\fg$ is known to be a free polynomial algebra with $ r =\rk\fg$ generators $\Phi_1,\ldots,\Phi_r$. The degrees of the generators are $\deg\Phi_l=d_l+1$ where $d_l$ are the exponents of $\fg$.

Let $e=\sum\limits_{i \in I} e_{\alpha_i}\in\fn_+ $ be the principal nilpotent element.  Consider the $\fsl_2$-tiple $(e,h,f)$. Then we have the Kostant slice in $\fg$, $$
\fg_{\can} = e+\fz_\fg(f)\subset e\oplus\fb_-.
$$
where $ \fz_\fg(f) $ denotes the centralizer of $ f$.
By \cite{Ko}, the adjoint orbit of any regular element in the Lie
algebra $\fg$ contains a unique element which belongs to $\fg_{\can}$. Thus, we have canonical isomorphisms
\begin{equation}
\fg_{\can} \tilde\to \fg \sslash G.
\end{equation}
In particular the restriction of the Poisson center $S(\fg)^\fg$ to $\fg_{can}\subset\fg=\fg^*$ is an isomorphism $S(\fg)^\fg\simeq \CO(\fg_{can})$.

Now let $\fl\subset\fg$ be a reductive subalgebra of $\fg$. We will use the following particular case of Knop's Theorem on a Harish-Chandra homomorphism for reductive group actions:

\begin{thm}\label{th:Knop} (Theorem~10.1 and Main Theorem of \cite{Kn}) The centralizer of $ZU(\fl)$ in $U(\fg)$ is the tensor product $U(\fg)^\fl\otimes_{ZU(\fl)}U(\fl)$. The same is true for the associated graded Poisson algebras: the Poisson centralizer of $S(\fl)^\fl$ in $S(\fg)$ is  the tensor product $S(\fg)^\fl\otimes_{S(\fl)^\fl}S(\fl)$.
\end{thm}

\subsection{Diagonal embeddings}
In this paper, we will work with the direct sum $\fg^{\oplus n}$ of $n$ copies of $\fg$. For $x\in\fg$ we denote by $x^{(i)}$ the image of $x$ in the $i$th summand. We denote by $\Delta$ the diagonal embedding $\Delta:\fg\to\fg^{\oplus n}$, i.e. $\Delta(x):=\sum\limits_{i=1}^nx^{(i)}$.

We will also need to work with a more complicated partial diagonal embedding.  If $ A_1, \dots, A_k $ are $ k $ disjoint subsets of $ \{1, \dots, n \}$, then we can define a partial diagonal embedding
\begin{align*}
\Delta^{A_1, \dots, A_k} : \fg^{\oplus k} &\rightarrow \fg^{\oplus n} \\
(x_1, \dots, x_k) &\mapsto \sum_{j = 1}^k \sum_{i \in A_j} x_j^{(i)}
\end{align*}
We will usually abbreviate this notation, by writing the sets $ A_i $ by listing their elements without any commas.  For example  $ \Delta^{1 \cdots n} = \Delta^{\{1, \dots, n \}} $ is the diagonal embedding defined above.

If $ |A_j | = 1 $ for all $ j$, then we will write $ x^{(a_1 \cdots a_k)} := \Delta^{a_1, \dots, a_k}(x)$.  Note that this is the result of taking $ x $ (a $k$-vector)  and putting it into the $ a_1, \dots, a_k $ summands of $ \fg^{\oplus n } $.  Note that this agrees with our notation $ x^{(i)} $ above.

Finally, each such diagonal embedding $\Delta^{A_1, \dots, A_k} $ induces maps
\begin{gather*}
\Delta^{A_1, \dots, A_k} : S(\fg^{\oplus k}) = S(\fg)^{\otimes k} \rightarrow S(\fg^{\oplus n}) = S(\fg)^{\otimes n} \\
\Delta^{A_1, \dots, A_k} : U(\fg^{\oplus k}) = U(\fg)^{\otimes k} \rightarrow U(\fg^{\oplus n}) = U(\fg)^{\otimes n}
\end{gather*}

\subsection{Classical Gaudin algebras}

The classical Gaudin system is an integrable system on $(\fg^*)^{\oplus n}$, i.e. a Poisson commutative subalgebra in $\CO((\fg^*)^{\oplus n})=S(\fg)^{\otimes n}$, defined as follows. Let $\Phi_l,\ l=1,\dots,r$ be the generators of the algebra of invariants $S(\fg)^{\fg}$. Let $\uz:=(z_1,\ldots,z_n)$ be a collection of pairwise distinct complex numbers. For any $w\in\BC$ consider the linear map $m_{\uz,w}:\fg\to\fg^{\oplus n}$ defined as $$m_{\uz,w}(x):= \Bigl(\frac{x}{w-z_1} , \ldots, \frac{x}{w-z_n}\Bigr).$$ This extends to the homomorphism $m_{\uz,w}:S(\fg)\to S(\fg)^{\otimes n}$.

\begin{defn}
The \emph{classical Gaudin subalgebra} $\ol{\A(\uz)}\subset S(\fg)^{\otimes n}$ is the subalgebra generated by $m_{\uz,w}(\Phi_l)$ for all $l=1,\ldots,r$ and $w\in\BC\backslash\{z_1,\ldots,z_n\}$.
\end{defn}

Note that this subalgebra does not change under simultaneous shifts and simultaneous dilations of the $z_i$'s. Note also that the subalgebra $\ol{\A(\uz)}$ is generated by the coefficients of the principal parts of $m_{\uz,w}(\Phi_l)$ at $w=z_i,\ i=1,\ldots, n$.

\begin{eg}
For $\fg=\fsl_2$ we have the only generator $$\Phi=C=2ef+\frac{1}{2}h^2\in S(\fsl_2).$$
We have $m_{\uz,w}(C)=\sum\limits_{i=1}^n\frac{C_i}{(w-z_i)^2}+\frac{H_i}{w-z_i}$, where
$$C_i=2e^{(i)}f^{(i)}+\frac{1}{2}(h^{(i)})^2 \ \text{ and } H_i=\sum\limits_{j\ne i}\frac{e^{(i)}f^{(j)}+e^{(j)}f^{(i)}+\frac{1}{2}h^{(i)}h^{(j)}}{z_i-z_j}$$
The elements $C_i, H_i$ generate the algebra $ \ol{\A(\uz)} $.  With the exception of the relation $ \sum H_i = 0 $, these elements are algebraically independent and they generate a polynomial ring of degree $ 2n-1 $.   Note that the elements $C_i$ generate the Poisson center of $S(\fsl_2)^{\otimes n}$ while the $H_i$'s (known as Gaudin Hamiltonians) are nontrivial Poisson commuting elements.
\end{eg}

More generally, for any semisimple $\fg$ we have the quadratic Casimir generator $C=\sum\limits_{a=1}^{\dim \fg}x_a^2$ where $x_a$ is an orthonormal basis of $\fg$ with respect to the Killing form. This gives rise to the quadratic generators $$C_i:=\sum\limits_{a=1}^{\dim \fg}(x_a^{(i)})^2 \ \text{ and } H_i:=\sum\limits_{j\ne i}\frac{\Omega^{(ij)}}{z_i-z_j}$$
 where $\Omega^{(ij)}:=\sum\limits_{a=1}^{\dim \fg}x_a^{(i)}x_a^{(j)}$. For $\fg=\fsl_2$, such elements generate the whole subalgebra $\ol{\A(\uz)}$ while for other $\fg$ there are additional generators of higher degree.

\subsection{Universal Gaudin subalgebra} The classical Gaudin subalgebra can be described universally as follows. Consider the Lie algebra $\hat{\fg}_-=t^{-1}\fg[t^{-1}]$. For any collection of nonzero complex numbers $w_1,\ldots,w_n$, we define the Lie algebra homomorphism $\varphi_{w_1,\ldots,w_n}:\hat{\fg}_-\to\fg^{\oplus n}$ given by evaluations at the points $w_1, \dots, w_n$, namely $$x(t)\mapsto (x(w_1), \ldots, x(w_n)).$$
Let us describe the universal Poisson commutative subalgebra $\ol{\A}\subset S(\hat{\fg}_-)$ such that $\ol{\A(\uz)}=\varphi_{w-z_1,\ldots,w-z_n}(\ol{\A})$ for any $w\in\BC\backslash\{z_1,\ldots,z_n\}$.

We have the following derivations of the Lie algebra $\hat{\fg}_-$:
\begin{equation}\label{der1}
\partial_t(g\otimes t^m)=mg\otimes t^{m-1}\quad\forall g\in\fg, m=-1,-2,\dots
\end{equation}
\begin{equation}\label{der2}
t\partial_t(g\otimes t^m)=mg\otimes t^{m}\quad\forall g\in\fg,
m=-1,-2,\dots
\end{equation}
The derivations (\ref{der1}), (\ref{der2}) extend to the
derivations of the algebras $S(\hat{\fg}_-)$ and $U(\hat{\fg}_-)$. The derivation (\ref{der2}) induces a grading of
these algebras.

Let $i_{-1}:S(\fg)\hookrightarrow S(\hat{\fg}_-)$ be the
embedding which maps $g\in\fg$ to $g\otimes t^{-1}$. Let
$\Phi_l,\ l=1,\dots,r$ be the generators of the algebra of
invariants $S(\fg)^{\fg}$. Define $\overline{S_l}:=i_{-1}(\Phi_l)$.

\begin{defn}
The \emph{universal classical Gaudin subalgebra} $\ol{\A}\subset S(\hat{\fg}_-)$ is the subalgebra generated by all $\partial_t^k \ol{S_l}$, $l=1,\dots,r$, $k=0,1,2,\dots$.
\end{defn}

\begin{prop}\label{pr:ClassicalGaudin} For any $w\in\BC\backslash\{z_1,\ldots,z_n\}$, we have $\ol{\A(\uz)}=\varphi_{w-z_1,\ldots,w-z_n}(\ol{\A})$.
\end{prop}

\begin{proof} $\varphi_{w-z_1,\ldots,w-z_n}(\ol{\A})$ does not depend on $w\in\BC$ since the subalgebra $\ol{\A} \subset S(\hat{\fg}_-)$ is stable under $\partial_t$. Indeed it is sufficient to check that for any $S\in\ol{\A}$ the derivative $\partial_w\varphi_{w-z_1,\ldots,w-z_n}(S)$ belongs to $\varphi_{w-z_1,\ldots,w-z_n}(\ol{\A})$, and we have $\partial_w\varphi_{w-z_1,\ldots,w-z_n}(S)=\varphi_{w-z_1,\ldots,w-z_n}(\partial_t S)\in\varphi_{w-z_1,\ldots,w-z_n}(\ol{\A})$.

Moreover, since $\ol{\A}$ is generated by all derivatives of $\ol{S_l}$, $l=1,\ldots,r$, its image under $\varphi_{w-z_1,\ldots,w-z_n}$ is generated by the coefficients of the Laurent expansions of $\varphi_{w-z_1,\ldots,w-z_n}(\ol{S_l})$, $l=1,\ldots,r$, at any point. So it is the same as the subalgebra generated by $\varphi_{w-z_1,\ldots,w-z_n}(\ol{S_l})$, $l=1,\ldots,r$, for all $w\in\BC\backslash\{z_1,\ldots,z_n\}$, which is $\ol{A(\uz)}$.
\end{proof}

\subsection{Quantum Gaudin subalgebras} \label{se:QuantumGaudin}
The universal classical Gaudin algebra can be lifted to a commutative subalgebra $ \A \subset U(\hat{\fg}_-) $ defined with the aid of the centre of the affine Kac-Moody algebra at the critical level. We use the following properties of the subalgebra $ \CA \subset U(\hat{\fg}_-) $, which are immediate from the Feigin-Frenkel \cite{FF} description of the center at the critical level, see \cite{R} for the details.

\begin{prop}\label{pr:S_l} There exist pairwise commuting elements $S_l\in U(\hat{\fg}_-)$ such that  \begin{enumerate} \item $\gr
S_l=\overline{S_l}$. \item The elements $S_l$ are $\fg$-invariant and homogeneous with respect to
$t\partial_t$. \item $\A$ is a free commutative algebra
generated by $\partial_t^k S_l$, $l=1,\dots,r$,
$k=0,1,2,\dots$.\end{enumerate}
\end{prop}

The evaluation homomorphism $\varphi_{w_1,\ldots,w_n}:\hat{\fg}_-\to\fg^{\oplus n}$ extends to an algebra morphism
$$\varphi_{w_1,\ldots,w_n}:U(\hat{\fg}_-)\to U(\fg^{\oplus n})=U(\fg)^{\otimes n}.$$
The image of $ \A \subset U(\hat{\fg}_-) $ under $\varphi_{w - z_1, \dots, w - z_n} $ is a commutative subalgebra $\A(\uz)\subset U(\fg)^{\otimes n}$ called the \emph{(quantum) Gaudin algebra}. Note that $\A(\uz)$ does not depend on $w\in\BC \setminus \{z_1, \dots, z_n \} $ since the subalgebra $\A \subset U(\hat{\fg}_-)$ is still stable under the derivation $\partial_t$ (analogously to Proposition~\ref{pr:ClassicalGaudin}). Since $S_l$ are $\fg$-invariant, the algebra $\A(\uz)$ is a subalgebra in the diagonal invariants $(U(\fg)^{\otimes n})^{\fg}$.

We will see (\propref{pr:Maximal}) that it is in fact a \emph{maximal} commutative subalgebra in $(U(\fg)^{\otimes n})^{\fg}$.

From the definition, it is easy to see that the Gaudin algebras $ \CA(\uz) $ are invariant under the action of $ \C^\times \ltimes \C $ acting on $ \C^n $ by simultaneous scaling and translation (see also \lemref{le:InvarianceA}).  In particular, in the case $ n =2 $, there is only one algebra $ \CA(1, 0) $ since for any distinct complex numbers $ z_1, z_2 $ we can convert them to $ (1,0)$ by the action of $ \C^\times \ltimes \C $.  This subalgebra $ \CA(1,0) \subset U(\fg)^{\otimes 2} $ is thus canonical, yet it remains somewhat mysterious and does not appear to admit any elementary definition.

\section{Gaudin algebras and coverings of moduli spaces}\label{sect-Cover}

Since the Gaudin algebras are invariant under scaling and translation of the parameters, we obtain a family of subalgebras of $ U(\fg)^{\otimes n} $ parametrized by $ \fh(\sl_n)^{reg} / \C^\times =  (\C^n \smallsetminus \Delta)  / \C^\times \ltimes \C  $ in the sense of \secref{Def:Family}.  We will now consider the closure of this family.  Recall that $ \fh(\sl_n)^{reg} / \C^\times $  is compactified by the space $ \CM_{n+1} $.  The following result is due to the third author \cite{Ryb13}.

\begin{thm}
The closure of the family of Gaudin algebras is parametrized by the moduli space $ \CM_{n+1} $.
\end{thm}

In particular, we have a Gaudin algebra  $\CA(\uz) \in (U (\fg)^{\otimes n })^\fg$ for any point $ \uz \in \CM_{n+1}$.  The subalgebras corresponding to boundary points of $\CM_{n+1}$ can be defined inductively using the following operad structure on commutative subalgebras.

For $N=k_1+\ldots+k_n$, we define the homomorphisms $$D_{k_1,\ldots,k_n} := \Delta^{\{1, \dots, k_1\}, \{k_1 + 1, \dots, k_1 +  k_2\}, \dots, \{k_1 + \dots + k_{n-1} + 1, \dots, N \}} :U(\fg)^{\otimes n}\hookrightarrow U(\fg)^{\otimes N}$$
and the homomorphisms
$$ I_i := \Delta^{k_1 + \cdots + k_{i-1} + 1, \dots, k_1 + \cdots + k_{i-1} + k_i} :U(\fg)^{\otimes k_i}\hookrightarrow U(\fg)^{\otimes N}$$
Clearly, all these homomorphisms are $\fg$-equivariant and every element in the image of $D_{k_1,\ldots,k_n}$ commutes with every element of $I_{i}((U(\fg)^{\otimes k_i})^\fg)$ for $i=1,\ldots,n$. This gives the following ``substitution'' homomorphism defining an operad structure on the spaces $(U(\fg)^{\otimes n})^\fg$
\begin{equation}\label{operad-Ug}
\gamma_{n,k_1,\ldots,k_n}=D_{k_1,\ldots,k_n}\otimes\bigotimes\limits_{i=1}^n I_{i}:(U(\fg)^{\otimes n})^\fg\otimes\bigotimes\limits_{i=1}^n (U(\fg)^{\otimes k_i})^\fg\to(U(\fg)^{\otimes N})^\fg.
\end{equation}
Let ${\uz}=\gamma({\uz^0},{\uz^1},\ldots,{\uz^n})$ be a boundary point of $\CM_{k_1+\ldots+k_n+1}$ where ${\uz^0}\in \CM_{n+1}$ and ${\uz^i}\in \CM_{k_i+1}$.

The following result (Theorem~3.13 of \cite{Ryb13}) describes the subalgebra $ \A(\uz) $.
\begin{thm} \label{th:GaudinBoundary}
We have $$ \A({\uz}) = \gamma_{n,k_1,\ldots,k_n}(\A({\uz^0})\otimes\bigotimes\limits_{i=1}^n\A({\uz^i}))$$
\end{thm}

\subsection{Representations and their tensor products}
Given a dominant weight $ \l \in \Lambda_+ $, we have the irreducible representation $ V(\l) $ and given a sequence $ \ul=(\l_1, \dots, \l_n) $, we can consider the tensor product $ V(\ul) = V(\l_1) \otimes \cdots \otimes V(\l_n) $.

Finally, we can consider tensor product multiplicity spaces
$$
V(\ul)^\mu = \Hom_{\fg}(V(\mu), V({\l_1}) \otimes \cdots \otimes V({\l_n}) )
$$

The Gaudin algebras $\A(\uz)$ (as well as their limits described above) act on $ V(\ul) $. Moreover, since $\A(\uz)$ lies in the diagonal invariants in $U(\fg)^{\otimes n}$, it acts on $ V(\ul)^\mu $.

\subsection{The covering} \label{se:covering}
The following result is due to the third author \cite{Ryb16}; it is equivalent to the Bethe Ansatz conjecture in the form of \cite{FFTL}. We will give a proof of (a generalization of) this theorem in Section~\ref{sect-cyclic}, see Theorem~\ref{th:CyclicGeneral} and Corollary~\ref{co:BAC}.

\begin{thm} \label{th:cyclic}
For any $ \uz \in \CM_{n+1} $, $ V(\ul)^\mu $ is a cyclic module for the algebra $ \CA(\uz) $.
\end{thm}

If $ \uz \in \CM_{n+1}(\BR) \subset \CM_{n+1}$, we know that $ \CA(\uz) $ acts semisimply, see \cite{FFR, Ryb13}.  Thus, the above theorem implies that $ \CA(\uz) $ acts with simple spectrum and so decomposes $ V(\ul)^\mu $ into eigenlines.

Let $ \CE_\uz(\ul)^\mu $ be the set of eigenlines of $ \CA(\uz) $ acting on $ V(\ul)^\mu $.

The sets $ \CE_\uz(\ul)^\mu $ form the fibres of a family $ \cE(\ul)^\mu \rightarrow \CM_{n+1}(\BR)$, where
$$
\cE(\ul)^\mu = \{ (\uz, L) : \uz \in \CM_{n+1}(\BR), L \in \PP(V(\ul)^\mu) : \text{ $L$ is an eigenline for $ \CA(\uz) $} \}
$$
It is easy to see that this is a real algebraic variety.

Because of Theorem \ref{th:cyclic}, $ \cE(\ul)^\mu \rightarrow \CM_{n+1}(\BR) $ is a covering space.

We let $ \cE_{n+1} = \cup_{\ul, \mu} \cE(\ul)^\mu $, where the union is taken over $ (\ul, \mu) \in \Lambda_+^{n+1} $.

\begin{thm} \label{th:OperadFromEigenlines}
$ \cE_{n+1} $ is a $\Lambda_+$-coloured operadic covering.
\end{thm}

\begin{proof}
We have to define the bijections $\Gamma$ and the $S_n$ action.

First, let us define the $S_n$ action. Note that for any $w\in S_n$, the subalgebra $ \CA(w(\uz))\subset U(\fg)^{\otimes n}$ is $w(\A({\uz}))\subset U(\fg)^{\otimes n}$, i.e. it is obtained from $\A({\uz})$ by the permutation of the tensor factors determined by $w$. Hence, the permutation of tensor factors of $V(\l_1)\otimes\ldots\otimes V(\l_n)$ given by $w$ takes eigenlines of $\A({\uz})$ to eigenlines of $ \CA(w(\uz))$, so we have the desired $S_n$ action.

Now let us produce the bijections $$
\Gamma : \bigsqcup_{\underline{\mu} \in \Lambda_+^n} \cE(\umu)^\nu \times \cE(\ul^{(1)})^{\mu_1} \times \cdots \times \cE(\ul^{(n)})^{\mu_n} \rightarrow \cE(\ul^{(1)} \sqcup \cdots \sqcup \ul^{(n)})^\nu
$$
covering the map
$$
\gamma : \bigsqcup_{\umu \in \Lambda_+^n} \CM_{n+1}(\BR) \times \CM_{k_1+1}(\BR) \times \cdots \times \CM_{k_n+1}(\BR) \rightarrow \CM_{k_1 + \cdots + k_n + 1}(\BR).
$$
Let ${\uz}=\gamma({\uz^0},{\uz^1},\ldots,{\uz^n})$ be a boundary point of $\CM_{k_1+\ldots+k_n+1}$ where ${\uz^0}\in \CM_{n+1}$ and ${\uz^i}\in \CM_{k_i+1}$.  Let $\ul=\ul^{(1)} \sqcup \cdots \sqcup \ul^{(n)}$, $\ul^{(i)}=(\l_1^{(i)},\ldots,\l_{k_i}^{(i)})$ be a collection of dominant weights and $\nu$ be a dominant weight.  Let
$$V(\ul)=\bigoplus\limits_{\umu}V(\ul)^{\umu}\otimes V(\umu)$$
be the decomposition of $V(\ul)$ into isotypic components, with respect to $D_{k_1,\ldots,k_n}(\fg^{\oplus n})$.  Here $V(\umu)$ is the irreducible representation of $\fg^{\oplus n}$ with highest weight ${\umu}=(\mu_1,\ldots,\mu_n)$ and
$$V(\ul)^{\umu}:=\Hom_{\fg^{\oplus n}}(V(\umu),V(\ul))=\bigotimes\limits_{i=1}^nV(\ul^{(i)})^{\mu_i}$$
is the multiplicity space.

By the above description of $\A(\uz) $ (\thmref{th:GaudinBoundary}), the eigenlines of  $\A({\uz})$ in $V(\ul)^{\nu}$ are tensor products of those of $\A({\uz^0})$ in $V(\umu)^\nu$ and $\bigotimes\limits_{i=1}^n\A({\uz^i})$ in $V(\ul)^{\umu}$. Note that the set of eigenlines of $\A({\uz^0})$ in $V(\umu)^\nu$ is $\cE(\umu)^\nu$ and the set of eigenlines of $\bigotimes\limits_{i=1}^n\A({\uz^i})$ in $V(\ul)^{\umu}$ is the product of $\cE(\ul^{(i)})^{\mu_i}$. Hence we get the desired bijection $\Gamma$. The coloured operad axioms follow from those for the tensor product of representations.

\end{proof}

By the construction from Section \ref{se:CoverToCategory}, the compatible covers $ \CE $ define a coboundary category $ \cC(\CE) $ whose objects are disjoint unions of simple objects $ S(\l)$.  We now give a precise statement of our main result.

\begin{thm} \label{th:Main}
There is an equivalence of coboundary categories $ \cC(\CE) \cong \gcrys $ taking $ S(\l) $ to $ B(\l) $ for each $ \l \in \Lambda_+ $.
\end{thm}

This theorem implies Etingof's conjecture.
\begin{cor}
Fix a base point $ \uz \in \CM_{n+1}$.  For each $ \ul, \mu $, there is a bijection $ \CE_\uz(\ul)^\mu \cong B(\ul)^\mu$.  These bijections are compatible with the actions of $ C_n $.
\end{cor}

\begin{proof}
Since we have a tensor functor, the equivalence $ \cC(\CE) \cong \gcrys $ takes $ S(\ul) $ to $ B(\ul) $.  Thus, we obtain a bijection
$$
\CE_\uz(\l)^\mu = \Hom_{\cC(\CE)}(S(\mu), S(\ul)) \cong \Hom_{\gcrys}(B(\mu), B(\ul)) = B(\ul)^\mu
$$
These bijections are compatible with the actions of $ C_n$, since these actions come from the commutors and associators in the categories.
\end{proof}

To prove the theorem we will use inhomogeneous Gaudin algebras.

\section{Inhomogeneous Gaudin algebras}

\subsection{The definition}
The following definition is due independently to  Feigin-Frenkel-Toledano Laredo \cite{FFTL} and the third author \cite{R}.

We have a Lie algebra map $ \fg \otimes t^{-1} \C[t^{-1}] \rightarrow \fg $ given by extracting the coefficient of $t^{-1} $ (here the codomain $ \fg $ is endowed with the trivial Lie bracket).  This yields an algebra morphism $\varphi_\infty:U(\hat{\fg}_-) \rightarrow S(\fg) $, called \emph{evaluation at $ \infty$}.    Thus given non-zero $ w_1, \dots, w_n \in \C$, we can define a map
$$
\varphi_{w_1,\ldots,w_n,\infty}=\varphi_{w_1,\ldots,w_n}\otimes\varphi_{\infty}:U(\hat{\fg}_-) \rightarrow U(\fg) \otimes \cdots \otimes U(\fg) \otimes S(\fg).
$$
Let $ z_1, \dots, z_n $ be pairwise distinct complex numbers.  The algebra $\A(\uz,\infty):=\varphi_{w-z_1,\ldots,w-z_n,\infty}(\A)$ is a commutative subalgebra in the diagonal invariants $(U(\fg)^{\otimes n} \otimes S(\fg))^{\fg}$, which does not depend on the choice of $w\in\BC\backslash\{z_1,\ldots,z_n\}$. The subalgebra $\A(\uz,\infty)\subset (U(\fg)^{\otimes n} \otimes S(\fg))^{\fg}$ can be regarded as the following family of subalgebras in $U(\fg)^{\otimes n}$.

Given $\chi\in\fg$ we can define the evaluation at $ \chi $, $S(\fg)\to\BC$ (after using the Killing form to identify $ \fg = \fg^* $).  We define $ \varphi_{w_1,\ldots,w_n, \chi} $ to be the composite map
$$
U(\hat{\fg}_-) \xrightarrow{\varphi_{w_1,\ldots,w_n,\infty}} U(\fg) \otimes \cdots \otimes U(\fg) \otimes S(\fg) \rightarrow U(\fg)^{\otimes n}
$$
where the  second map evaluates the last tensor factor at $\chi$. Following $ \CA $ under the homomorphism $ \varphi_{w - z_1, \dots, w - z_n, \chi} $ defines the \emph{inhomogeneous Gaudin algebra} $ \CA_\chi(\uz) \subset U(\fg)^{\otimes n} $ (as usual we can choose any $ w $ distinct from $ z_1, \dots, z_n$).  This is a commutative subalgebra of $ U(\fg)^{\otimes n} $ which is maximal commutative for $ \chi $ regular and $ z_1, \dots, z_n $ distinct. Since $\A$ lies in $\fg$-invariants in $U(\hat{\fg}_-)$ the subalgebra $\A_{\chi}(\uz)$ lies in the diagonal invariants of $\fz_\fg(\chi)$, the centralizer of $\chi$ in $\fg$.

We will also be interested in the subalgebras $\A_{\chi_0}(\uz)$ for non-regular $\chi_0\in\fg$ which are smaller than $\A_{\chi}(\uz)$ for $\chi$ regular. We denote such subalgebras by $\A_{\chi_0}^0(\uz)$ to distinguish them from the limiting subalgebras $\lim\limits_{\chi\to\chi_0}\A_{\chi}(\uz)$ which have the same size as generic ones. We will see later that for semisimple $\chi_0$, the subalgebra $\A_{\chi_0}^0(\uz)$ is a maximal commutative subalgebra in $(U(\fg)^{\otimes n})^{\fz_\fg(\chi_0)}$. If we set $ \chi = 0$, then we then obtain the Gaudin algebras $\A_0^0(\uz)= \CA(\uz) \subset (U(\fg)^{\otimes n})^\fg $.

We will now record some elementary properties of these algebras.

First, we consider the invariance of these algebras under natural changes of the parameters.
\begin{lem} \label{le:InvarianceA}
Let $ \chi \in \fg, (z_1, \dots, z_n) \in \C^n $.   For all $ a \in \C, s \in \C^\times $, we have
\begin{align*}
\A_{\chi}(z_1 + a, \dots, z_n + a) &= \CA_\chi(z_1, \dots, z_n) \\
\A_{\chi}(sz_1, \dots, sz_n) &= \CA_{s\chi}(z_1, \dots, z_n)
\end{align*}
\end{lem}
\begin{proof}
The universal Gaudin algebra $ \CA $  is stable under $\partial_t$. Hence for any $S\in\A$ the derivative $\partial_w\varphi_{w-z_1,\ldots,w-z_n,\chi}(S)$ belongs to $\varphi_{w-z_1,\ldots,w-z_n,\chi}(\A)$: indeed, $$\partial_w\varphi_{w-z_1,\ldots,w-z_n,\chi}(S)=\varphi_{w-z_1,\ldots,w-z_n,\chi}(\partial_t S)\in\varphi_{w-z_1,\ldots,w-z_n,\chi}(\A).$$
So the subalgebra $\CA_\chi(z_1, \dots, z_n)$ does not depend on $w$ and we have $\A_{\chi}(z_1 + a, \dots, z_n + a) = \CA_\chi(z_1, \dots, z_n)$.
On the other hand, the universal Gaudin algebra $ \CA $ is invariant under the automorphism of $ U(\hat{\fg}_-) $ induced by $ t \mapsto st $.  When we follow this automorphism through the evaluation maps, we get the desired result.
\end{proof}

Next, we consider what happens to these algebras under the action of the symmetric group $ S_n $ on $ U(\fg)^{\otimes n} $ and the diagonal adjoint action of $ G $ on $ U(\fg)^{\otimes n} $ (where $ G $ is the simply-connected semisimple group whose Lie algebra is $ \fg$).  The following result is obvious from the definition.
\begin{lem} \label{le:ConjugateA}
Let $ \chi \in \fg, (z_1, \dots, z_n) \in \C^n $.
\begin{enumerate}
\item For all $ \sigma \in S_n $, we have
$$ \sigma(\CA_\chi(z_1, \dots, z_n)) = \CA_\chi(z_{\sigma(1)}, \dots, z_{\sigma(n)})$$
\item For all $ g \in G $, we have
$$ g(\CA_\chi(z_1, \dots, z_n)) = \CA_{Ad_g(\chi)}(z_1, \dots, z_n) $$
\end{enumerate}
\end{lem}

We will be particularly interested in the case when $ \chi $ lies in the Cartan subalgebra $ \fh $.  In this case, $ \CA_\chi(\uz) $ commutes with $\Delta(\fh) $ and thus acts on weight spaces in representations.  More generally $ \CA_\chi(\uz) $ commutes with $ \Delta(\fz_{\fg}(\chi)) $ hence acts on multiplicity spaces of the decomposition with respect to $\fz_{\fg}(\chi)$.

\subsection{Generators of Gaudin algebras}\label{ss:Generators}

Consider the generators $S_l\in\A$ such that $\ol{S_l}=i_{-1}(\Phi_l)$ defined in Proposition~\ref{pr:S_l} (clearly they have degree $d_l+1$ with respect to the PBW filtration).
Consider the following rational section of $U(\fg)^{\otimes n}\otimes\CO(-2(d_l+1))$ on $\BP^1$:
$$S_l(w;\uz;\chi):=\varphi_{w-z_1,\ldots,w-z_n,\chi}(S_l)(dw)^{d_l+1},$$ $l=1,\ldots,r$. Clearly $S_l(w;\uz;\chi)$ has degree $d_l+1$ poles at the points $z_1,\ldots,z_n,\infty$.
By definition, $\A_\chi(\uz)$ is generated by $\varphi_{w-z_1,\ldots,w-z_n,\chi}(S_l)$ for all $w\in\BC\backslash\{z_1,\ldots,z_n\}$. Hence the subalgebra $\A_\chi(\uz)$ is generated by the coefficients of the principal parts of $S_l(w;\uz;\chi)$ at any $n$ of these $n+1$ singular points of $S_l(w;\uz;\chi)$. Let $$S_l(w;\uz;\chi)=\sum\limits_{k=0}^{d_l}s_{l,k}^i(\uz,\chi)(w-z_i)^{-k-1}(dw)^{d_l+1}+holo,\ i=1,\ldots n,$$ and $$S_l(w;\uz;\chi)=\sum\limits_{k=0}^{d_l}s^\infty_{l,k}(\uz,\chi)w^{k+1}(dw^{-1})^{d_l+1}+holo$$ be the expansions of $S_l(w;\uz;\chi)$ at $ w= z_i$ and $w = \infty$, respectively. The following is well-known (cf. \cite{FFTL} and \cite{FMTV}).

\begin{prop} \label{pr:Hamiltonians}
Suppose now that $ \chi \in \fh^{reg} $ and $ z_1, \dots, z_n $ are distinct.  In this case, $ \CA_\chi(\uz) $ contains \begin{enumerate} \item \emph{KZ Gaudin Hamiltonians}
$$
H_i:=\sum\limits_{j\ne i} \frac{\Omega^{(ij)}}{z_i-z_j}+\chi^{(i)},
$$
for all $i=1,\ldots,n$;
\item the \emph{diagonal Cartan subalgebra} spanned by $\Delta(h)$ for all $h\in\fh$;
\item the \emph{dynamical Gaudin Hamiltonians}:
$$
G_h:=\sum\limits_{i=1}^nz_i h^{(i)}+\sum\limits_{\alpha\in\Delta_+} \frac{\langle h,\alpha\rangle}{\langle \chi,\alpha\rangle} \Delta(e_\alpha)\Delta(f_\alpha)
$$
for all $h\in\fh$.
\end{enumerate}
\end{prop}

\begin{proof} The elements $H_i$ are $s_{l,0}^i(\uz,\chi)$, where $S_l$ is the quadratic Casimir generator. The elements $s^\infty_{l,d_l-1}(\uz,\chi)$, for all generators $S_l$ span the diagonal Cartan subalgebra. Finally, the space spanned by $G_h$ and $\Delta(h)$, for all $h\in\fh$, is the span of $s^\infty_{l,d_l-2}(\uz,\chi)$ and $s^\infty_{l,d_l-1}(\uz,\chi)$ for all generators $S_l$.
\end{proof}

\subsection{Shift of argument subalgebras}\label{se:Achi}
Now let us suppose that we take $ n = 1$, so we have the algebra $ \CA_\chi(z) \subset U(\fg) $.  In this case, the algebra does not depend on the value of $ z$ and so we can just write $ \CA_\chi $. The algebraically independent generators are $s^\infty_{l,k}(\chi):=s^\infty_{l,k}(0,\chi)$ for all $S_l$ and $k=0,\ldots,d_l$. This algebra is known as the \emph{quantum shift of argument algebra}, since it quantizes the shift of argument algebras in $ S(\fg) $ defined by Mishchenko-Fomenko (see \cite{R}), which are (generically) maximal Poisson commutative subalgebras in $S(\fg)$. The shift of argument subalgebra $\ol{\A_\chi}:=\gr \A_\chi\subset S(\fg)$ can be described as the subalgebra generated by all the derivatives along $\chi$ of all adjoint invariants in $S(\fg)$. More precisely, it is generated by all elements of the form $\partial_\chi^k\Phi_l= \text{gr}\ s^\infty_{l,k}(\chi)$, for all generators $\Phi_l\in S(\fg)^\fg$, $l=1,\ldots r,\ k=0,\ldots,d_l$, where $d_l=\deg\Phi_l-1$ are the exponents of the Lie algebra $\fg$. Then the number of generators is $\sum\limits_{l=1}^{r}(d_l+1)=\frac{1}{2}(\dim \fg+\rk\fg)$, which is the maximal possible transcendence degree for Poisson commutative subalgebras in $S(\fg)$.

The subalgebra $\ol{\A_\chi}$ commutes with the centralizer of $\chi$ (i.e. we have $\ol{\A_\chi}\subset S(\fg)^{\fz_\fg(\chi)}$). It is natural to expect that it is a \emph{maximal} Poisson commutative subalgebra $S(\fg)^{\fz_\fg(\chi)}$ of maximal possible transcendence degree, but it is not yet proved in full generality.   In this section, we will establish this result when $ \chi \in \fh $.  We will use the following properties of shift of argument subalgebras, see \cite{Sh,Tar} for details.

\begin{prop}\label{pr:ShuvalovTarasov} \cite{Sh,Tar}
\begin{enumerate} \item For any regular $ \chi$, the subalgebra $\ol{\A_\chi}$ is a free polynomial algebra with $\frac{1}{2}(\dim \fg+\rk\fg)$ generators, i.e. all the elements $\partial_\chi^k\Phi_l,\ l=1,\ldots r,\ k=0,\ldots,d_l$ are algebraically independent generators.
\item For  $\chi\in\fh^{reg}$, the restriction of $\ol{\A_\chi}$ to the affine subspace $e+\fb_-$ is an isomorphism $\A_\chi\simeq\CO(e+\fb_-)$.
\item For   $\chi\in\fh^{reg}$, the subalgebra $\ol{\A_\chi}$ is a maximal Poisson commutative subalgebra in $S(\fg)$.
\end{enumerate}
\end{prop}

Moreover, in \cite{Tar} a more general result is proved. One can consider $\chi_0\in\fh$ not necessarily regular. Then the subalgebra $\ol{\A^0_{\chi_0}}$ is smaller but one can make it larger in the following way. This subalgebra commutes with the centralizer subalgebra $\fz_\fg(\chi_0)$, so given a regular Cartan element $\chi_1$ from $\fz_\fg(\chi_0)'$ we can construct a bigger commutative subalgebra generated by $\ol{\A^0_{\chi_0}}\subset U(\fg)$ and $\A_{\chi_1}\subset U(\fz_\fg(\chi_0))$. Denote this subalgebra by $\ol{\A_{(\chi_0,\chi_1)}}$.

\begin{prop} \label{pr:ShuvalovTarasov2} \cite{Sh,Tar}
\begin{enumerate}
\item The subalgebras $\ol{\A_{\chi_0}^0}\subset S(\fg)$ and $\ol{\A_{\chi_1}}\subset S(\fz_\fg(\chi_0))$ both contain $S(\fz_\fg(\chi_0))^{\fz_\fg(\chi_0)}$.
\item The subalgebra $\ol{\A_{(\chi_0,\chi_1)}}$ for regular $\chi_1$ is a free polynomial algebra with $\frac{1}{2}(\dim \fg+\rk\fg)$ generators. The set of generators is the union of standard generators of $\ol{\A_{\chi_1}}$ and $\partial_{\chi_0}^k\Phi_l,\ k\le d'_l$ for some $d'_l\le d_l$.
\item Suppose that the centralizer subalgebra $\fz_\fg(\chi_0)$ is generated by $\fh$ and some subset of the $e_{\alpha_i}, f_{\alpha_i}$( i.e. of root subspaces corresponding to simple roots). Then for regular $\chi_1$ the restriction of $\ol{\A_{(\chi_0,\chi_1)}}$ to the affine subspace $e+\fb_-$ is an isomorphism $\ol{\A_{(\chi_0,\chi_1)}}\simeq\CO(e+\fb_-)$.
\item For regular  $\chi_1$ the subalgebra $\ol{\A_{(\chi_0,\chi_1)}}$ is a maximal Poisson commutative subalgebra in $S(\fg)$.
\end{enumerate}
\end{prop}

We have the following important consequence of this result.

\begin{cor} \label{co:ShuvalovTarasov} For any $\chi_0\in\fh$ the subalgebra $\ol{\A_{\chi_0}^0}$ is a maximal Poisson commutative subalgebra in $S(\fg)^{\fz_\fg(\chi_0)}$ of maximal possible transcendence degree, equal to $$\frac{1}{2}(\dim \fg + \rk \fg - (\dim \fz_\fg(\chi_0)' - \rk \fz_\fg(\chi_0)')).$$
It is a free polynomial algebra generated by $\partial_{\chi_0}^k\Phi_l,\ k\le d'_l$ for some $d'_l\le d_l$.
\end{cor}

When $ \chi \in \fh^{reg}$ (and $ n= 1$), then both algebras $ \CA_\chi $ and $\ol{\A_\chi}$ contain the quadratic elements
$$
G_h:= \sum_{\alpha \in \Delta_+} \frac{\langle h,\alpha\rangle}{\langle \chi,\alpha\rangle} e_\alpha f_\alpha
$$
for all $h \in \fh $ (this is a special case of \propref{pr:Hamiltonians}) and we let $ Q_\chi $ denote the span of these elements in $U(\fg)$ (resp. $\ol{Q_\chi}$ in $S(\fg)$).  So $ Q_\chi $ is a copy of $ \fh $ sitting inside $\CA_\chi $.

\begin{prop}\label{pr:Q-chi} \cite{Ryb05} For generic $\chi\in\fh$ (i.e. outside some countable union of proper Zariski closed subvarieties) the subalgebra $\ol{\A_\chi}\subset S(\fg)$ is the Poisson centralizer of the subspace $ \ol{Q_\chi} $ and the subalgebra $\A_\chi\subset U(\fg)$ is the centralizer of the subspace $ Q_\chi \subset U(\fg)$.
\end{prop}

\subsection{Maximality of inhomogeneous Gaudin subalgebras}

\begin{prop}\label{pr:Maximal} For any semisimple $\chi$ and any $\uz=(z_1,\ldots,z_n)$ with $z_i\ne z_j$, the subalgebra $\A_\chi^0(\uz)$ is a maximal commutative subalgebra in $(U(\fg)^{\otimes n})^{\fz_\fg(\chi)}$. It is a free polynomial algebra in $$\frac{1}{2}(n \dim\fg+n\rk\fg - (\dim\fz_\fg(\chi)'+\rk\fz_\fg(\chi)'))$$ generators.
\end{prop}

\begin{proof} We just adapt Tarasov's argument for $\A_\chi$ \cite{Tar} to our situation. First, it is sufficient to prove that the associated graded subalgebra $\ol{\A_\chi^0(\uz)}$ is a maximal Poisson commutative subalgebra in $(S(\fg)^{\otimes n})^{\fz_\fg(\chi)}$. This subalgebra has maximal possible transcendence degree for Poisson commutative subalgebras; hence if it is not maximal then there is an algebraic extension of $\ol{\A_\chi^0(\uz)}$ in $S(\fg)^{\otimes n}$. So it is sufficient to show that $\ol{\A_\chi^0(\uz)}$ is algebraically closed in $(S(\fg)^{\otimes n})^{\fz_\fg(\chi)}$.

\begin{lem} The associated graded subalgebra $\ol{\A_\chi^0(\uz)}$ contains the product of $\ol{\A(\uz)}$ and $\Delta(\ol{\A_\chi^0})$.
\end{lem}

\begin{proof} Indeed the leading terms of the generators of $\A_\chi^0(\uz)$ corresponding to the expansion at $t=z_i$ generate $\ol{\A(\uz)}$ and the leading terms of the generators of $\A_\chi^0(\uz)$ corresponding to the expansion at $t=\infty$ generate $\Delta(\ol{\A_\chi^0})$.
\end{proof}

The subalgebras $\ol{\A(\uz)}$ and $\Delta(\ol{\A_\chi^0})$ both contain $\Delta(S(\fg)^\fg)$. According to Theorem~\ref{th:Knop}, the Poisson centralizer of $\Delta(S(\fg)^\fg)$ is the tensor product of free $\Delta(S(\fg)^\fg)$-modules: $(S(\fg)^{\otimes n})^\fg\otimes_{\Delta(S(\fg)^\fg)}\Delta(S(\fg))$. Hence everything commuting with $\ol{\A_\chi^0(\uz)}$ lies in $(S(\fg)^{\otimes n})^\fg\otimes_{\Delta(S(\fg)^\fg)}\Delta(S(\fg))$.

On the other hand, $\Delta(\ol{\A_\chi^0})$ is known to be maximal Poisson commutative in $\Delta(S(\fg)^{\fz_\fg(\chi)})$ according to Proposition~\ref{pr:ShuvalovTarasov}. So it is sufficient to show that $\ol{\A(\uz)}$ is maximal Poisson commutative (equivalently, of the right transcendence degree and algebraically closed) in $(S(\fg)^{\otimes n})^\fg$. The rank of the Poisson structure on $(\fg^*)^{\oplus n} \sslash G$ is $(n-1)\dim\fg+(n+1)\rk\fg$, hence the maximal possible transcendence degree for a Poisson commutative subalgebra in $(S(\fg)^{\otimes n})^\fg$ is $\frac{1}{2}((n-1)\dim\fg+(n+1)\rk\fg))$.

Consider the filtration on $(S(\fg)^{\otimes n})^\fg$ determined by the grading on the last tensor factor. It is sufficient to show that the associated graded of $\ol{\A(\uz)}$ has the desired transcendence degree and is algebraically closed.

\begin{lem} The associated graded of $\ol{\A(\uz)}$ with respect to this filtration contains the product of $\ol{\A(1,0)}^{(i,n)}$ for $i=1,\ldots,n-1$.
\end{lem}

\begin{proof} Indeed, the leading terms of the coefficients of $\ol{S_l}(w;\uz)$ at $z_i$ are the generators of $\ol{\A(1,0)}^{(i,n)}$.
\end{proof}

According to Lemma above, it is sufficient to show that $\prod\limits_{i=1}^{n-1}\ol{\A(1,0)}^{(i,n)}$ has the maximal possible transcendence degree for Poisson commutative subalgebra in $(S(\fg)^{\otimes n})^\fg$ and is algebraically closed in $S(\fg)^{\otimes n}$. Consider the Kostant slice $\fg_{can}:=e+\fz_\fg(f)\subset\fg=\fg^*$ embedded into the last summand of $ \fg^{\oplus n}$.  We obtain a restriction map
$$ S(\fg)^{\otimes n} = \CO(\fg^{\oplus n}) \rightarrow \CO(\fg^{\oplus n-1} \times \fg_{can}) = S(\fg)^{\otimes (n-1)} \otimes \CO(\fg_{can}).
$$

\begin{lem} This restriction isomorphically takes the subalgebra $\prod\limits_{i=1}^{n-1}\ol{\A(1,0)}^{(i,n)}\subset S(\fg)^{\otimes n}$ to $B\otimes \CO(\fg_{can})\subset S(\fg)^{\otimes (n-1)}\otimes\CO(\fg_{can})$, where $B$ is a polynomial algebra in $\frac{n-1}{2}(\dim\fg+\rk\fg)$ generators. The evaluation of $\CO(\fg_{can})$ at any $\chi\in\fg_{can}$ isomorphically takes $B$ to $\ol{\A_\chi}^{\otimes (n-1)}$.
\end{lem}

\begin{proof}
The subalgebra $\prod\limits_{i=1}^{n-1}\ol{\A(1,0)}^{(i,n)}\subset S(\fg)^{\otimes n}$ is generated by the central elements of the last tensor factor and another $\frac{n-1}{2}(\dim\fg+\rk\fg)$ generators. The restrictions of the central elements generate $\CO(\fg_{can})$. On the other hand, since every $\chi\in\fg_{can}$ is regular, the restriction of the last factor to any $\chi\in\fg_{can}$ takes $\prod\limits_{i=1}^{n-1}\ol{\A(1,0)}^{(i,n)}$ to $\ol{\A_\chi}^{\otimes (n-1)}$, thus taking the rest of the generators to algebraically independent elements.
\end{proof}

Note that the transcendence degree of $B\otimes \CO(\fg_{can})$ is maximal, i.e. $\frac{n-1}{2}\dim\fg+\frac{n+1}{2}\rk\fg$. So it is sufficient to prove that $B\otimes \CO(\fg_{can})$ is algebraically closed in $S(\fg)^{\otimes (n-1)}\otimes\CO(\fg_{can})$. According to Tarasov \cite{Tar}, $\ol{\A_\chi}^{\otimes (n-1)}$ is algebraically closed for regular semisimple $\chi$, so $B\otimes \CO(\fg_{can})$ is algebraically closed in $S(\fg)^{\otimes (n-1)}\otimes\CO(\fg_{can})$ after localization by $\CO(\fg_{can})$. Hence for any element $f$ in the algebraic closure of $B\otimes \CO(\fg_{can})$ in $S(\fg)^{\otimes (n-1)}\otimes\CO(\fg_{can})$ there exists $g\in \CO(\fg_{can})$ such that $fg\in B\otimes \CO(\fg_{can})$. But this means that either $f$ is already in $B\otimes \CO(\fg_{can})$ (and we are done) or we can assume $g$ irreducible and $fg$ is not divisible by $g$ in $B\otimes \CO(\fg_{can})$. In the latter case there exists $\chi\in\fg_{can}$ such that $g(\chi)=0$ and $fg(\chi)\ne0$, hence a contradiction.
\end{proof}

\begin{cor}\label{co:Center} $\A_\chi^0(\uz)$ contains the center of $\Delta(U(\fz_\fg(\chi)))$ for all $\uz$.
\end{cor}

\begin{proof} Since $\A_\chi^0(\uz)$ is  $\fz_\fg(\chi)$-invariant, any element of the center of $\Delta(U(\fz_\fg(\chi)))$ commutes with any element of $\A_\chi^0(\uz)$. On the other hand the center of $\Delta(U(\fz_\fg(\chi)))$ is contained in $(U(\fg)^{\otimes n})^{\fz_\fg(\chi)}$. Hence from Proposition~\ref{pr:Maximal} we have the assertion.
\end{proof}

\section{Compactification of some families}\label{sect-MF}

\subsection{Holonomy Lie algebra}\label{sect-holonomy}
Our goal now is to describe the closure of the locus of shift of argument algebras inside $ U(\fg) $ in a similar way as we described the closure of the locus of Gaudin algebras in $ (U(\fg)^{\otimes n})^\fg $.  We begin by reviewing some results of de Concini-Procesi \cite{DCP2} and Aguirre-Felder-Veselov \cite{AFV2}.

\begin{defn}
The \emph{holonomy Lie algebra} $\ft_{\Delta}$ associated to the root system $\Delta $ is generated by $\{t_\alpha \ |\ \alpha\in\Delta_+\}$ subject to the following relations: for each subspace $S\subset \fh^*$ of dimension $2$ and each $ \alpha \in \Delta_+ \cap S $, we have
\begin{equation}\sum_{\beta \in\Delta_+ \cap S} [t_\alpha,t_\beta]=0.
\end{equation}
\end{defn}

The Lie algebra $ \ft_\Delta $ is graded and we write $ \ft_\Delta^{(1)} $ for the first graded piece (which has a basis given by the set of $ t_\alpha $).

Fix $ \chi \in \fh^{reg} $.  For $ h \in \fh $ let
$$
G_h:= \sum_{\alpha \in \Delta_+} \frac{\langle h,\alpha\rangle}{\langle \chi,\alpha\rangle} t_\alpha
$$

Let $ Q_\chi $ be the span of the elements $ G_h $ for $ h \in \fh $.  It is an abelian Lie subalgebra of $ \ft_\Delta $.

We have a morphism $ Q : \fh^{reg} \to \Gr(r,\ft^{(1)}_\Delta)$ given by $ \chi \mapsto Q_\chi $.

The following result was proven in \cite{AFV2} and can also be extracted from \cite{DCP2}.

\begin{thm}\label{dc-p} The morphism $ Q $ extends to an inclusion $ \CM_\Delta \rightarrow \Gr(r, \ft_\Delta^{(1)}) $.
\end{thm}

\begin{proof} The variety $\CM_\Delta$ is covered by the charts $U_{\CS}^b$. So to prove the assertion we need, for all $\CS,b$, a basis of the subspace $Q_\chi$ well-defined for all $ \chi \in U_{\CS}^b$. This was constructed by De Concini and Procesi in \cite{DCP2}, see Theorem on page~12.
\end{proof}

\subsection{Commuting quadratic elements.}
The Lie algebra $ \ft_\Delta $ is relevant for our situation because of the following observation.

\begin{prop} There is a homomorphism $\ft_{\Delta}\to U(\fg)$ which maps $t_{\alpha}$ to $e_\alpha f_\alpha$.
\end{prop}

\begin{proof}
Note that for any $2$-dimensional subspace $S\subset\fh^*$ the element $\sum\limits_{\alpha\in S} e_\alpha f_\alpha$ is (up to a quadratic term from $U(\fh)$) the Casimir element of the Levi subalgebra generated by the root subsystem $\Delta\cap S$. Hence it commutes with any $e_\alpha f_\alpha$ with $\alpha\in S$ and we are done.
\end{proof}

\subsection{Maximal commutative subalgebras.}

As described above, for each $ \chi \in \fh^{reg} $, we have a maximal commutative subalgebra $ \CA_\chi \subset U(\fg) $ which contains $ Q_\chi $.  We write $ \overline{\CA_\chi} \subset S(\fg) $ for the associated graded of $ \CA_\chi $.  It is a Poisson commutative subalgebra of $ S(\fg) $.

\begin{thm} \label{th:dCPFamily} The map $ \chi \mapsto \ol{\CA_\chi} $ extends to a family of subalgebras of $ S(\fg) $ parametrized by the de Concini-Procesi space $ \CM_\Delta $.
\end{thm}

\begin{proof}
Let $\CM_{Poisson}$ be the closure in question. There is an obvious map $\CM_{Poisson}\to\CM_{\Delta}$ which takes the quadratic component of the corresponding subalgebra. Due to the results of Shuvalov \cite{Sh}, this map is bijective on $\BC$-points. Since $\CM_{\Delta}$ is smooth, this is an isomorphism.
\end{proof}

\begin{thm} The closure of the space of commutative subalgebras $\A_\chi\subset U(\fg)$ is isomorphic to $\CM_{\Delta}$.
\end{thm}

\begin{proof}
Let $U_S$ be the chart on $\CM_{\Delta}$ corresponding to the nested set $S$. We can assume without loss of generality that $S$ is compatible with the dominant Weyl chamber. The idea is to construct a set of generators $P_{l,i}(\chi)\in\A_\chi$ which is well-defined on $U_S$ outside some codimension $2$ subvariety. The generators $P_{l,i}(\chi)$ are thus well-defined and algebraically independent for all $\chi\in U_S$.
Let $U_S^{(k)}$ be the union of codimension $\ge k$ strata in $U_S$.

\begin{lem}\label{le:AchiCodim1} For each $N \in \BN $, the map $\fh_{reg}\to \Gr(d_N, F^N U(\fg))$ extends to $U_S\backslash \Theta$ where $\Theta$ is a proper closed subvariety of $U_S^{(1)}$. The subalgebra corresponding to a point $(\chi_0,\chi_1)\in U_S^{(1)}\backslash \Theta$ is generated by $\A_{\chi_0}^0\subset U(\fg)$ and $\A_{\chi_1}\subset U(\fz_\fg(\chi_0))$. In particular, $\gr\A_{(\chi_0,\chi_1)}=\ol{\A_{(\chi_0,\chi_1)}}$.
\end{lem}

\begin{proof}
Let $\chi(t)$ be such that $\chi(0)=\chi_0$. Clearly the subalgebra $\lim\limits_{t\to0}\A_{\chi(t)}$ contains $\A_{\chi_0}^0$. Since $\A_{\chi_0}^0$ is maximal commutative subalgebra in $U(\fg)^{\fz_\fg(\chi_0)}$ it contains the center $Z$ of $U(\fz_\fg(\chi_0))$. By the Theorem~\ref{th:Knop}, the centralizer of $Z$ in $U(\fg)$ is the tensor product of $U(\fg)^{\fz_\fg(\chi_0)}$ and $U(\fz_\fg(\chi_0))$ over $Z$. This means that the centralizer of $\A_{\chi_0}^0$ in $U(\fg)$ is $\A_{\chi_0}^0\cdot U(\fz_\fg(\chi_0))$, and hence $\lim\limits_{t\to0}\A_{\chi(t)}\subset \A_{\chi_0}^0\cdot U(\fz_\fg(\chi_0))$. On the other hand, the subalgebra $\lim\limits_{t\to0}\A_{\chi(t)}$ contains the quadratic component $Q_{\chi_1}$ of $\A_{\chi_1}\subset U(\fz_\fg(\chi_0))$ since it is so for the associated graded algebra and since the quadratic component of $\A_\chi$ is the symmetrization of its associated graded. For generic $\chi_1$ (i.e. outside some closed proper $\Theta\subset U_S^{(1)}$) the centralizer of $Q_{\chi_1}$ in $U(\fz_\fg(\chi_0))$ is $\A_{\chi_1}$ hence $\lim\limits_{t\to0}\A_{\chi(t)}\subset \A_{\chi_0}^0\cdot \A_{\chi_1}$ and we are done.
\end{proof}

Let $e=\sum e_{\alpha_i}\in\fn$ be the principal nilpotent element. Let $\xi:\fn_-\to\BC$ be the character given by the scalar product with $e$, i.e. taking $e_{-\alpha_i}$ to $1$. Consider the left ideal $J_\xi$ in $U(\fg)$ generated by $x-\xi(x),\ x\in\fn_-$. By PBW theorem, the quotient $U(\fg)/J_\xi$ is isomorphic to $U(\fb)$ as a vector space, so we have a ``Whittaker'' projection $\Wh:U(\fg)\to U(\fb)$.

\begin{lem} For any $\chi\in U_S\backslash \Theta$, the map $\Wh:\A_\chi\to U(\fb)$ is an isomorphism of vectors spaces. Moreover the preimages of $h_i$ and $e_\alpha$ form a set of algebraically independent generators of $\A_\chi$.
\end{lem}

\begin{proof}
Let $e,h,f$ be the $\fsl_2$-triple containing $e$ such that $h\in\fh$. Consider the grading on the algebras $U(\fg)$ and $S(\fg)$ determined by $\frac{1}{2}\ad h$. We denote by $p(a)$ the degree of $a\in U(\fg)$ with respect to this grading. Define a filtration on $U(\fg)$ as $\deg+p$ where $\deg$ is the PBW filtration. Clearly the associated graded with respect to this filtration is still $S(\fg)$ with the grading determined by $\deg+p$. Next, on $\A_\chi$ the filtration is the same as PBW since $\A_\chi\subset U(\fg)^{\fh}$. The advantage of this new filtration on $U(\fg)$ and grading on $S(\fg)$ is that the associated graded of the map $\Wh$ is the homomorphism $S(\fg)\to S(\fb)$ of restriction to $e+\fb_-$ (which now becomes homogeneous).

Now it is sufficient to check the assertion of the Lemma for associated graded algebras with respect to the filtration. But this is true according to Proposition~\ref{pr:ShuvalovTarasov2} part (3).
\end{proof}

From the last Lemma, it follows that the elements $\Wh^{-1}(h_i)$ and $\Wh^{-1}(e_\alpha)$ are well-defined on $U_S$ and we have $\gr \Wh^{-1}(h_i)=\ol{\Wh}^{-1}(h_i)$ and $\gr \Wh^{-1}(e_\alpha)=\ol{\Wh}^{-1}(e_\alpha)$ for all $\chi\in U_S$. Indeed, the subset of $U_S$ where it does not hold is a complement to a divisor.

\end{proof}

\subsection{Operadic nature of shift of argument algebras}
Assume that $ \fg $ is simple.

Let $ \chi_0 \in \fh, \chi_0 \ne 0 $ and let $ \fg_1 = \fz_\fg(\chi_0)' $ be the derived subalgebra of the centralizer.  If $ \chi_0 $ is not regular, then $\CA_{\chi_0}^0 $ will not be maximal commutative.  Let $ \chi \in \pi^{-1}(\chi_0) $, where $ \pi : \CM_\Delta \rightarrow \PP(\fh) $.  Then by \lemref{le:FibresCM}, we can identify $ \pi^{-1}(\chi_0) = \CM_{\Delta_1} $, where $ \Delta_1 = \{ \alpha \in \Delta : \langle \alpha, \chi_0 \rangle = 0 \} $ is the root system of $ \fg_1 $.  Choose a point $ \chi_1 \in \CM_{\Delta_1} $ and regard the pair $ \chi = (\chi_0, \chi_1)$ as a point in $ \CM_\Delta$. Consider the subalgebras $ \CA_{\chi_0}^0 \subset U(\fg) $ and $ \CA_{\chi_1} \subset U(\fg_1) $ (here we make use of the inclusion $ \fg_1 \hookrightarrow \fg $). From Proposition~\ref{pr:ShuvalovTarasov2} we have the following.

\begin{cor}
 With the above notation, the algebra $ \CA_\chi $ is generated by $ \CA_{\chi_0}^0 $ and $ \CA_{\chi_1} $. Moreover we have  $ \CA_\chi = \CA_{\chi_0}^0 \otimes_{ZU(\fg_1)} \CA_{\chi_1} $.
\end{cor}

We can exploit this operadic nature as follows.

Now, let $ \chi_0, \chi_1 $ be as above, with the additional assumption that they are both real.  Also, recall that the algebra $ \A_{\chi_0}^0 $ acts on $ V(\l)^\nu := \Hom_{\fg_1}(V(\nu), V(\l)) $ where $ V(\nu) $ is the irreducible representation of $ \fg_1 $ of highest weight $\nu $ and $ V(\l)$ is the irreducible representation of $\fg$ of highest weight $\l$.

So consider a irreducible representation $ V(\lambda) $ of $\fg $ and decompose it under the action of $ \fg_1 $.  We have
$$
V(\l) = \bigoplus_{\nu \in \Lambda_+(\fg_1)} V(\l)^\nu \otimes V(\nu)
$$

\begin{cor} \label{co:LinesProduct}
The action of $ \CA_\chi $ on $ V(\l) $ is compatible with this decomposition and we have a natural bijection
$$
\CE_{\A_\chi}(V(\l)) = \sqcu_{\nu \in \Lambda_+(\fg_1)} \CE_{\A_{\chi_0}^0}(V(\l)^\nu) \times \CE_{\A_{\chi_1}}(V(\nu))
$$
\end{cor}

\subsection{Partial description of the compactification}\label{subsect-PieceClosure}

The closure of the parameter space for the family of subalgebras ${\A}_\chi(z_1,\ldots,z_n)\subset U(\fg)^{\otimes n}$ is unknown in general. We will need the following partial results which describe small pieces of this closure.

Let $\Phi_l$, $l=1,\ldots,r$ be the generators of $S(\fg)^\fg$ and $\chi\in\fh$ be any Cartan element. According to Shuvalov \cite{Sh}, there exist $d'_l\le d_l$ such that $\partial_\chi^k\Phi_j$ for $k\le d'_l$ are algebraically independent and generate the classical shift of argument subalgebra $\ol{\A_\chi}$.
Consider the subalgebra $\A^0_\chi(z,0)$ for $ z \in \Cx$. Recall the generators of the inhomogeneous Gaudin algebra from subsection~\ref{ss:Generators} in this case. We have the rational section
$$S_l(w;z,0;\chi):= \varphi_{w-z, w, \chi} (S_l)(dw)^{d_l+1}$$
of $U(\fg)\otimes U(\fg)\otimes\CO(-2(d_l+1))$ on $\BP^1$ for $l=1,\ldots,r$.

Let $$S_l(w;z,0;\chi)=\sum\limits_{k=0}^{d_l}s_{l,k}^0(z,\chi)w^{-k-1}(dw)^{d_l+1}+holo,$$ $$S_l(w;z,0;\chi)=\sum\limits_{k=0}^{d_l}s_{l,k}^z(z,\chi)(w-z)^{-k-1}(dw)^{d_l+1}+holo$$ and $$S_l(w;z,0;\chi)=\sum\limits_{k=0}^{d_l}s^\infty_{l,k}(z,\chi)w^{k+1}(dw^{-1})^{d_l+1}+holo$$ be the expansions of $S_l(w;z,0;\chi)$ at $0$, $z$ and $\infty$, respectively.

\begin{lem}\label{lem:2PointsLimits} Then \begin{enumerate}
\item $z^{d_l-k}s_{l,k}^0(z,\chi)$ with $0\le k\le d_l$ and $s^\infty_{l,k}(z,\chi)$ with $0\le k\le d'_l$ are regular in $z$ and algebraically independent for all $z\in\BC$;
\item $s_{l,k}^0(z,\chi)$ with $d_l-d'_l\le k\le d_l$ and $s^z_{l,k}(z,\chi)$ with $d_l-d'_l\le k\le d_l$ are regular in $z^{-1}$ and algebraically independent for all $z$ in some neighborhood of $\infty$.
\end{enumerate}
\end{lem}

\begin{proof}
The generators $z^{d_l-k}s_{l,k}^0(z,\chi)$ are the coefficients of $S_l(zw;z,0;\chi)$. We have $S_l(zw;z,0;\chi)=\varphi_{w-1,w,z \chi}(S_l))(dw)^{d_l+1}$, hence the assertion on regularity in $z$. On the other hand, $s_{l,k}^0(z,\chi)$ and $s^z_{l,k}(z,\chi)$ are the coefficients of the expansions of $S_l(w;z,0;\chi)= \varphi_{w-z,w,z\chi}(S_l)(dw)^{d_l+1}$ at $0$ and $z$, hence regular in $z^{-1}$.

The images of $z^{d_l-k}s_{l,k}^0(z,\chi)$ with $0\le k\le d_l$ are the generators of $\ol{\A(1,0)}$. Indeed the leading term of these coefficients w.r.t. PBW grading is the same as that of $z^{d_l-k}s_{l,k}^0(z,0)$ since the terms containing $\chi$ have lower PBW degree. Similarly, the associated graded of the elements $s^\infty_{l,k}(z,\chi)$ with $0\le k\le d'_l$ in $\gr U(\fg)\otimes U(\fg)=S(\fg)\otimes S(\fg)$ are $\Delta(\partial_\chi^k\Phi_l)$, i.e. the generators of $\Delta(\ol{A_\chi})$.

Consider the values of $s_{l,k}^0(z,\chi)$, with $0\le k\le d'_l$, and $s^z_{l,k}(z,\chi)$, with $0\le k\le d'_l$ at $z=\infty$. The leading term of $s_{l,k}^z(z,\chi)$ at $z=\infty$ is $\partial_\chi^k\Phi_l^{(1)}$ in $S(\fg)\otimes S(\fg)$. Similarly, the leading term of $s_{l,k}^0(z,\chi)$ at $z=\infty$ is $\partial_\chi^k\Phi_l^{(2)}$. Hence these elements are the algebraically independent generators of $\ol{\A_\chi}\otimes \ol{\A_\chi}$. On the other hand the condition of being algebraically independent is Zariski open in $z$, hence the $s_{l,k}^0(z,\chi)$ with $d_l-d'_l\le k\le d_l$ and $s^z_{l,k}(z,\chi)$, with $d_l-d'_l\le k\le d_l$, are regular in $z^{-1}$ and algebraically independent for all $z$ in some neighborhood of $\infty$.
\end{proof}

\begin{prop} \label{pr:MultiClosure1}
Let $ \chi \in \fh^{reg} $.  The map from $ \BC^\times $ to commutative subalgebras of $ U(\fg) \otimes U(\fg) $ given by $ z \mapsto \A_\chi(z,0) $ extends to a family of subalgebras parametrized by $ \BP^1 $.  The extended map is defined on the boundary as follows:
\begin{enumerate}
\item $ 0 $ goes to the subalgebra generated by $ \Delta(\CA_\chi)$ and $ \CA(1,0) $.
\item $ \infty $ goes to the subalgebra $ \A_\chi \otimes \A_\chi $.
\end{enumerate}
\end{prop}

\begin{proof} Indeed, any rational map from $\BP^1$ to a projective variety $X$ is a regular map $\BP^1\to X$. In particular, each of the maps $\Cx \rightarrow \Gr(d_N, F^N U(\fg)^{\otimes 2})$ extends to $\BP^1$. Hence there is a family parametrized by $\BP^1$ extending  $\A_\chi(z,0)$. So it remains to compute the limits $\lim\limits_{z\to0}\A_\chi(z,0)$ and $\lim\limits_{z\to\infty}\A_\chi(z,0)$.

The first one is the same as $\lim\limits_{z\to0}\A_{z\chi}(1,0)$ hence contains $\A(1,0)$. On the other hand, $\A_\chi(z,0)$ is the image of $\A \subset U(\hat{\fg}_-)$ under the map
$$
\varphi_{w-z,w,\chi}:U(\hat{\fg}_-)\to U(\fg)\otimes U(\fg)\otimes S(\fg) \to U(\fg) \otimes U(\fg)
$$
Hence the limit of $\A_\chi(z,0)$ as $z\to0$ contains the image of $\varphi_{w,w,\chi}$ which is $ \Delta(\CA_\chi)$.

The subalgebra in $U(\fg)\otimes U(\fg)$ generated by $\A(1,0)$ and $ \Delta(\CA_\chi)$ has the same Hilbert series as $\A_\chi(z,0)$: indeed, the product of $\A(1,0)$ and $\Delta(\A_\chi)$ is in fact the tensor product over the $\Delta(U(\fg)^\fg)$, the expansions of $S_l(w;z,0;\chi)$ at $w=\infty$ have the same degrees as the generators of $\A_\chi$  and  the expansions of $S_l(w;z,0;\chi)$ at $w=0$ have the same degrees as the generators of A(1,0), so the product of $\A(1,0)$ and $ \Delta(\CA_\chi)$ has the same number of algebraically independent generators of the same degrees as $\A_\chi(z,0)$. Hence  the product of $\A(1,0)$ and $ \Delta(\CA_\chi)$ coincides with the limit $\lim\limits_{z\to0}\A_\chi(z,0)$.

The limit $\lim\limits_{z\to\infty}\A_\chi(z,0)$ is dealt with similarly, see \cite{R}, Theorem~2.
\end{proof}

\begin{prop} \label{pr:MultiClosure2}
Let $ \alpha \in \Delta_+ $ and let $ \chi_0 \in \fh $ be a generic element of the hyperplane $ \{ \alpha = 0 \}$.
Let $D\subset \BC^\times$ be the (finite) set of such $t\in\BC^\times$ that $\chi_0+th_\alpha\not\in\fh^{reg}$.

The map from $Y:=(\BC^\times\backslash D)\times\BC^\times$ to commutative subalgebras of $U(\fg)\otimes U(\fg)$ which sends $(t,z)\in\BC^\times\times\BC^\times$ to $\A_{\chi_0+th_\alpha}(z,0)$ extends to a map from $X\supset Y$ where $X$ is the blow-up of the point $(0,\infty)\in(\BC\backslash D)\times\BP^1$. The extended map is defined on the boundary as follows:
\begin{enumerate}
\item $(t,0)$ with $t\ne0$ goes to the subalgebra generated by $\A(1,0)$ and $\Delta(\A_{\chi_0+th_\alpha})$.
\item $(t,\infty)$ with $t\ne0$ goes to $\A_{\chi_0+th_\alpha}\otimes\A_{\chi_0+th_\alpha}$.
\item $(0,z)$ with $z\ne0,\infty$ goes to the subalgebra generated by $\A_{\chi_0}^0(z,0)$ and $\Delta(h_\alpha)$. The subalgebra $\A_{\chi_0}^0(z,0)$ contains $\Delta(C_\alpha)$.
\item $(0,0)$ goes to the subalgebra generated by $\A(1,0)$, $\Delta(\A_{\chi_0}^0)$ and $\Delta(h_\alpha)$.
\item The image of a point $t=0, z = \infty, [t:z^{-1}]=[a:b]$ in the special fiber of the blow-up is the subalgebra generated by $\A_{\chi_0}^0\otimes\A_{\chi_0}^0$, $\Delta(h_\alpha)$ and $ah_\alpha^{(1)}+b(\Delta(C_\alpha))$.
\end{enumerate}
Here $ C_\alpha $ denotes the quadratic Casimir element in the root $ \sl_2 $ subalgebra corresponding to $ \alpha $.
\end{prop}

\begin{proof}
The variety $X$ is covered by three charts,
\begin{gather*}
U_0:=\{(t,z)\ : \ t\in\BC\backslash D, z\in\BC\} \\
U_\infty^+:=\{(t,z,(a:b))\ :\ t\in\BC\backslash D, z\in U, az^{-1}=bt, [a:b]\ne[1:0]\} \\
U_\infty^-:=\{(t,z,(a:b))\ :\ t\in\BC\backslash D, z\in U, az^{-1}=bt, [a:b]\ne[0:1]\}
\end{gather*}
where $U$ is a neighborhood of $\infty\in\BP^1$. For each chart we define a set of generators of $\A_{\chi_0+th_\alpha}(z,0)$ which is regular on it and algebraically independent at any point of the corresponding chart.

For $U_0$: Consider the elements $z^{d_l-k}s_{l,k}^0(z,\chi),\ s_{l,k}^\infty(z,\chi)\in\A_\chi(z,0)$ from Lemma~\ref{lem:2PointsLimits} for $\chi=\chi_0+th_\alpha$. These elements are regular in $z$ and $\chi$ hence regular on $U_0$. Note that $d'_l=d_l$ for all $l=1,\ldots,r$ if $\chi$ is regular, and we have $d'_l=d_l$ or $d_l-1$ (in fact for only one $l$) if $\chi=\chi_0$. Hence the elements $z^{d_l-k}s_{l,k}^0(z,\chi)$ with $0\le k\le d_l$ and $s_{l,k}^\infty(z,\chi)$ with $0\le k\le d_l-1$ are algebraically independent for all $\chi=\chi_0+th_\alpha$. On the other hand, the elements $s_{l,d_l}^\infty(z,\chi)$ span $\Delta(\fh)$ for any regular $\chi$. Hence one can take $z^{d_l-k}s_{l,k}^0(z,\chi)$ with $0\le k\le d_l$, $s_{l,k}^\infty(z,\chi)$ with $0\le k\le d_l-1$ and some basis of $\Delta(\fh)$ as the desired set of generators on $U_0$.

For $U_\infty^\pm$: From Lemma~\ref{lem:2PointsLimits}, we have algebraically independent elements $s_{l,k}^0(z,\chi)$ with $d_l-d'_l\le k\le d'_l$ and $s^z_{l,k}(z,\chi)$ with $d_l-d'_l\le k\le d'_l$ which are regular on $U_\infty^+\cup U_\infty^-$. For $t=0,\ z=\infty$ these elements generate $\A_{\chi_0}^0\otimes\A_{\chi_0}^0$ hence one can complete this set of generators with $\Delta(h_\alpha)$ and $G_{h_\alpha}(z, \chi)$ to get an algebraically independent set in a punctured neighborhood of $z=\infty$: indeed, the limits of $\Delta(h_\alpha)$ and $G_{h_\alpha}(z, \chi)$ as $t\to0$ and $z=at^{-1}\to \infty$, $a\ne0$ are algebraically independent with the generators of  $\A_{\chi_0}^0\otimes\A_{\chi_0}^0$.

Let us show that we get the desired subalgebras at the special points in $X$. Note that all of them are maximal commutative subalgebras in $U(\fg)\otimes U(\fg)$ so we we have to check that the specializations at these point contain the corresponding subalgebras. 
\begin{enumerate}
    \item For $(t,0)\in U_0$ with $t\ne0$ the corresponding subalgebra contains $\A(1,0)$ (generated by $\lim\limits_{z\to0}z^{d_l-k}s_{l,k}^0(z,\chi)=\lim\limits_{z\to0}s_{l,k}^0(z,z\chi)=s_{l,k}^0(z,0)$ with $0\le k\le d_l$) and $\Delta(\A_{\chi_0+th_\alpha})$. The latter is generated by $s_{l,k}^\infty(0,\chi)=\Delta(s_{l,k}^\infty(\chi))$ (where $s_{l,k}^\infty(\chi)$ are the generators of $\A_\chi$ introduced in section~\ref{se:Achi}), with $0\le k\le d_l-1$ and by $\Delta(\fh)$).
    \item For $(t,\infty)\in U_\infty^\pm$ with $t\ne0$ the corresponding subalgebra contains $\A_{\chi_0+th_\alpha}^{(1)}$ and $\A_{\chi_0+th_\alpha}^{(2)}$ generated by the $z\to\infty$ limits of $s_{l,k}^z(z,\chi)$ and $s^0_{l,k}(z,\chi)$ respectively (for all $k$).
    \item For $(0,z)\in U_0$ with $z\ne0,\infty$ the corresponding subalgebra contains by $\A_{\chi_0}^0(z,0)$ (generated by $z^{d_l-k}s_{l,k}^0(z,\chi_0)$ with $0\le k\le d_l$, $s_{l,k}^\infty(z,\chi_0)$ with $0\le k\le d_l-1$) and $\Delta(h_\alpha)$. According to Corollary~\ref{co:Center} this subalgebra contains $\Delta(C_\alpha)$.
    \item For $(0,0)\in U_0$ we have $\A(1,0)$ (generated by $\lim\limits_{z\to0}z^{d_l-k}s_{l,k}^0(z,\chi)=\lim\limits_{z\to0}s_{l,k}^0(z,z\chi)=s_{l,k}^0(z,0)$ with $0\le k\le d_l$), $\Delta(\A_{\chi_0}^0)$ (generated by $s_{l,k}^\infty(0,\chi_0)=\Delta(s_{l,k}^\infty(\chi_0))$ with $0\le k\le d_l-1$) and $\Delta(h_\alpha)$.
    \item For the point $t=0, z = \infty, [t:z^{-1}]=[a:b]$ in the special fiber of the blow-up, the corresponding subalgebra contains $\A_{\chi_0}^{(1)}$ and $\A_{\chi_0}^{(2)}$ generated by the $z\to\infty$ limits of $s_{l,k}^z(z,\chi_0)$ and $s^0_{l,k}(z,\chi)$ respectively (with $d_l-d'_l\le k\le d'_l$). Clearly it also contains $\Delta(h_\alpha)$. The last generator is $\lim\limits_{t\to0}G_{h_\alpha}(a, bt^{-1}\chi_0+h_\alpha)=ah_\alpha^{(1)}+b\Delta(e_\alpha)\Delta(f_\alpha)$ which is  $ah_\alpha^{(1)}+b\Delta(e_\alpha)\Delta(f_\alpha)$ up to an expression of $\Delta(h_\alpha)$, so we are done.
\end{enumerate}

\end{proof}

We will also need the closure of the parameter space for the family of subalgebras $ \A_\chi(z_1,z_2,z_3) \subset U(\fg) \otimes U(\fg)  \otimes U(\fg)$ where $\chi \in \fh^{reg}$  is fixed. Note that the subalgebra $ \A_\chi(z_1,z_2,z_3) $ does not change under simultaneous additive shift of the $z_i$'s, so the parameter space for this family of subalgebras is $ \BC^3\smallsetminus \Delta / \BC$. The limit $\lim\limits_{\varepsilon\to0}\A_\chi(\varepsilon z_1,\varepsilon z_2,\varepsilon z_3)=\lim\limits_{\varepsilon\to0}\A_\eps\chi(z_1,z_2,z_3)$ is the subalgebra generated by $ \Delta(\A_\chi) $ and $ \A(z_1,z_2,z_3) $: indeed, it is generated by $\lim\limits_{\eps\to0}s_{l,k}^{z_i}(\uz,\eps\chi)=s_{l,k}^{z_i}(\uz,0)$ (i.e. generators of $\A(z_1,z_2,z_3)$) and by $\lim\limits_{\eps\to0}s_{l,k}^{\infty}(\eps\uz,\chi)=\Delta(s_{l,k}^{\infty}(\chi))$ (i.e. generators of $ \Delta(\A_\chi) $). So the closure of the parameter space contains the parameter space for the subalgebras $ \A(z_1,z_2,z_3) $, i.e. the Deligne-Mumford space $\CM_4=\BP^1$.

Let $Y$ be the closure of the quadratic cone $\{u_1u_2+u_2u_3+u_3u_1=0\}\subset\BC^3$ in the projective space $\BP^3\supset\BC^3$. Note that $Y$ contains the parameter space $ \BC^3 \smallsetminus \Delta / \BC$ as an open subset due to the embedding $u_1=(z_2-z_3)^{-1}, u_2=(z_3-z_1)^{-1}, u_3=(z_1-z_2)^{-1}$. The complement of this open subset consists of $4$ projective lines $ l_0, \dots, l_3 $, where $l_i=\{u_j=0\ |\ j\ne i\}$, $i=1,2,3$, and $l_0$ is the intersection of $Y$ with the hyperplane at infinity. Note that $[z_1-z_3:z_2-z_3]$ is a homogeneous coordinate on $l_0$. Let $X$ be the blow-up of $Y$ at the $3$ points $l_i\cap l_0$. Denote by $m_i$ the corresponding exceptional projective lines. Note that the open subset of $X$ determined by $u_1u_2u_3\ne0$ is the blow-up of $\BC^2$ at the origin. In the coordinates $(z_1,z_2,z_3)$ the lines $m_i$ are just $m_1=\{z_2=z_3\}$, $m_2=\{z_3=z_1\}$, $m_3=\{z_1=z_2\}$, and $l_0$ is the exceptional fiber of the blow-up.

\begin{prop} \label{pr:MultiClosure3}
Let $ \chi \in \fh^{reg} $.  The map from $ (\BC^3 \smallsetminus \Delta) / \BC $ to commutative subalgebras of $ U(\fg) \otimes U(\fg)  \otimes U(\fg)$ given by $ (z_1,z_2,z_3) \mapsto \A_\chi(z_1,z_2,z_3) $ extends to a family of subalgebras parametrized by $X$ where $u_1=(z_2-z_3)^{-1}, u_2=(z_3-z_1)^{-1}, u_3=(z_1-z_2)^{-1}$.  The extended map is defined on the boundary as follows:
\begin{enumerate}
\item $(u_1,u_2,u_3)=0$ goes to the subalgebra $ \A_\chi \otimes \A_\chi  \otimes \A_\chi$.
\item For $\{i,j,k\}=\{1,2,3\}$, any point $u_i\in l_i\backslash\{0,\infty\}$ goes to the subalgebra $ \Delta^{i, jk}(\A_\chi \otimes \A_\chi (u_i^{-1},0))$.
\item $\infty\in l_i$ goes to the subalgebra generated by $ \Delta^{i, jk}(\A_\chi \otimes \CA_\chi)$ and $ \CA(1,0)^{(jk)}$.
\item For $\{i,j,k\}=\{1,2,3\}$, any point $z\in m_i\backslash\{0,\infty\}$ goes to the subalgebra generated by $ \Delta^{i, jk}(\CA_\chi(z,0))$ and $ \CA(1,0)^{(jk)} $.
 \item $ [z_1-z_3 : z_2-z_3]\in l_0 $ goes to the subalgebra generated by $ \Delta(\CA_\chi)$ and $ \CA(z_1,z_2,z_3) $.
\end{enumerate}
\end{prop}

\begin{proof}
The variety $X$ is covered by the following 2 open subsets:
$$U_\infty:=\Spec \BC[u_1,u_2,u_3]/(u_1u_2+u_2u_3+u_3u_1)=Y\backslash l_0 \text{ and } U_0=\{u_1u_2u_3\ne0\}.$$
We have to prove that the family of subalgebras $ \A_\chi(z_1,z_2,z_3) $ extends to both $U_0$ and $U_{\infty}$. Note that $U_0$ is a line bundle over $ \CM_4$ and we can repeat the proof of Theorem~3.13 of \cite{Ryb13}. Namely, this variety is covered by $3$ affine open subsets $U_1:=\Spec\BC[z_1-z_2,\frac{z_2-z_3}{z_1-z_2},\frac{z_2-z_3}{z_1-z_3}]$ and $U_2, U_3$ defined analogously. Let us provide a set of algebraically independent generators of $\A_\chi(z_1,z_2,0)$ which is regular on $U_1$ (for $U_2, U_3$ the construction will be the same). Consider the sections $S_l(w;z_1,z_2,z_3;\chi)$. The coefficients of the principal parts of their expansions at $z_1,z_2$ and $\infty$ freely generate $ \A_\chi(z_1,z_2,z_3) $. Hence the coefficients of the principal part of
$$S_l(w;\frac{z_1}{z_1-z_2},\frac{z_2}{z_1-z_2},\frac{z_3}{z_1-z_2};(z_1-z_2)\chi) \text{ at } w=\frac{z_2}{z_1-z_2}$$
and those of
$$S_l(w;z_1,z_2,z_3;\chi) \text{ at $w=z_1$ and $w=\infty$ } $$
freely generate $ \A_\chi(z_1,z_2,z_3) $ as well. On the other hand, these coefficients are regular on $U_1$ and their values at the points on the exceptional fiber freely generate $ \Delta_{123}(\CA_\chi)\cdot \CA(z_1,z_2,z_3) $.

Now let us show that there is a set of algebraically independent generators of $\A_\chi(z_1,z_2,z_3)$ which is regular on $U_\infty$. We can take just the coefficients of the principal part of $S_l(w;z_1,z_2,z_3;\chi)$ at $w=z_1,z_2,z_3$, which are clearly regular on $U_\infty$. The same argument as in Proposition~\ref{pr:MultiClosure1} shows that at the boundary points we get the desired subalgebras.
\end{proof}

\begin{cor} \label{ContractiblePentagon}
We have a contractible pentagon in the real locus of the parameter space of algebras $\A_\chi(z_1,z_2,z_3)$ such that the subalgebras attached to the vertices are
\begin{gather*}
 \Delta(\CA_\chi)\cdot \CA(1(23)), \quad  \A_\chi^{(1)}\otimes\Delta^{23}(\CA_\chi)\cdot \CA(1,0)^{(23)}, \quad \A_\chi^{\otimes3}, \\
 \Delta^{12}(\CA_\chi)\cdot \CA(1,0)^{(12)} \otimes \A_\chi^{(3)}, \  \Delta(\CA_\chi)\cdot \CA((12)3).
 \end{gather*}
 The first one is connected with the last one by the path $ \Delta(\CA_\chi)\cdot \CA(1,z,0) $, $0\le z \le1$.
\end{cor}

\begin{rem} The $X$ arising in Proposition~\ref{pr:MultiClosure3} can be regarded as a moduli space of stable \emph{framed} rational curves.  A generic point of this moduli space consists of a smooth rational curve with a fixed non-zero tangent vector at one framed point and then $3$ more marked points.  This framed curve can be regarded as a degeneration of a $\BP^1$ with $1$ more marked point which goes to $\infty$.

We identify a point $(z_1, z_2, z_3) \in \BC^3 \smallsetminus \Delta / \C $ with  $\BP^1$, together with a marking of the points $ z_1,z_2,z_3 $, along with the framed point $\infty$ and a non-zero tangent vector at $\infty$:
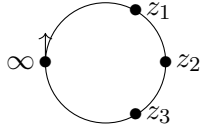
\begin{equation*}
\begin{tikzpicture}[scale=0.8]
\draw (0:0) circle (1);
\node at (60:1) {$\bullet$};
\node[right] at (60:1) {$z_1$};
\node at (0:1) {$\bullet$};
\node[right] at (0:1) {$z_2$};
\node at (-60:1) {$\bullet$};
\node[right] at (-60:1) {$z_3$};
\draw [->] (-1,0) -- (-1,0.5);
\node at (-1,0) {$\bullet$};
\node[left] at (-1,0) {$\infty$};
\end{tikzpicture}
\end{equation*}
The automorphism group of $ \BP^1 $ preserving the framing at $ \infty $ is the additive group $ \C $.  In this way, we get an isomorphism between $ \BC^3 \smallsetminus \Delta / \C $ and the moduli space of framed rational curves with three marked points.

The variety $ X $ can be regarded as a compactification of stable framed rational curves with three marked points.  The geometry of this space is determined by the following rule describing two possible degenerations:
\begin{enumerate}
\item when some of distinguished points (either marked or nodes) $z_{i_1},\ldots,z_{i_k}$ on some component of the curve collide at some point $z$ they make a new smooth rational component which intersects the old one normally at $z$, as for the usual Deligne-Mumford compactification;
\item when some of the distinguished points $z_{i_1},\ldots,z_{i_k}$ go to the framed point keeping the differences $z_{i_r}-z_{i_s}$ constant, then they make a new framed component which has the common tangent vector at the framed point with the old one.
\end{enumerate}
So the curves arising as the degenerations consist of the components of two types: first, (unframed) rational curves with at least $3$ distinguished points, and second, rational curves with a framing at $\infty$ with at least $1$ more distinguished point. All the framed points are glued together. In particular, the singular point of the variety $X$ corresponds to the degenerate curve with $3$ framed components having $1$ marked point on each of them:
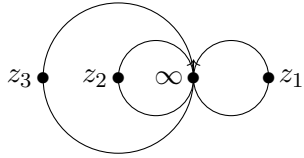
\begin{equation*}
\begin{tikzpicture}[scale=0.5]
\draw (0:0) circle (1);
\draw [->] (-1,0) -- (-1,0.5);
\node at (-1,0) {$\bullet$};
\node[left] at (-1,0) {$\infty$};
\node at (0:1) {$\bullet$};
\node[right] at (0:1) {$z_1$};

\draw (-2,0)  circle (1);
\node at (-3,0) {$\bullet$};
\node[left] at (-3,0) {$z_2$};

\draw (-3,0)  circle (2);
\node at (-5,0) {$\bullet$};
\node[left] at (-5,0) {$z_3$};
\end{tikzpicture}
\end{equation*}

The lines $l_i$ correspond to $2$-component curves with $2$ framed components with $1$ and $2$ marked points on them:
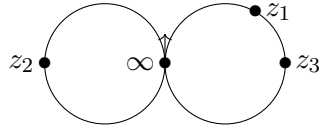
\begin{equation*}
\begin{tikzpicture}[scale=0.8]
\draw (0:0) circle (1);
\draw [->] (-1,0) -- (-1,0.5);
\node at (-1,0) {$\bullet$};
\node[left] at (-1,0) {$\infty$};
\node at (60:1) {$\bullet$};
\node[right] at (60:1) {$z_1$};
\node at (0:1) {$\bullet$};
\node[right] at (0:1) {$z_3$};

\draw (-2,0)  circle (1);
\node at (-3,0) {$\bullet$};
\node[left] at (-3,0) {$z_2$};
\end{tikzpicture}
\end{equation*}
the lines $m_i$ corresponds to $2$-component curves with $1$ framed component with $1$ marked point and $1$ ordinary component with $2$ marked points:
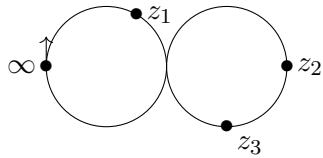
\begin{equation*}
\begin{tikzpicture}[scale=0.8]
\draw (0:0) circle (1);
\node at (60:1) {$\bullet$};
\node[right] at (60:1) {$z_1$};
\draw [->] (-1,0) -- (-1,0.5);
\node at (-1,0) {$\bullet$};
\node[left] at (-1,0) {$\infty$};

\draw (0:2) circle (1);
\node at (0:3) {$\bullet$};
\node[right] at (0:3) {$z_2$};
\node at (2,-1) {$\bullet$};
\node[below right] at (2,-1) {$z_3$};
\end{tikzpicture}
\end{equation*}
and the line $l_0$ corresponds to $2$-component curves with $1$ framed component without marked points and $1$ ordinary component with $3$ marked points:
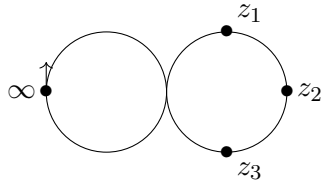
\begin{equation*}
\begin{tikzpicture}[scale=0.8]
\draw (0:0) circle (1);
\draw [->] (-1,0) -- (-1,0.5);
\node at (-1,0) {$\bullet$};
\node[left] at (-1,0) {$\infty$};

\draw (0:2) circle (1);
\node at (2,1) {$\bullet$};
\node[above right] at (2,1) {$z_1$};
\node at (0:3) {$\bullet$};
\node[right] at (0:3) {$z_2$};
\node at (2,-1) {$\bullet$};
\node[below right] at (2,-1) {$z_3$};
\end{tikzpicture}
\end{equation*}

Finally the $3$ distinguished points on $l_0$, namely, $[1:0],[0:1]$ and $[1:1]$ correspond to $3$-component curves with $1$ framed component and $2$ ordinary ones:
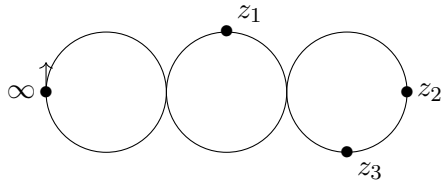
\begin{equation*}
\begin{tikzpicture}[scale=0.8]
\draw (0:0) circle (1);
\draw [->] (-1,0) -- (-1,0.5);
\node at (-1,0) {$\bullet$};
\node[left] at (-1,0) {$\infty$};

\draw (0:2) circle (1);
\node at (2,1) {$\bullet$};
\node[above right] at (2,1) {$z_1$};

\draw (0:4) circle (1);
\node at (0:5) {$\bullet$};
\node[right] at (0:5) {$z_2$};
\node at (4,-1) {$\bullet$};
\node[below right] at (4,-1) {$z_3$};
\end{tikzpicture}
\end{equation*}

The subalgebra corresponding to a curve is roughly the product of $\A_\chi(\uz)$'s corresponding to the framed components and $\A(\uz)$'s corresponding to the ordinary ones.

\end{rem}

We expect that if $\chi\in\fh^{reg}$ is fixed, then the closure of the parameter space for $\A_\chi(z_1,\ldots,z_n)$ is always isomorphic to the above compactification of the moduli space of framed rational curves. In the particular case $\fg=\fsl_2$, this was proved by Pakharev in his BSc. thesis set-theoretically (i.e. the closure was determined as a stratified topological space, not as algebraic variety).

\section{Cyclicity}\label{sect-cyclic}

According to \cite{FFR}, the algebra $\A_{\chi}(\uz)$ has simple spectrum on any finite-dimensional irreducible representation of $\fg^{\oplus n}$, whenever $\chi$ is \emph{real} regular semisimple and $\uz$ is a collection of pairwise distinct real numbers. In this section we generalize this to some limit subalgebras. In particular, we show that for arbitrary $\chi\in\CM_\Delta(\BR)$, the algebra $\A_{\chi}$ acts on $V(\lambda)$ with simple spectrum. We use the same argument as in \cite{Ryb16}.

\subsection{Cyclicity for degenerate and limit subalgebras}

Let $V(\l)$ be irreducible finite-dimensional $\fg$-module with highest weight $\l$, and $\pi_\l: U(\fg)\to \End(V(\l))$ be the corresponding algebra homomorphism. Let $\chi_0\in\fh$ be arbitrary Cartan element, $\uz=(z_1,\ldots,z_n)$ be a collection of pairwise distinct complex numbers. We consider the Gaudin subalgebra $\A^0_{\chi_0}(\uz)\subset U(\fg)^{\otimes n}$. Note that for singular $\chi_0$ it is smaller than $\A_{\chi}(\uz)$ for regular $\chi$. We set $\fg_1:=\fz_\fg(\chi_0)'$, the commutator subalgebra of the centralizer of $\chi_0$, and $\fn_1=\fg_1\cap\fn$. The subalgebra $\A^0_{\chi_0}(\uz)$ lies in the diagonal invariants of $\fg_1$ in $U(\fg)^{\otimes n}$ hence acts on the space of diagonal $\fn_1$-invariants (i.e. highest vectors with respect to $\fg_1$) in any tensor product $V(\l_1)\otimes\ldots\otimes V(\l_n)$.

\begin{thm}\label{th:CyclicGeneral} For any collection $\ul:=(\l_1,\ldots,\l_n)$ of dominant weights the subalgebra $\A^0_{\chi_0}(\uz)$ has a cyclic vector in the space $V(\ul)^{\fn_1}$ of highest vectors with respect to $\Delta(\fg_1)$.
\end{thm}

\begin{proof}
Let $e,f,h$ a principal $\fsl_2$-triple in $\fg_1$ such that $h\in\fh$ and $e\in\fn_1$. Then $\chi_0+f$ is a regular element of $\fg$. Consider the subalgebra $\A_{\chi_0+f}\subset U(\fg)$. The operator $\ad \Delta(h)$ defines a grading on $U(\fg)^{\otimes n}=\bigoplus\limits_{k\in\BZ} (U(\fg)^{\otimes n})^k$, where
$$(U(\fg)^{\otimes n})^k:=\{x\in U(\fg)^{\otimes n}\ : \ \Delta(h)x-x\Delta(h) = kx\}.$$
This grading induces an increasing filtration on $U(\fg)^{\otimes n}$ defined by $(U(\fg)^{\otimes n})^{(k)}:=\bigoplus\limits_{i\le k} (U(\fg)^{\otimes n})^k$, for $k\in\BZ$.

\begin{lem} We have $\A_{\chi_0+f}(\uz)\subset (U(\fg)^{\otimes n})^{(0)}$, hence we get a \emph{bounded} increasing filtration on $\A_{\chi_0+f}(\uz)$. The associated graded $\gr \A_{\chi_0+f}(\uz)$ is naturally a commutative subalgebra in $\gr (U(\fg)^{\otimes n}) = U(\fg)^{\otimes n}$ and is the same as the limit $\lim\limits_{t\to0}\A_{\chi_0+tf}(\uz)$.
The $0$th graded component of $\gr \A_{\chi_0+f}(\uz)$ is $\A^0_{\chi_0}(\uz)\subset U(\fg)^{\otimes n}$.
\end{lem}

\begin{proof} Consider the subalgebra $\A_{\chi_0+tf}(\uz)\subset U(\fg)^{\otimes n}\otimes\BC[t]$. This subalgebra is the image of the universal one $\A(\uz,\infty)\subset U(\fg)^{\otimes n}\otimes S(\fg)$ under the evaluation of the last tensor factor at $\chi_0+tf$. All elements of the subalgebra $\A(\uz,\infty)$ are invariant with respect to the diagonal $\fg$ action so in particular are annihilated by $\ad \Delta(h)$. Hence all elements of the subalgebra $\A_{\chi_0+tf}(\uz)\subset U(\fg)^{\otimes n}\otimes\BC[t]$ are annihilated by $2\frac{\partial}{\partial t}+\ad \Delta(h)$. This means that $\A_{\chi_0+tf}(\uz)$ lies in the $0$-th graded component of $U(\fg)^{\otimes n}\otimes\BC[t]$ with respect to the grading determined by $\ad \Delta(h)$ on $U(\fg)^{\otimes n}$ and by $\deg t=2$. Hence the evaluation of $\A_{\chi_0+tf}(\uz)$ at $t\ne0$ (e.g. $t=1$) is in the $0$-th component of $U(\fg)^{\otimes n}$ with respect to the filtration defined above. Moreover for any element of $\A_{\chi_0+tf}(\uz)$ its leading term as $t\to0$ (i.e. the coefficient at the smallest degree of $t$) is the projection to the subspace with the biggest eigenvalue of $\ad \Delta(h)$. Hence $\gr \A_{\chi_0+f}(\uz)=\lim\limits_{t\to0}\A_{\chi_0+tf}(\uz)$. The $0$-degree part of $\gr \A_{\chi_0+f}(\uz)$ is then just the evaluation of $\A_{\chi_0+tf}(\uz)$ at $t=0$ which is $\A^0_{\chi_0}(\uz)$.
\end{proof}

Consider the action of $\A_{\chi_0+f}(\uz)$ on $V(\ul)$. By \cite{FFR}, $V(\ul)$ is cyclic as a $\A_{\chi_0+f}(\uz)$-module. The operator $\Delta(h)$ defines a grading on $V(\ul)$ which agrees with the grading of $U(\fg)^{\otimes n}$ defined above and with the decomposition into $\Delta(\fg_1)$-isotypic components $V(\ul)=\bigoplus_{\nu\in X(\fg_1)}I_\nu$ (here $I_\nu=m_\nu V_\nu$ are isotypic components with respect to $\Delta(\fg_1$)). Clearly we have $V(\ul)^{\fn_1}=\bigoplus\limits_{\nu}I_\nu^{top}$ (here the superscript $top$ means the top degree part with respect to the $h$-grading). Now we want to deduce the cyclicity of $I_\nu^{top}$ with respect to $\A^0_{\chi_0}(\uz)$ from the cyclicity of $V(\ul)$ with respect to $\A_{\chi_0+f}(\uz)$. The main difficulty is that the decomposition $V(\ul)=\bigoplus_{\nu\in X(\fg_1)}I_\nu$ does not agree with the action of $\A_{\chi_0+f}(\uz)$. However, the following holds (and it is sufficient for our purposes):

\begin{lem}\label{lem-isotypic} There is a decomposition $V(\ul)=\bigoplus_{\nu\in X(\fg_1)} J_\nu$ such that it is preserved by $\A_{\chi_0+f}(\uz)$ and $\gr J_\nu=I_\nu$.
\end{lem}

\begin{proof} The isotypic components $I_\nu$ are the eigenspaces for the center $Z$ of $\Delta(U(\fg_1))$. According to Corollary~\ref{co:Center}, $Z\subset \A^0_{\chi_0}(\uz)\subset\gr\A_{\chi_0+f}(\uz)$. Hence for any $C\in Z$ there exists $\tilde{C}\in\A_{\chi_0+f}(\uz)$ such that $\tilde{C}=C+N$ where $\deg N <0$. Consider the image of $\tilde{C}$ in $\End(V(\ul))$. We have the Jordan decomposition $\pi_\ul(\tilde{C})=\pi_\ul(\tilde{C})_s+\pi_\ul(\tilde{C})_n$ where both $\pi_\ul(\tilde{C})_s$ and $\pi_\ul(\tilde{C})_n$ are polynomials of $\pi_\ul(\tilde{C})$, hence lie in the image of $\A_{\chi_0+f}(\uz)$. Since the image of $\pi_\ul(C)$ is semisimple and $\pi_\ul(\tilde{C})_n$ is a nilpotent operator expressed as a polynomial of $\pi_\ul(C+N)$, we have $\deg \pi_\ul(\tilde{C})_n<0$. Hence $\pi_\ul(\tilde{C})_s = \pi_\ul(C)$ modulo lower degree. Hence the projectors to the eigenspaces of $\pi_\ul(\tilde{C})_s$ are the the projectors to the eigenspaces of $C$ modulo lower degree. Thus the desired decomposition of $V_\ul$ is the decomposition into the eigenspaces with respect to $\pi_\ul(\tilde{C})_s$ for some generic $C\in Z$.
\end{proof}

By \cite{FFR}, $V(\ul)$ is cyclic as a $\A_{\chi_0+f}(\uz)$-module. Hence each $J_\nu$ is cyclic with respect to $\A_{\chi_0+f}(\uz)$. In particular, the quotient $J_\nu^{(top)}/J_\nu^{(top-1)}$ is cyclic with respect to $\A_{\chi_0+f}(\uz)/\A_{\chi_0+f}(\uz)^{(-1)}=\A^0_{\chi_0}(\uz)$. On the other hand, according to Lemma~\ref{lem-isotypic}, $J_\nu^{(top)}/J_\nu^{(top-1)}$ is isomorphic to $I_\nu^{top}=m_\nu V_\nu^{\fn_1}$ as $\A^0_{\chi_0}(\uz)$-module. Hence the space of highest vectors of each $\Delta(\fg_1)$-isotypic component in $V(\ul)$ is cyclic as $\A^0_{\chi_0}(\uz)$-module, and hence the whole space $V(\ul)^{\fn_1}$ is a cyclic $\A^0_{\chi_0}(\uz)$-module.
\end{proof}

As a corollary we get the Theorem~\ref{th:cyclic} (proved earlier in \cite{Ryb16}). We reproduce the proof here for completeness. 

\begin{cor}\label{co:BAC} 
For any $\uz\in\CM_{n+1}$ and any collection of dominant integral weights $\ul=(\l_1,\ldots,\l_n)$ and a dominant integral weight $\mu$, the space $V(\ul)^\mu$ is a cyclic module for the algebra $\A(\uz)$.
\end{cor}

\begin{proof}
According to the above Theorem for $\chi_0=0$, any Gaudin subalgebra $\A(\uz)$ corresponding to a point $\uz$ in the open stratum of $\CM_{n+1}$ acts cyclically on the space of singular vectors $V(\ul)^{\fn}$. Since the subalgebra $\A(\uz)$ contains the center of $\Delta(U(\fg))$, its image in $\End(V(\ul)^\fn)$ contains the projector to any $V(\ul)^\mu$, so $V(\ul)^\mu$ is cyclic with respect to $\A(\uz)$ as well. 

For $\uz$ in smaller stata we proceed by induction on $n$. Suppose that $\uz=\gamma_{k,m_1,\ldots,m_k}(\underline{w},\underline{u_1},\ldots,\underline{u_k})$. Then by Theorem~\ref{th:GaudinBoundary} the corresponding subalgebra $\A(\uz)$ is generated by $I_{i}(\A(\underline{u_i}))$ and $D_{m_1,\ldots,m_k}(\A(\underline{w}))$, where  
$I_{i}$ and $D_{m_1,\ldots,m_k}$ are given by 
$$ I_i := \Delta^{m_1 + \cdots + m_{i-1} + 1, \dots, m_1 + \cdots + m_{i-1} + m_i} :U(\fg)^{\otimes m_i}\hookrightarrow U(\fg)^{\otimes n};$$
$$D_{m_1,\ldots,m_n} := \Delta^{\{1, \dots, m_1\}, \{m_1 + 1, \dots, m_1 +  m_2\}, \dots, \{m_1 + \dots + m_{n-1} + 1, \dots, n \}} :U(\fg)^{\otimes k}\hookrightarrow U(\fg)^{\otimes n}.$$

Let $V(\ul)=\bigoplus  V(\ul)^{\underline{\nu}}\otimes V(\underline{\nu})$ be the decomposition of $V(\ul)$ into the sum of isotypic component with respect to $D_{m_1,\ldots,m_k}(\fg^{\oplus k})$ with $V(\underline{\nu})$ being the irreducible representation of $\fg^{\oplus k}$ with the highest weight $\underline{\nu}=(\nu_1,\ldots,\nu_k)$ and $ V(\ul)^{\underline{\nu}}:=\Hom_{\fg^{\oplus k}}(V(\underline{\nu}),V(\underline{\l}))$ being the multiplicity space. Restricting to singular vectors of the weight $\mu$ in $V(\ul)$ we have $V(\ul)^\mu=\bigoplus  V(\ul)^{\underline{\nu}}\otimes V(\underline{\nu})^\mu$
By induction hypothesis, the multiplicity spaces $ V(\ul)^{\underline{\nu}}=\bigotimes\limits_{i=1}^k V(\l_{m_1+\ldots+m_i+1},\ldots,\l_{m_1+\ldots+m_{i+1}})^{\nu_i}$ are cyclic $\bigotimes\limits_{i=1}^k I_{i}(\A(\underline{u_i}))$-modules. On the other hand, each $V(\underline{\nu})^\mu$ is a cyclic $D_{m_1,\ldots,m_k}(\A(\underline{w}))$-module. Hence the entire module $V(\underline{\l})^\mu$ is cyclic with respect to $\A(\uz)$.
\end{proof}

\begin{cor}\label{cyclic} The $\fg$-module $V(\l)$ is cyclic as
an $\A_\chi$-module for any $\chi\in \CM_{\Delta}$.
\end{cor}

\begin{proof}

Let $\chi_0$ be the projection of $\chi$ to $\fh$. Denote by $\Delta_1\subset\Delta$ the root system of the commutant $\fg_1:=\fz_\fg(\chi_0)'$. Then the subalgebra $\A_\chi\subset U(\fg)$ is generated by the subalgebras $\A_{\chi_0}^0\subset U(\fg)$ and $\A_{\chi_1}\subset U(\fg_1)\subset U(\fg)$ for some $\chi_1\in\CM_{\Delta_1}$. Hence we can proceed by induction on $\rk\fg$ using Theorem~\ref{th:CyclicGeneral}.
\end{proof}

\begin{cor}\label{co:SimpleSpec} For $\chi\in \CM_{\Delta}(\BR)$ and
any dominant  $\l$, the quantum shift of argument subalgebra
$\A_\chi \subset U(\fg)$ is diagonalizable and has simple spectrum on
the $\fg$-module $V(\l)$.
\end{cor}

\begin{proof}
Following the same strategy as in \cite{FFR}, we see that the algebra $\A_\chi$ with real $\chi$ (with respect to the compact real form of $\fg$) acts by Hermitian operators on $V(\l)$ (with respect to the Hermitian form preserved by the compact form of $\fg$), and hence is diagonalizable. By Corollary~\ref{cyclic}, for
any $\chi\in \CM_{\Delta}$ it has a cyclic vector. The two properties may only be realized if $\A_\chi$ has simple spectrum.
\end{proof}

From these results, we obtain the following two further corollaries.

\begin{cor}\label{co:SimpleSpecGeneral} Suppose that $\chi_0\in\fh_\BR$ is a real Cartan element and $z_1,\ldots,z_n\in\BR$ pairwise distinct real numbers. For any collection $\ul:=(\l_1,\ldots,\l_n)$ of dominant  weights the subalgebra $\A^0_{\chi_0}(z_1,\ldots,z_n)$ has simple spectrum in the space $V(\ul)^{\fn_1}$ of highest vectors with respect to $\Delta(\fg_1)$.
\end{cor}

\begin{cor}\label{co:SimpleSpecLimits} All subalgebras arising in Propositions~\ref{pr:MultiClosure1},~\ref{pr:MultiClosure2}~and~\ref{pr:MultiClosure3} have a cyclic vector in any tensor product of the appropriate number of irreducible representations of $\fg$. For the real values of the parameters all these subalgebras have simple spectrum.
\end{cor}

\subsection{Cover of the de Concini-Procesi space}
We define

\begin{equation}
 \CE(\l) = \{(\chi, L) : \chi \in \CM_\Delta(\BR) : L \text{ is an eigenline for } \A_\chi\}
 \end{equation}
We have obvious maps $ \CE(\l) \rightarrow \BP(V(\l)) $ and $ \CE(\l) \rightarrow \CM_\Delta(\BR) $.

Corollary~\ref{co:SimpleSpec} immediately implies the following.

\begin{cor}\label{cor-covering} The map $ \CE(\l) \rightarrow \CM_\Delta(\BR) $ is an unramified covering.
\end{cor}

The fibres of this cover will be denoted $ \CE_\chi(\l) $ for any $ \chi \in \CM_\Delta(\BR)$.

\section{Crystal structure on the fibre} \label{se:CrystalFibre}
Fix $ \l \in \Lambda_+ $ and $ \chi \in \fh_+^\emptyset$ with $ w_0(\chi) = -\chi$ (this condition is not important, but is used to simplify the discussion later).  Our goal is to define a crystal structure on the fibre $ \CE_\chi(\l)$ and prove that it is isomorphic to $ B(\l) $.

If $ \fg = \fg_1 \oplus \fg_2 $, then we can write $ V(\l) = V(\l_1) \otimes V(\l_2) $ where $ V(\l_i)$ is an irrep of $ \fg_i $.  Moreover $ \A_\chi = \A_{\chi_1} \otimes \A_{\chi_2} $ where $ \A_{\chi_i} $ is a shift of argument algebra in $ U(\fg_i) $.  So we get a decomposition $ \CE_\chi(\l) = \CE_{\chi_1}(\l_1) \times \CE_{\chi_2}(\l_2) $.  Thus, it suffices to give a $ \fg_i$-crystal structure on $ \CE_{\chi_i}(\l_i)$, for $ i = 1,2 $.

In light of this observation, we will assume that $ \fg $ is simple for the remainder of this section.

\subsection{The definition}

To define the crystal structure, we must define a weight map $ wt : \CE_\chi(\l) \rightarrow \Lambda $ and crystal operators $ e_i, f_i : \CE_\chi(\l) \rightarrow \CE_\chi(\l) \sqcup \{0\} $ for each $ i \in I $.

First, as noted above, every eigenline for $ \CA_\chi $ lives inside a weight space of $ V(\l) $.  Thus, for $ L \in \CE_\chi(\l)$,  we define $ wt(L) = \mu $ if $ L $ lives inside the $ \mu $ weight space.

Now fix $ i \in I $ and fix a point $ \chi^i \in \CM_\Delta(\BR)_+^{\{i\}} $.  As usual, we will write $ \chi^i = (\chi^i_0, \chi^i_1) $ where $ \chi^i_0 \in \fh $ and $ \chi^i_1 \in \CM_{\Delta_1}(\BR)$.  (In this case, $\chi^i$ determines $ \chi^i_1 $, since $\CM_{\Delta_1}(\BR)$ is a point.)

Let $ \fg^i = \fz_\fg(\chi^i_0)' $, so that $ \fg^i $ is the $ \sl_2 $ root subalgebra corresponding to $ \alpha_i $.

By \corref{co:LinesProduct}, we have
\begin{equation} \label{eq:ithDecomp}
\CE_{\chi^i}(\l) = \CE_{\A_{\chi^i}}(V(\l)) = \sqcu_{\nu \in \Lambda_+(\fg^i)} \CE_{\A_{\chi^i_0}}(V(\l)^\nu) \times \CE_{\A_{\chi^i_1}}(V(\nu))
\end{equation}
Since $ \fg^i $ is isomorphic to $ \sl_2$, we see that $\CE_{\A_{\chi^i_1}}(V(\nu)) $ is just the set of weight spaces of the representation $ V(\nu)$, all of which are 1-dimensional.  We write $ \CE(\nu) $ for the set of the these weight spaces.

Note that if $ L \in \CE(\nu) $ is a weight line in an irrep of $ \sl_2 $, so then $ e_{\alpha_i}(L) $ is either another weight line or it is 0.  Moreover if $ L_1 \otimes L_2 \in \CE_{\chi^i}(\l) $ (given with respect to \eqref{eq:ithDecomp}), then $ e_{\alpha_i}(L_1 \otimes L_2) = L_1 \otimes e_{\alpha_i}(L_2) $.

Now, as in \secref{se:RealLocus}, there exists a unique homotopy class of path $ p_{\chi, \chi^i} $ connecting $ \chi $ and $ \chi^i $ inside $ \CM_\Delta(\R)_+ $.  Thus, we get a parallel transport map $ p_{\chi, \chi^i} : \CE_\chi(\l) \rightarrow \CE_{\chi^i}(\l) $ and we define
$$
e_i : \CE_\chi(\l) \rightarrow \CE_\chi(\l) \sqcup \{0\} \quad \text{ by } e_i = p_{\chi, \chi^i}^{-1} \circ e_{\alpha_i} \circ p_{\chi, \chi^i}
$$

In a similar way, we define $ f_i $.

\begin{prop} \label{pr:seminormal}
This defines a semi-normal $\fg$-crystal structure on $ \CE_\chi(\l) $.
\end{prop}

\begin{proof}
All the conditions in the definition of a semi-normal crystal only involve one $ i $ at a time.  Since the definition of $ e_i, f_i $ come from transporting the natural crystal structure on the set $\CE(\nu) $, these conditions are immediate.
\end{proof}

Our goal is now to prove the following result.

\begin{thm} \label{th:CEBlambda}
There is an isomorphism of crystals $ \CE_\chi(\l) \cong B(\l) $.
\end{thm}

\subsection{Tensor products}
In order to establish the above isomorphism of crystals, we will study the tensor product of these crystals.

By \propref{pr:MultiClosure1}, we have a family of subalgebras of $ U(\fg) \otimes U(\fg) $ parametrized by $ \BP^1 $ whose general fibre (for $ z \in \C^\times$) is $ \CA_\chi(z,0) $, whose fibre at $ 0 $ is generated by $ \Delta(\CA_\chi), \CA(1,0) $ and whose fibre at $ \infty $ is $ \CA_\chi \otimes \CA_\chi $.   We choose a path $ p_{\infty, 0} $ connecting $ \infty, 0 $ inside $ \R \BP^1 $, staying inside the positive real numbers $\R \BP^1_+$.

  Now consider the tensor product $ V(\l_1) \otimes V(\l_2) $ for $ \l_1, \l_2 \in \Lambda_+ $.  Since every algebra $ \CA_{\chi}(z,0) $ acts semisimply with simple spectrum on $ V(\l_1) \otimes V(\l_2) $ by \corref{co:SimpleSpecLimits}, we get a cover of $ \R \BP^1 $ given by the eigenlines of these algebras acting on $ V(\l_1) \otimes V(\l_2) $.   The eigenlines for the algebra at $ \infty $ are of the form $ L_1 \otimes L_2 $ where $ L_i \subset V(\l_i) $ is an eigenline for $ \A_\chi$.  On the other hand, the eigenlines for the algebra at $ 0 $ are compatible with the decomposition
  $$
  V(\l_1) \otimes V(\l_2) = \bigoplus_\mu  (V(\l_1) \otimes V(\l_2))^\mu \otimes V(\mu)
  $$
  and eigenlines are all of the form $ L_1 \otimes L_2 $ where $ L_1 $ is an eigenline for $ \CA(1,0) $ acting on $  (V(\l_1) \otimes V(\l_2))^\mu $  and $ L_2 $ is an eigenline for $ \CA_\chi $ acting on $ V(\mu) $.

  Thus, we get a parallel transport map bijection
$$
p_{\infty, 0} : \CE_{\chi}(\l_1) \times \CE_\chi(\l_2) \rightarrow \sqcu_\mu \CE(\l_1, \l_2)^\mu \times \CE_\chi(\mu)
$$

We give $ \CE_{\chi}(\l_1) \times \CE_\chi(\l_2)$ a crystal structure using the tensor product of crystals and we give $  \sqcup_\mu \CE(\l_1, \l_2)^\mu \times \CE_\chi(\mu) $ a crystal structure by just using the crystal structure on each $\CE_\chi(\mu) $.

\begin{thm} \label{th:TensorCE}
The map $p_{\infty, 0} $ is an isomorphism of crystals.
\end{thm}

\begin{proof}
It is immediate that the map $p_{\infty, 0} $ commutes with the weight maps.  So must check that it commutes with the crystal operators.

We begin by proving the theorem in the case that $ \fg = \sl_2 $.
The subalgebra $\A_h(z,0)$ is generated by the central elements $C^{(1)}, C^{(2)}$, the diagonal Cartan element $\Delta(h)=h^{(1)}+h^{(2)}$ and the single element which actually depends on the parameter $z$, namely $H_1(z)=\frac{\Delta(C)}{z}+h^{(1)}$. On the tensor product $V(\l_1)\otimes V(\l_2)$, the central elements act by scalars and the decomposition into the direct sum of eigenspaces for $\Delta(h)$ is just the weight decomposition.

Since $\A_h(z,0)$ acts with simple spectrum on $V(\l_1)\otimes V(\l_2)$ for all real $z$, the eigenvalues of $H_1(z)$ on each weight space (with respect to $\Delta(h)$) are pairwise distinct real numbers. This means that the parallel transport preserves the weight decomposition and does not change the \emph{increasing order of the eigenvalues} of $H_1(z)$ on each weight space. In particular, the parallel transport $p_{\infty, 0} $ on each weight space moves the eigenvectors of $h^{(1)}$ to the eigenvectors of $\Delta(C)$ in the same order with respect to increasing of the eigenvalues. The isomorphism of crystals is the only map $\CE_{\chi}(\l_1) \times \CE_\chi(\l_2) \rightarrow \sqcu_\mu \CE(\l_1, \l_2)^\mu \times \CE_\chi(\mu)$ which satisfies this property.  This can be seen from Figure \ref{CrystalTensor}; if we cut along this diagram along a diagonal line, the order of the points $(b_1,b_2) $ according to $ wt(b_1) $ is the same as the order according to the length of the string they lie in.

Now, we consider the general case.  Fix $ i \in I $.  We will show that the map $ p_{\infty, 0} $ is compatible with the actions of $ e_i $ on both sides.

As above, let $ \chi^i_0 \in \fh_+ $ be a general point on the hyperplane $ \alpha_i = 0 $.  Choose some $ T > 0 $ such that $ \chi^i_0 + t h_i \in \fh_+^{reg} $ for all $ t \in (0,T] $.  We can identify $ \chi^i_0 + T h_i = \chi $, since $ \fh_+^\emptyset $ is contractible.

We consider the family of subalgebras from \propref{pr:MultiClosure2} restricted to a locus $X(\BR)_+ \subset X $, which we define to be the preimage of $ [0,T] \times \R \PP^1_+ $ under the blow-down map, except that we only include positive part of the exceptional fibre.  The locus $X(\BR)_+$ can be identified with a pentagon with vertices $(T, \infty), (T, 0), (0,0), (0,\infty)_0, (0, \infty)_\infty $ where the last two points both sit inside the exceptional fibre.  From \propref{pr:MultiClosure2}, we see that the algebras at these points are respectively
\begin{gather*}
\CA_\chi \otimes \CA_\chi, \ \CA(1,0) \Delta(\CA_\chi), \ \CA(1,0)\Delta(\CA_{\chi^i}),\\
(\CA^0_{\chi^i_0} \otimes \CA^0_{\chi^i_0})\Delta(\C[h_i, C_i]),  \ \CA_{\chi^i} \otimes \CA_{\chi^i}
\end{gather*}
where $ C_i \in U(\fg^i) $ is the Casimir element and where $ \CA_{\chi^i} = \CA_{\chi^i_0}^0 \otimes \C[h_i] $.

Over every point in the locus $ X(\BR)_+ $, the correponding subalgebra of $ U(\fg) \otimes U(\fg) $ acts semisimply and with simple spectrum on $ V(\l_1) \otimes V(\l_2) $ by \corref{co:SimpleSpecLimits} and so we get a parallel transport for each edge of this pentagon.  Moreover because the parallel transport extends over the interior of the pentagon, we get a commutative pentagon of parallel transport as follows.

\begin{equation}
\begin{tikzcd}[column sep = tiny]
\CE_{\CA_\chi \otimes \CA_\chi}(V(\l_1) \otimes V(\l_2)) \ar[r] \ar[d]  & \CE_{\CA(1,0)\Delta(\CA_\chi)}(V(\l_1) \otimes V(\l_2)) \ar[d]\\
 \CE_{\CA_{\chi^i} \otimes \CA_{\chi^i}} (V(\l_1) \otimes V(\l_2) \ar[dr]  &  \CE_{\CA(1,0)\Delta(\CA_{\chi^i})}(V(\l_1) \otimes V(\l_2)) \ar[d]
 \\
 & \CE_{(\CA^0_{\chi^i_0} \otimes \CA^0_{\chi^i_0})\Delta(\C[h_i,C_i])}(V(\l_1) \otimes V(\l_2))
 \end{tikzcd}
 \end{equation}

Analyzing these eigenlines of these algebras, we see that this gives a commutative diagram
\begin{equation}
\begin{tikzcd}[cramped, column sep = -11em]
\CE_\chi(\l_1) \times \CE_\chi(\l_2) \ar[r] \ar[dd] &  \sqcu_\mu \CE(\l_1, \l_2)^\mu \times \CE_\chi(\mu) \ar[d]\\
 & \bigsqcup\limits_{\mu, \gamma} \CE(\l_1,\l_2)^\mu \times \CE_{\CA^0_{\chi_0^i}}(V(\mu)^\gamma) \times \CE(\gamma) \ar[dd]
 \\
 \bigsqcup\limits_{\nu_1, \nu_2} \CE_{\CA^0_{\chi^i_0}}(V(\l_1)^{\nu_1}) \times \CE_{\CA^0_{\chi^i_0}}(V(\l_2)^{\nu_2}) \times \CE(\nu_1) \times \CE(\nu_2) \ar[dr]  & \\
  & \bigsqcup\limits_{\nu_1, \nu_2, \gamma} \CE_{\CA^0_{\chi^i_0}}(V(\l_1)^{\nu_1})\times \CE_{\CA^0_{\chi^i_0}}(V(\l_2)^{\nu_2}) \times \CE(\nu_1,\nu_2)^\gamma \times \CE(\gamma)
\end{tikzcd}
\end{equation}
In this diagram, $ \mu $ varies over dominant weights of $ \fg $ and $ \nu_1, \nu_2, \gamma $ vary over dominant weights for $ \fg^i $.  Also $ \CE(\gamma) $ denotes the set of weight spaces (all 1-dimensional) in the representation $V(\gamma) $.

Now, the upper two vertical arrows are the parallel transport used to define the crystal operator $ e_i $ on $ \CE_\chi(\l) $.  Also, the diagonal arrow commutes with the crystal operator $ e_i $ because the theorem holds for $\sl_2 $.  Thus, in order to prove that the top horizontal arrow commutes with the crystal operator $ e_i$, it suffices to check this for the lower vertical arrow.  But the lower vertical arrow can be written as
\begin{equation*}
\begin{tikzcd}
 \bigsqcup_\gamma \bigsqcup\limits_{\mu} \CE(\l_1,\l_2)^\mu \times \CE_{\CA^0_{\chi^i_0}}(V(\mu)^\gamma) \times \CE(\gamma) \ar[d] \\ \bigsqcup\limits_{\gamma} \bigsqcup\limits_{\nu_1, \nu_2} \CE_{\CA^0_{\chi^i_0}}(V(\l_1)^{\nu_1}) \times \CE_{\CA^0_{\chi^i_0}}(V(\l_2)^{\nu_2}) \times \CE(\nu_1,\nu_2)^\gamma \times \CE(\gamma)
\end{tikzcd}
\end{equation*}
and is compatible with this decomposition into the disjoint union over $ \gamma$, since at every point $ (0,z) \in X $, the corresponding algebra contains the Casimir $C_i $ of $ \fg^i $ embedded by the coproduct.
This shows that this lower vertical arrow is compatible with the crystal operator $ e_i $.
\end{proof}

\subsection{Restriction}

Now, we prove that this crystal structure is compatible with restriction to Levi subalgebras.

Fix $ J \subset I $.   Pick $ \chi^J \in \CM_\Delta(\BR)_+^J $ and as usual write $ \chi^J = (\chi^J_0, \chi^J_1) $ where $ \chi_0^J \in \fh_+^J $ and $ \chi_1^J \in \CM_{\Delta^J}(\BR)_+$, where $ \Delta_J = \Delta \cap \spn(\alpha_j : j \in J) $ denotes those roots spanned by simple roots from $ J $.  Let $ \fg_J $ be the Lie algebra with root system $ \Delta_J $; it is the derived subalgebra of the centralizer of $ \chi^J_0 $.

As in \corref{co:LinesProduct}, we have a decomposition
$$
\CE_{\chi^J}(\l) = \sqcu_\nu \CE_{\A^0_{\chi^J_0}(V(\l)^\nu)} \times \CE_{\chi^J_1}(\nu)
$$
where $ \nu $ ranges over the dominant weights of $ \fg_J $ and $\CE_{\chi^J_1}(\nu)$ denotes (as usual) the eigenlines for $ \CA_{\chi^J_1} \subset U(\fg_J) $ acting on $ V(\nu)$ .

We endow $ \CE_{\chi^J}(\l)$ with a $\fg_J$-crystal structure by using the  $ \fg_J$-crystal structure on $ \CE_{\chi^J_1}(\nu)$.  On the other hand, $ \CE_\chi(\l) $ carries a $ \fg$-crystal structure which we can restrict to a $ \fg_J $-crystal structure.

\begin{thm} \label{th:RestrictCrystal}
The parallel transport map $ p_{\chi, \chi^J} : \CE_\chi(\l) \rightarrow \CE_{\chi^J}(\l) $ is a $\fg^J$-crystal isomorphism.
\end{thm}

\begin{proof}
The compatibility with the weight maps is immediate, since all algebras contain the Cartan subalgebra of $ \fg^J $.

So it remains to check the compatibility with all crystal operators $ e_i $ for $ i \in J $. Fix $ i \in J $.

We will consider the following points: $ \chi $ (generic point of $ \fh(\BR)_+ $), $\chi^J$ (generic point of $ \CM_\chi(\BR)^J_+ $), $\chi^i $ (generic point of $ \CM_\chi(\BR)^{i}_+ $), and we define $ \chi^{J,i} := (\chi^J_0, \chi_J^i) $ where $ \chi_J^i $ denotes a generic point in $ \CM_{\Delta_J}(\BR)^{i}_+ $.

Since all these points live in the contractible set $ \CM_\Delta(\BR)_+ $, we obtain a commutative square of parallel transports
\begin{equation*}
\begin{tikzcd}
\CE_\chi(\l) \ar[r,"{p_{\chi, \chi^J}}"] \ar[d,"{p_{\chi, \chi^i}}"'] & \CE_{\chi^J}(\l) = \sqcu_\nu \CE_{\A^0_{\chi^J_0}(V(\l)^\nu)} \times \CE_{\chi^J_1}(\nu) \ar[d,"{(id, p_{\chi^J, \chi^{J,i}})}"] \\
\begin{matrix} \CE_{\chi^i}(\l)  \\ \rotatebox{90}{=} \\ \sqcu_{\gamma \in \Lambda_+(\fg^i)} \CE_{\A^0_{\chi^i_0}}(V(\l)^\gamma) \times \CE(\gamma) \end{matrix} \ar[r,"{p_{\chi^i, \chi^{J,i}}}"] & \begin{matrix}\CE_{\chi^{J,i}}(\l)  \\ \rotatebox{90}{=} \\ \sqcu_{\mu, \nu} \CE_{\A^0_{\chi^J_0}(V(\l)^\nu)} \times  \CE_{\chi_J^i}(V(\nu)^\gamma) \times \CE(\gamma) \end{matrix}
\end{tikzcd}
\end{equation*}
Now the crystal structure on the top row are each defined by following the vertical arrows and applying $ e_{\alpha_i} \in \fg $.  The bottom horizontal arrow is commutes with $ e_{\alpha_i} $, since along the path $ p_{\chi^i, \chi^{J, i}} $ all algebras contain the Casimir element $ C_i $.   Thus, we conclude that $ p_{\chi, \chi^J} $ is commutes with $ e_i$.
\end{proof}

\subsection{Normality of $ \CE(\l)$}

Now, we are finally in a position to complete the proof of the main theorem of this section.

\begin{proof}[Proof of \thmref{th:CEBlambda}]
First, note that it suffices to show that $ \CE_\chi(\l) $ is normal, since once we show that it is normal, then it must contain a copy of $ B(\l) $ (since it has a highest weight element of weight $ \l $) and since its size is the same as $ B(\l) $ this gives us an isomorphism.

By \thmref{th:rank2}, it suffices to check that $ \CE_\chi(\l)_J $ is normal for any subset $J \subset I $ of size 2.  By \thmref{th:RestrictCrystal}, $ \CE_\chi(\l)_J $ is a disjoint union of copies of the $\fg_J$-crystals $ \CE_{\chi_J}(\nu) $.

Thus, it suffices to prove $ \CE_\chi(\l) $ is normal under the assumption that $ \fg $ is of rank 2.  Choose a dominant weight $ \omega $ for $ \fg $ such that $ V(\omega) $ is multiplicity-free and is a tensor generator for the category of $ \fg$-representations (for example, we can take $ \omega $ to be the first fundamental weight with the usual Bourbaki labelling).  As all the weight spaces of $ V(\omega) $ are one-dimensional, the set $ \CE_\chi(\omega) $ is in bijection with the set of weights of the representation $ V(\omega) $. Thus we obtain a bijection between $ \CE_\chi(\omega) $ and $ B(\omega)$. Hence \propref{pr:MultFree} and \propref{pr:seminormal} imply that $ \CE_\chi(\omega) $ is isomorphic to $ B(\omega) $.

For any dominant weight $ \l $, there exists some $ N$ for which there is an embedding $ V(\l) \subset V(\omega)^{\otimes N} $.  From \thmref{th:TensorCE}, we conclude that there is a crystal embedding $ \CE_\chi(\l) \subset \CE_\chi(\omega)^{\otimes N}$.  Since $ \CE_\chi(\omega) $ is normal, we see that $ \CE_\chi(\omega)^{\otimes N} $ is normal and thus $ \CE_\chi(\l) $ is normal as desired.
\end{proof}

\section{Monodromy and Sch\"utzenberger involutions}
\subsection{Cactus group action on cover}

Let $ \l $ be a dominant  weight.  The Weyl group also acts on the set of one-dimensional subspaces of weight spaces in $ V(\l)$ because $ W = N_G(T)/ T $.  Every eigenline for a shift of argument algebra $ \A_\chi $ is a subspace of a weight space because $ \fh \subset \A_\chi $.

\begin{prop} \label{cor-Wactioncover}
If $ L $ is an eigenline for $ \A_\chi $ and $ w \in W $, then $ wL $ is an eigenline for $ \A_{w \chi} $.  Thus there is an action of $ W $ on the cover $ \CE(\l) $ compatible with its action on the base $ \CM_\Delta(\BR)$.
\end{prop}

Thus, from \lemref{le:CoverEquiv} and \thmref{cor-DJS} , we get an action of the cactus group $C_\Delta $ on the set $ \CE_\chi(\l) $ where $ \chi \in \fh_+^\emptyset $ is a fixed base point.   We also have an action of $ C_\Delta $ on the crystal $ B(\l) $ by \thmref{th:InternalAction}.

\begin{thm} \label{th:Etingof2}
The isomorphism of crystals $ \CE_\chi(\l) \cong B(\l) $ from \thmref{th:CEBlambda} is compatible with the action of $ C_\Delta $.
\end{thm}

Since $ \CE_\chi(\l) $ is isomorphic to $ B(\l) $, for each connected subset $ J \subseteq I$, we have the Sch\"utzenberger involution $ \xi_J : \CE_\chi(\l) \rightarrow  \CE_\chi(\l) $.  On the other hand, we have the elements $ s_J \in C_\Delta $ acting by monodromy.  So to prove the theorem, it suffices to show that for each $ J$, $ \xi_J $ and $ s_J $ are equal as permutations of $  \CE_\chi(\l)$.

\subsection{Full Sch\"utzenberger involution}
We begin with the case of the full Sch\"utzenberger involution, $\xi$, i.e. $ J = I $.
\begin{prop} \label{pr:FullSI}
We have $ \xi = s_I $ as permutations of $ \CE_\chi(\l) $.
\end{prop}

\begin{proof}
Recall that $ \xi $ is characterized by the properties that it acts by $ w_0 $ on weights and that $ \xi (e_i(b)) = f_{\theta(i)} (\xi(b))$ (where $ \theta: I \rightarrow I $ is the canonical Dynkin diagram automorphism).

Recall also that $ s_I : \CE_\chi(\l) \to \CE_\chi(\l) $ is just defined by the action of $ w_0 : V(\l) \rightarrow V(\l) $ which takes eigenlines of $ \A_\chi $ to eigenlines of $ \A_{w_0(\chi)} = \A_{-\chi} = \A_\chi$ (here we use that $ w_0(\chi) = -\chi$).
So the desired property on weights is clear.  Thus it suffices to show that for any $ L \in \CE_\chi(\l) $ and any $ i \in I$,  we have $ w_0 (e_i(L)) = f_{\theta(i)} ( w_0(L)) $.

Recall the point $ \chi^i \in \CM_\Delta(\BR)^{i} $.  We have $ w_0(\chi^i) = \chi^{\theta(i)} $ and moreover $ w_0(p_{\chi, \chi^i}) = p_{\chi, \chi^{\theta(i)}} $.  Thus applying \lemref{le:SquareMonodromies}, we see that
$$
p_{\chi, \chi^{\theta(i)}}(w_0(L)) = w_0(p_{\chi, \chi^i} (L))
$$
So, applying the definition of the crystal structure, we see that
\begin{align*}
 f_{\theta(i)} ( w_0(L)) &= p_{\chi, \chi^{\theta(i)}}^{-1}( f_{\alpha_i}(p_{\chi, \chi^{\theta(i)}}(w_0(L)))) \\
 &= p_{\chi, \chi^{\theta(i)}}^{-1}( f_{\alpha_i} (w_0( p_{\chi, \chi^i}(L)))) \\
 &= p_{\chi, \chi^{\theta(i)}}^{-1} (w_0 ( e_{\alpha_i}(p_{\chi, \chi^i}(L)))) \\
 &= w_0( p_{\chi, \chi^i}( e_{\alpha_i}(p_{\chi, \chi^i}(L)))) = w_0 (e_i(L))
 \end{align*}
 as desired.
 \end{proof}

\subsection{Partial Sch\"utzenberger involutions}
Now, let $ J \subsetneq I $.

\begin{prop} \label{pr:PartialSI}
 $ \xi_J $ and $ s_J $ agree as permutations of $ \CE_\chi(\l) $.
 \end{prop}

 \begin{proof}
Recall that $ \xi_J $ is determined by the restriction of $ \CE_\chi(\l) $ as a $ \fg_J $ crystal.  And recall that $ s_J : \CE_\chi(\l) \rightarrow \CE_\chi(\l) $ is defined as the composition $ w_0^J \circ p_{\chi, w_0^J(\chi)} \circ p_{\chi, \chi^J} $, which we can rewrite as $ p_{\chi^J, \chi} \circ w_0^J \circ p_{\chi, \chi^J} $.  Here, as before $ \chi^J $ is a $w_0^J$-invariant generic point in $ \CM_\Delta(\BR)_+^J $.

Now by \thmref{th:RestrictCrystal}, $ p_{\chi, \chi^J} $ is a $ \fg_J $-crystal isomorphism.  Thus it suffices to check that $ \xi_J = w_0^J $ as permutations of $ \CE_{\chi^J}(\l) $.  But we have that
$$
\CE_{\chi^J}(\l) = \sqcu_\nu \CE_{\A^0_{\chi^J_0}(V(\l)^\nu)} \times \CE_{\chi^J_1}(\nu)
$$
and both $ \xi_J, w_0^J $ act only on the $\CE_{\chi^J_1}(\nu) $ factors in this decomposition.  So the result immediately follows from \propref{pr:FullSI} applied to the Lie algebra $ \fg_J $.
\end{proof}

\section{The equivalence}
Our goal in this section is to prove \thmref{th:Main} on the equivalence between $ \cC(\CE) $ and $ \gcrys $ as coboundary categories.

\subsection{Construction of the tensor functor}
Recall that the underlying category of $ \cC(\CE) $ is just $ \Lset $.  We define a functor $ \Phi : \cC(\CE) \rightarrow \gcrys $ by
$$
\Phi((A_\lambda)_{\lambda \in \Lambda_+}) = \sqcu_{\lambda \in \Lambda_+} A_\lambda \times \CE_\chi(\l)
$$
where we regard $ \CE_\chi(\l) $ as a crystal using the results of \secref{se:CrystalFibre}. Since $ \CE_\chi(\l) $ is isomorphic to $ B(\l) $, we immediately see that $ \Phi $ is an equivalence of categories.

Recall the simple objects $ S(\lambda) $ of $ \cC(\CE) $.  From now on, we will mostly work with these objects.  It is straightforward (but notationally messy) to extend everything we write to general objects.

We upgrade $ \Phi $ to a monoidal functor by defining
$$
\phi : \Phi(S(\l_1)) \otimes \Phi(S(\l_2)) \rightarrow \Phi( S(\l_1) \otimes S(\l_2) )
$$
as follows.  The left hand side is given by the tensor product of crystals
$$
\CE_\chi(\l_1) \otimes \CE_\chi(\l_2)
$$
while following the construction from \secref{se:CoverToCategory}, we see that
$$
(S(\l_1) \otimes S(\l_2))_\mu = \CE(\l_1, \l_2)^{\mu}
$$
and thus we see that
$$
\Phi(S(\l_1) \otimes S(\l_2)) = \sqcu_\mu \CE(\l_1, \l_2)^{\mu} \times \CE_\chi(\mu)
$$
Hence $ \phi $ is required to be an isomorphism of crystals
$$
\CE_\chi(\l_1) \otimes \CE_\chi(\l_2)  \rightarrow  \sqcu_\mu \CE(\l_1, \l_2)^{\mu} \times \CE_\chi(\mu)
$$
Fortunately, we have already constructed such an isomorphism of crystals in \thmref{th:TensorCE}, and we define $ \phi $ to be the isomorphism $ p_{\infty,0} $ appearing in that statement.

Now, we need to show compatibility with the associators and the commutors.

\subsection{Compatibility with associators}
We must show the commutativity of the following diagram
\begin{equation}
\begin{tikzcd}
(\Phi(S(\l_1)) \otimes \Phi(S(\l_2))) \otimes \Phi(S(\l_3)) \ar[r,"\alpha"] \ar[d,"{\phi \circ (\phi \otimes id)}"] & \Phi(S(\l_1)) \otimes (\Phi(S(\l_2)) \otimes \Phi(S(\l_3))) \ar[d,"{\phi \circ (id \otimes \phi)}"]\\
\Phi((S(\l_1) \otimes S(\l_2)) \otimes S(\l_3)) \ar[r,"{\Phi(\alpha)}"] & \Phi(S(\l_1) \otimes (S(\l_2) \otimes S(\l_3)))
\end{tikzcd}
\end{equation}

Applying the definitions of the objects, we reach the diagram
\begin{equation} \label{eq:AssociatorsWant}
\begin{tikzcd}
\CE_\chi(\l_1) \times \CE_\chi(\l_2) \times \CE_\chi(\lambda_3) \ar[r,"{id}"] \ar[d,"{p}"] & \CE_\chi(\l_1) \times \CE_\chi(\l_2) \times \CE_\chi(\lambda_3)  \ar[d,"{p}"] \\
\sqcu_\mu \CE(\ul)^\mu_{(12)3} \times \CE_\chi(\mu) \ar[r,"{p_{(12)3), 1(23)}}"] & \sqcu_\mu \CE(\ul)^\mu_{1(23)} \times \CE_\chi(\mu)
\end{tikzcd}
\end{equation}
where we use Lemma \ref{BracketsPoints} to identify $ (S(\ul)_{(12)3)})_\mu = \CE(\ul)^\mu_{(12)3} $.  Here the top horizontal arrow is the identity map, the bottom horizontal arrow is given by parallel transport in $ \CM_4(\BR)_+$, and the two vertical arrows are more complicated.  The left vertical arrow (the right one is similar) is given by the composition
\begin{align*}
\CE_\chi(\l_1) &\times \CE_\chi(\l_2) \times \CE_\chi(\lambda_3) \xrightarrow{p_{\infty, 0} \times {id}} \sqcup_\nu \CE(\l_1, \l_2)^\nu \times \CE_\chi(\nu) \times \CE_\chi(\l_3) \\
&\xrightarrow{p_{\infty,0}} \CE(\l_1, \l_2)^\mu  \times \CE(\nu,\l_3)^\mu \times \CE_\chi(\mu) \xrightarrow{\Gamma} \sqcu_\mu \CE(\ul)^\mu_{(12)3} \times \CE_\chi(\mu)
\end{align*}
where in the last step we use the operadic structure, as in section \ref{se:CoverToCategory}.

Thus, if we remove the top horizontal edge and expand the two vertical edges, then (\ref{eq:AssociatorsWant}) turns into the following pentagon
\begin{equation} \label{eq:PentagonWant}
\begin{tikzcd}[column sep = -2em]
&  \CE_\chi(\l_1) \times \CE_\chi(\l_2) \times \CE_\chi(\l_3) \ar[dl] \ar[dr]  & \\
\sqcup_\nu \CE(\l_1, \l_2)^\nu \times \CE_\chi(\nu) \times \CE_\chi(\l_3) \ar[d] & & \CE_\chi(\l_1) \times \sqcup_\nu \CE(\l_2, \l_3)^\nu \times \CE_\chi(\nu) \ar[d] \\
 \sqcu_\mu \CE(\ul)^\mu_{(12)3} \times \CE_\chi(\mu) \ar[rr] & & \sqcu_\mu \CE(\ul)^\mu_{1(23)} \times \CE_\chi(\mu)
\end{tikzcd}
\end{equation}
These are the eigenlines for the following commutative algebras acting on $ V(\ul) $:
\begin{gather*}
 \Delta_{123}(\CA_\chi)\cdot \CA(1(23)), \quad  \A_\chi^{(1)}\otimes\Delta^{23}(\CA_\chi)\cdot \CA(1,0)^{(23)}, \quad \A_\chi^{\otimes3}, \\
 \Delta^{12}(\CA_\chi)\cdot \CA(1,0)^{(12)} \otimes \A_\chi^{(3)}\otimes, \  \Delta(\CA_\chi)\cdot \CA((12)3)
 \end{gather*}
 (read counter-clockwise from the bottom right).

 In Corollary \ref{ContractiblePentagon}, we showed that these algebras are the vertices of a contractible pentagon within the closure of the space of inhomogeneous Gaudin algebras, all of which act semisimply with simple spectrum on $ V(\ul) $ by \corref{co:SimpleSpecLimits}.  Thus we conclude that (\ref{eq:PentagonWant}) is a commutative pentagon and thus (\ref{eq:AssociatorsWant}) is a commutative square and so we deduce that our functor is compatible with the associators, as desired.

\subsection{Compatibility with commutors}
We must show commutativity of the diagram
\begin{equation}
\begin{tikzcd}
\Phi(S(\l_1)) \otimes \Phi(S(\l_2)) \ar[r, "\sigma"] \ar[d,"\phi"] & \Phi(S(\l_2)) \otimes \Phi(S(\l_1)) \ar[d,"\phi"] \\
\Phi(S(\l_1) \otimes S(\l_2)) \ar[r,"\Phi(\sigma)"] & \Phi(S(\l_2) \otimes S(\l_1))
\end{tikzcd}
\end{equation}
Applying the definitions of the objects, we reach the diagram
\begin{equation} \label{eq:CommutorsWant}
\begin{tikzcd}
\CE_\chi(\l_1) \otimes \CE_\chi(\l_2) \ar[r,"\sigma"] \ar[d,"{p_{\infty, 0}}"] & \CE_\chi(\l_2) \otimes \CE_\chi(\l_1) \ar[d,"{p_{\infty,0}}"]
 \\
\sqcu_\mu \CE(\l_1, \l_2)^\mu \times \CE_\chi(\mu) \ar[r,"s"] & \sqcu_\mu \CE(\l_2, \l_1)^\mu \times \CE_\chi(\mu)
\end{tikzcd}
\end{equation}
Here the top horizonal arrow is the commutor in the category of crystals, the two vertical arrows are given by monodromy, and the bottom horizontal arrow is given by the flip map $ s : V(\l_1) \otimes V(\l_2) \rightarrow V(\l_2) \otimes V(\l_1) $ which takes eigenlines for the algebra $\CA(1,0) \Delta(\CA_\chi) $ to eigenlines for $\CA(1,0) \Delta(\CA_\chi) $.

To prove the commutativity of \eqref{eq:CommutorsWant}, we start by applying \lemref{le:SquareMonodromies} to the map $ s : V(\l_1) \otimes V(\l_2) \rightarrow V(\l_2) \otimes V(\l_1) $ and we obtain a commutative square:
\begin{equation*}
\begin{tikzcd}
\CE_\chi(\l_1) \otimes \CE_\chi(\l_2) \ar[r,"s"] \ar[d,"{p_{\infty, 0}}"] & \CE_{\chi}(\l_2) \otimes \CE_\chi(\l_1) \ar[d,"{s(p_{\infty,0})}"] \\
\sqcu_\mu \CE(\l_1, \l_2)^\mu \times \CE_\chi(\mu) \ar[r,"s"
] & \sqcu_\mu \CE(\l_2, \l_1)^\mu \times \CE_\chi(\mu)
\end{tikzcd}
\end{equation*}

Then we apply \lemref{le:SquareMonodromies} to the map $ w_0 : V(\l_2) \otimes V(\l_1) \rightarrow V(\l_2) \otimes V(\l_1) $ and we obtain the commutative square:
\begin{equation*}
\begin{tikzcd}
\CE_\chi(\l_2) \otimes \CE_\chi(\l_1) \ar[r,"w_0"] \ar[d,"s(p_{\infty, 0})"] & \CE_{w_0(\chi)}(\l_2) \otimes \CE_{w_0(\chi)}(\l_1) \ar[d,"w_0(s(p_{\infty,0}))"] \\
\sqcu_\mu \CE(\l_2, \l_1)^\mu \times \CE_\chi(\mu) \ar[r,"w_0"] & \sqcu_\mu \CE(\l_2, \l_1)^\mu \times \CE_{w_0(\chi)}(\mu)
\end{tikzcd}
\end{equation*}

Combining these two squares, we obtain the commutative rectangle
\begin{equation*}
\begin{tikzcd}
\CE_\chi(\l_1) \otimes \CE_\chi(\l_2) \ar[r,"s"] \ar[d,"{p_{\infty, 0}}"] & \CE_\chi(\l_2) \otimes \CE_\chi(\l_1) \ar[r,"{w_0}"] & \CE_{w_0(\chi)}(\l_2) \otimes \CE_{w_0(\chi)}(\l_1) \ar[d,"{w_0(s(p_{\infty,0}))}"] \\
\sqcu_\mu \CE(\l_1, \l_2)^\mu \times \CE_\chi(\mu) \ar[r,"s"] & \sqcu_\mu \CE(\l_2, \l_1)^\mu \times \CE_\chi(\mu)\ar[r,"{w_0}"] & \sqcu_\mu \CE(\l_2, \l_1)^\mu \times \CE_{w_0(\chi)}(\mu)
\end{tikzcd}
\end{equation*}

We will now analyze the arrows in this diagram.  The top right and bottom right horizontal arrows (given acting by $ w_0 $ on the eigenlines) are the same as the Sch\"utzenberger involutions $ \xi \otimes \xi $ and $ \xi $, respectively, on the crystals by \propref{pr:FullSI}.

We now consider the path $ w_0(s(p_{\infty,0})) $.  The flip map $ s $ conjugates $ \CA_\chi(z_1, z_2) $ to $ \CA_\chi(z_2, z_1)  $ (this is a special case of \lemref{le:ConjugateA}). Recall that the path $ p_{\infty,0} $ is given by  $\CA_\chi(z,0) $ where  $ z $ varies from $\infty $ to $0 $ through the positive reals in $ \CA_{\chi}(z, 0) $.  Thus $ s(p_{\infty, 0}) $ is given by  $\CA_\chi(0,z) $ where  $ z $ varies from $\infty $ to $0 $ through the positive reals in $ \CA_{\chi}(0, z) $.  By \lemref{le:InvarianceA}, we know that $ \CA_\chi(0,z) = \CA_\chi(-z,0) $.

By \lemref{le:ConjugateA}, we know that $ w_0 $ conjugates $ \CA_\chi(-z,0) $ to $ \CA_{w_0(\chi)}(-z, 0) $.  We assume that $ w_0(\chi) = -\chi $ and so we get
$$
\CA_{w_0(\chi)}(-z, 0) = \CA_{-\chi}(-z,0) = \CA_\chi(z,0)
$$
where in the second step we use \lemref{le:InvarianceA}.  Thus we conclude that $ w_0(s(p_{\infty,0})) = p_{\infty, 0} $.

Thus, the previous commutative rectangle is
\begin{equation*}
\begin{tikzcd}
\CE_\chi(\l_1) \otimes \CE_\chi(\l_2) \ar[r,"s"] \ar[d,"{p_{\infty, 0}}"] & \CE_\chi(\l_2) \otimes \CE_\chi(\l_1) \ar[r,"{\xi \otimes \xi}"] & \CE_\chi(\l_2) \otimes \CE_\chi(\l_1) \ar[d,"{p_{\infty,0}}"] \\
\sqcu_\mu \CE(\l_1, \l_2)^\mu \times \CE_\chi(\mu) \ar[r,"s"] & \sqcu_\mu \CE(\l_2, \l_1)^\mu \times \CE_\chi(\mu)\ar[r,"{\xi}"] & \sqcu_\mu \CE(\l_2, \l_1)^\mu \times \CE_\chi(\mu)
\end{tikzcd}
\end{equation*}
Now, since $ p_{\infty, 0} $ is a crystal isomorphism and $ \xi $ is natural for crystal isomorphisms (and is an involution), we can exchange their positions on the right side of the diagram.  After recalling that the crystal commutor is given by $ \sigma = \xi \circ (\xi \otimes \xi) \circ s $, we see that this rectangle reduces to the square \eqref{eq:CommutorsWant}.

This completes the proof of \thmref{th:Main}.

\begin{appendix} 
	
	\section{Notation}\label{append}
	
	Here is a list of the notation used in the paper.
	
	\subsection{Lie algebras, roots, weights}
	\begin{itemize}
		\item $ \fg $, a semisimple Lie algebra of rank $ r $.
		\item $ \fh $, the Cartan subalgebra
		\item $\Delta $, the set of roots of $ \fg $, $ \Delta_+ $ the set of positive roots
		\item $ \{\alpha_i \}_{i \in I} $, the set of simple roots, so $ |I| = r $
		\item for $ \alpha \in \Delta_+ $, an $ \mathfrak{sl}_2 $ triple $ e_\alpha, h_\alpha, f_\alpha $
		\item for $J \subseteq I$, the Lie subalgebra generated by all $ e_{\alpha_j}, f_{\alpha_j} $ for $ j \in J$; its root system is $ \Delta_J $, those roots which are linear combinations of the $ \alpha_j $
		\item $ \fh_+ \subset \fh(\R) $, the closed dominant Weyl chamber
		\item $\fh_+^J \subset \fh_+ $, the face of the closed dominant Weyl chamber corresponding to $ J \subseteq I $, so
		$$ \fh_+^J = \{ \chi \in \fh(\R) : \alpha_j(\chi) = 0 \text{ for $ j \in J $ and } \alpha_i(\chi) > 0 \text{ for $ i \in I \smallsetminus J $} \} $$
		\item $ \Lambda_+ $, the set of dominant weights of $ \fg $
		\item for $ \l \in \Lambda_+$,  the irreducible representation  $V(\l) $ of $ \fg $ with highest weight $ \l $
		\item $\ul = (\l_1, \dots, \l_n) $, a sequence of dominant weights and $ V(\ul) = V(\l_1) \otimes \cdots \otimes V(\l_n) $ the corresponding tensor product
		\item for $ V$ a representation of $ \fg $ and $ \mu $ a dominant weight, $ V^\mu = \Hom_\fg(V(\mu), V) $ is the multiplicity space, in particular $ V(\ul)^\mu = \Hom_\fg(V(\mu), V(\ul)) $
	\end{itemize}
	
	\subsection{Weyl and cactus groups and moduli spaces}
	\begin{itemize}
		\item $W $ the Weyl group of $ \fg $ with generators $ s_i $, $i \in I $
		\item $w_0$ the longest element of $ W $
		\item $ \theta : I \rightarrow I $ the bijection corresponding to $ w_0 $.
		\item $ W_J $ the standard parabolic subgroup of $ W $, generated by $ s_i $, $ i \in J $, for $ J \subseteq I $
		\item $ w_0^J $ the longest element of the subgroup $ W_J $ (so $ w_0 $ is the longest element of the group $ W $)
		\item $ \theta_J : J \rightarrow J $ the bijection corresponding to $ w_0^J $
		\item $ C_\Delta $ the cactus group associated to $ W $, generated by $ s_J $, such that $ J \subseteq I $ connected
		\item $ \CM_\Delta $ the De Concini-Procesi wonderful compactification associated to the root system $ \Delta $.
		\item $\CM_\Delta(\BR) $, the real locus of $ \CM_\Delta $ and its non-negative part $ \CM_\Delta(\BR)_+ $
		\item $ \CM_\Delta(\BR)_+^J $, the face of $ \CM_\Delta(\BR)_+ $ corresponding to a subset $ J \subsetneq I $
		\item $ \CM_{n+1} = \overline{M}_{0,n+1}$ the Deligne-Mumford space of stable genus 0 curves with $ n+1 $ marked points, which coincides with $ \CM_{\Delta(\sl_n)} $
		\item  $ C_n $, the cactus group on $ n$-strands, which coincides with $ C_{\Delta(\sl_n)} $
		\item the generators $ s_{pq} $ of $ C_n $, for $ 1 \le p < q \le n $, so under the identification $ C_n = C_{\Delta(\sl_n)} $ we have $ s_{pq} = s_{\{p, p+1, \dots, q-1\}}$
	\end{itemize}
	
	\subsection{Crystals}
	\begin{itemize}
		\item for a crystal $ B $ of type $ \fg $, the Kashiwara operators $ e_i, f_i : B \rightarrow B \sqcup \{0 \}$
		\item for a normal crystal $ B$, the Sch\"utzenberger involution $ \xi : B \rightarrow B $
		\item for $ J \subseteq I $, and $ B$ a crystal of type $ \fg $, the restriction $ B_J $ of $B$ as a $ \fg_J $ crystal
		\item for a normal crystal $ B $ and $J \subseteq I $,  the partial Sch\"utzenberger involutions $ \xi_J : B \rightarrow B $
		\item $B(\l) $ the crystal of the representation $ V(\l) $
		\item $\ul = (\l_1, \dots, \l_n) $ a sequence of dominant weights and $ B(\ul) = B(\l_1) \otimes \cdots \otimes B(\l_n) $ the corresponding tensor product
		\item for $ \mu \in \Lambda_+ $, $ B(\ul)^\mu = \Hom(B(\mu), B(\ul)) $
	\end{itemize}
	
	\subsection{Commutative subalgebras}
	\begin{itemize}
		\item for $ \chi \in \CM_\Delta $, the shift of argument subalgebra $ \CA_\chi \subset U (\fg) $
		\item for $ \chi_0 \in \fh \smallsetminus \fg^{reg} $, the degenerate shift of argument subalgebra $ \CA^0_{\chi_0} \subset U (\fg) $
		\item for $ \uz \in \CM_{n+1} $, the Gaudin algebra $ \CA(\uz) \subset (U (\fg)^{\otimes n})^{\fg} $
		\item for $ \chi \in \fh^{reg}$ and $ z_1, \dots, z_n \in \C^n \smallsetminus \Delta $, the inhomogeneous Gaudin algebra $ \A_\chi(z_1, \dots, z_n) \subset (U (\fg))^{\otimes n} $
		\item for a commutative algebra $ A$ acting semisimply and cyclically on a vector space $ V$, $\cE_A(V) $ the eigenlines for $ A $ acting on $ V $
		\item $ \cE_\chi(\l) = \cE_{\CA_\chi}(V(\l)) $ for any point $ \chi \in \CM_\Delta(\BR) $
	\end{itemize}
	
	\subsection{Direct sums and diagonals}
	\begin{itemize}
	\item for $x\in\fg$, the image $x^{(i)}$ of $x$ in the $i$-th summand of $ \fg^{\oplus n}$.
	\item  $\Delta$ is the diagonal embedding $\Delta:\fg\to\fg^{\oplus n}$, so  $\Delta(x):=\sum\limits_{i=1}^nx^{(i)}$
	\item for a collection of $ k $ disjoint subsets of $ \{1, \dots, n \}$, a partial diagonal embedding $ \Delta^{A_1, \dots, A_k} $ defined by
	$$
	\Delta^{A_1, \dots, A_k}(x_1, \dots, x_k) = \sum_{j=1}^k \sum_{i \in A_j} x_j^{(i)}
	$$
	\end{itemize}
	
	\subsection{Monodromy}
	\begin{itemize}
		\item for a covering space $ \cX \rightarrow M $, $\cX_z$ the fibre over a point $ z \in M$
		\item for a homotopy class of path $ p $ in $M $ from $ y $ to $ z $, the monodromy map $ p : \cX_y : \rightarrow \cX_z $
		\item for $ y, z \in \CM_\Delta(\BR)_+ $, $p_{y,z}$ denotes the unique homotopy class of path staying within $ \CM_\Delta(\BR)_+ $
	\end{itemize}

\end{appendix}

\footnotesize{
{\bf I.H.}:
Northeastern University, Department of Mathematics;\\
360 Huntington Ave, Boston, MA, USA 02115;\\
{\tt i.halacheva@northeastern.edu}}

\footnotesize{
{\bf J.K.}:
University of Toronto, Department of Mathematics;\\
Room 6290, 40 St. George Street, Toronto, ON, Canada M5S 2E4;\\
{\tt jkamnitz@gmail.com}}

\footnotesize{
{\bf L.R.}: National Research University
Higher School of Economics, Russian Federation,\\
Department of Mathematics, 20 Myasnitskaya st, Moscow 101000;\\
Institute for Information Transmission Problems of RAS;\\
{\tt leo.rybnikov@gmail.com}}

\footnotesize{
{\bf A.W.}: 
Department of Mathematics, University of British Columbia; \\
1984 Mathematics Road, Vancouver, BC, Canada V6T 1Z2\\
{\tt alex.weekes@gmail.com}}

\end{document}